\documentclass[a4paper, 11pt]{amsart}

\usepackage[dvips]{graphicx}
\usepackage{amsmath,amsthm,amssymb,pinlabel}
\usepackage[all]{xy}
\usepackage[dvipsnames]{xcolor} \SelectTips{cm}{}

\allowdisplaybreaks

\usepackage{hyperref}
\hypersetup{colorlinks=true,citecolor=brown,linkcolor=brown,urlcolor=black,filecolor=black}

\theoremstyle{definition}
\newtheorem{dfn}{Definition}[section]
\newtheorem{rem}[dfn]{Remark}
\newtheorem{example}[dfn]{Example}

\theoremstyle{plain}
\newtheorem{thm}[dfn]{Theorem}
\newtheorem{prop}[dfn]{Proposition}

\newtheorem{prob}[dfn]{Problem}

\makeatletter
    
    \@addtoreset{equation}{section}
  \makeatother

\newcommand{\hpi}{\widehat{\pi}}

\newcommand{\Q}{\mathbb{Q}}
\newcommand{\Z}{\mathbb{Z}}

\newcommand{\calI}{\mathcal{I}}
\newcommand{\calM}{\mathcal{M}}
\newcommand{\calC}{\mathcal{C}}

\newcommand{\calT}{\mathcal{T}}
\newcommand{\calV}{\mathcal{V}}

\newcommand{\frakL}{\mathfrak{L}}

\newcommand{\frakM}{\mathfrak{M}}

\newcommand{\Aut}{\operatorname{Aut}}

\newcommand{\Hom}{\operatorname{Hom}}
\newcommand{\Der}{\operatorname{Der}}

\newcommand{\calG}{\mathcal{G}}

\newcommand{\filt}[1]{F^{ #1}}

\newcommand{\GL}{\mathbb{Q}\, \vert \pi \vert } 
\newcommand{\bGL}{\mathbb{Q}\, \vert\!\vert \pi \vert\!\vert}

\newcommand{\skein}{\mathcal{S}}
\newcommand{\cskein}{\widehat{\mathcal{S}} }
\newcommand{\hskein}{\mathcal{S}^{-1}}

\newcommand{\tskein}{\mathcal{A}}
\newcommand{\ctskein}{\widehat{\mathcal{A}}}
\newcommand{\tzeroskein}{\mathcal{A}_0 }

\newcommand{\torelli}{\mathcal{I}}

\newcommand{\zettaiti}[1]{\lvert #1 \rvert}
\newcommand{\zettaisq}[1]{\lvert\!\lvert #1  \rvert\!\rvert}

\newcommand{\gyaku}[1]{ #1^{-1}}

\newcommand{\arcsinh}{\mathrm{arcsinh}}
\newcommand{\arccosh}{\mathrm{arccosh}}

\begin{document}

\title{Generalized Dehn twists in low-dimensional~topology}

\author{Yusuke Kuno}
\address{Department of Mathematics, Tsuda University,
2-1-1 Tsuda-machi, Kodaira-shi Tokyo 187-8577, Japan}
\email{kunotti@tsuda.ac.jp}

\author{Gw\'ena\"el Massuyeau}
\address{ {IMB}, Universit\'e  Bourgogne Franche-Comt\'e \& CNRS, 21000 Dijon, France }
\email{gwenael.massuyeau@u-bourgogne.fr}

\author{Shunsuke Tsuji}
\address{Research Institute for Mathematical Sciences,
Kyoto University, 
Sakyo-ku, Kyoto, 606-8502, Japan }
\email{tsujish@kurims.kyoto-u.ac.jp}

\date{September 16th, 2019}

\subjclass[2010]{}
\thanks{This paper is dedicated to {Vladimir  G. Turaev}, 
on the occasion of the conference ``\emph{Geometry, Topology of manifolds, and Physics}''
(Strasbourg, June 2018) in his honor.}

\begin{abstract}
The generalized Dehn twist along a closed curve in an oriented surface is an algebraic construction 
which involves intersections of loops in the surface.
It is defined as an automorphism of the Malcev completion of the fundamental group of the surface.
As the name suggests, for the case where the curve has no self-intersection, 
it is induced from the usual Dehn twist along the curve.
In this expository article, after explaining their definition, 
we review several results about generalized Dehn twists such as their realizability as diffeomorphisms of the surface, their diagrammatic description in terms of decorated trees and the Hopf-algebraic framework underlying their construction.
Going to the dimension three, we also overview the relation between generalized Dehn twists 
and  $3$-dimensional homology cobordisms, and we survey the variants of generalized Dehn twists  for skein algebras of the surface.
\end{abstract}

\maketitle

{ \setcounter{tocdepth}{1} \tableofcontents}

\section{Introduction}

Let $C$ be a simple closed curve in the interior of an oriented surface.
The (right-handed) \emph{Dehn twist along $C$}, denoted by $t_C$, is an orientation-preserving diffeomorphism of the surface obtained by cutting the surface along $C$, rotating, and gluing back:
\begin{equation} \label{eq:tC}
{\unitlength 0.1in%
\begin{picture}(29.1400,7.4200)(3.8600,-11.9700)%
%
\special{pn 13}%
\special{ar 900 500 600 200 0.5404195 2.6011732}%
%
\special{pn 13}%
\special{ar 900 1300 600 200 3.6820122 5.7427658}%
%
\special{pn 13}%
\special{ar 2699 500 600 200 0.5404195 2.6011732}%
%
\special{pn 13}%
\special{ar 2699 1300 600 200 3.6820122 5.7427658}%
%
\special{pn 13}%
\special{ar 900 900 50 200 4.7123890 1.5707963}%
%
\special{pn 13}%
\special{ar 2700 900 50 200 4.7123890 1.5707963}%
%
\special{pn 13}%
\special{pn 13}%
\special{pa 2700 1100}%
\special{pa 2691 1096}%
\special{fp}%
\special{pa 2676 1075}%
\special{pa 2672 1066}%
\special{fp}%
\special{pa 2664 1038}%
\special{pa 2661 1026}%
\special{fp}%
\special{pa 2656 996}%
\special{pa 2655 984}%
\special{fp}%
\special{pa 2652 953}%
\special{pa 2651 940}%
\special{fp}%
\special{pa 2650 908}%
\special{pa 2650 895}%
\special{fp}%
\special{pa 2651 862}%
\special{pa 2652 850}%
\special{fp}%
\special{pa 2654 818}%
\special{pa 2656 807}%
\special{fp}%
\special{pa 2661 776}%
\special{pa 2663 763}%
\special{fp}%
\special{pa 2671 736}%
\special{pa 2676 726}%
\special{fp}%
\special{pa 2690 704}%
\special{pa 2700 700}%
\special{fp}%
%
\special{pn 13}%
\special{pn 13}%
\special{pa 900 1100}%
\special{pa 889 1095}%
\special{fp}%
\special{pa 874 1070}%
\special{pa 869 1057}%
\special{fp}%
\special{pa 861 1025}%
\special{pa 859 1012}%
\special{fp}%
\special{pa 854 979}%
\special{pa 853 965}%
\special{fp}%
\special{pa 851 932}%
\special{pa 850 918}%
\special{fp}%
\special{pa 850 884}%
\special{pa 851 871}%
\special{fp}%
\special{pa 853 837}%
\special{pa 854 823}%
\special{fp}%
\special{pa 858 790}%
\special{pa 861 777}%
\special{fp}%
\special{pa 868 745}%
\special{pa 872 733}%
\special{fp}%
\special{pa 889 705}%
\special{pa 900 700}%
\special{fp}%
%
\special{pn 8}%
\special{pa 400 900}%
\special{pa 1400 900}%
\special{fp}%
%
\special{pn 8}%
\special{pa 2200 900}%
\special{pa 2760 900}%
\special{fp}%
%
\special{pn 8}%
\special{ar 2760 1020 40 120 4.7123890 6.2831853}%
%
\special{pn 8}%
\special{ar 2840 1020 40 86 1.5707963 3.1415927}%
%
\special{pn 8}%
\special{pn 8}%
\special{pa 2880 1020}%
\special{pa 2880 1030}%
\special{fp}%
\special{pa 2873 1070}%
\special{pa 2870 1078}%
\special{fp}%
\special{pa 2848 1104}%
\special{pa 2840 1106}%
\special{fp}%
%
\special{pn 8}%
\special{pn 8}%
\special{pa 2880 1020}%
\special{pa 2880 1012}%
\special{fp}%
\special{pa 2881 973}%
\special{pa 2881 965}%
\special{fp}%
\special{pa 2883 927}%
\special{pa 2884 919}%
\special{fp}%
\special{pa 2887 881}%
\special{pa 2888 873}%
\special{fp}%
\special{pa 2893 835}%
\special{pa 2894 828}%
\special{fp}%
\special{pa 2901 792}%
\special{pa 2903 785}%
\special{fp}%
\special{pa 2912 751}%
\special{pa 2914 744}%
\special{fp}%
\special{pa 2927 712}%
\special{pa 2930 706}%
\special{fp}%
\special{pa 2952 684}%
\special{pa 2960 682}%
\special{fp}%
%
\special{pn 8}%
\special{ar 3040 780 40 120 1.5707963 3.1415927}%
%
\special{pn 8}%
\special{ar 2960 780 40 100 4.7123890 6.2831853}%
%
\special{pn 8}%
\special{pa 3040 900}%
\special{pa 3200 900}%
\special{fp}%
\put(8.0000,-6.0000){\makebox(0,0)[lb]{$C$}}%
\put(15.0000,-10.0000){\makebox(0,0)[lb]{$\ell$}}%
\put(33.0000,-10.0000){\makebox(0,0)[lb]{$t_C(\ell)$}}%
\end{picture}}%
\end{equation}
Dehn twists appear in a number of basic constructions in low-dimensional topology. 
This mainly stems from the so-called ``Dehn--Lickorish theorem'' \cite{Dehn38, Lic62}, stating that 
Dehn twists give rise to generators for the mapping class group of a compact oriented surface.
For instance, the presentation of closed orientable 3-manifolds in terms of framed links in the 3-sphere \cite{Lic62,Ki78}
relies crucially on this fact.
For another instance, Dehn twists appear as monodromy around critical points of Lefschetz fibrations and thus provide 
a combinatorial approach to study this interesting class of $4$-manifolds; see, e.g., the survey article \cite{KS09}.
Compared with general elements of the mapping class group, Dehn twists are easy to handle since the support of $t_C$ lies in an annulus neighborhood of the curve $C$.
So, when one wants to understand a given element of the mapping class group, 
it is sometimes convenient to write it as a product of Dehn twists.

This article is aimed at giving a survey  on  ``generalized'' Dehn twists $t_{\gamma}$
 along \emph{non-simple} closed curves $\gamma$ in an oriented surface.
This construction has been  introduced in \cite{KK14}  and studied in \cite{Kuno13, MT13, KKgroupoid, KK16}.
It originates from the study of the action of ``usual'' Dehn twists
on the fundamental group of the surface, or more precisely, on its nilpotent quotients and its Malcev completion.
After explaining the definition of generalized Dehn twists, we present several results to illustrate how they are related to other objects in low-dimensional topology such as the mapping class group of a surface,  $3$-dimensional cobordisms over the surface, and skein algebras of the surface.

\subsection{Preliminary discussion}
\label{subsec:ridea}

Before giving a precise definition in Section~\ref{sec:Dtf}, 
let us explain how one is led to the generalized Dehn twist $t_\gamma$ along a closed curve $\gamma$.
First of all, notice that the cut-and-paste procedure 
illustrated by \eqref{eq:tC}
apparently does not work if the simple closed curve $C$ is replaced by a self-intersecting closed curve $\gamma$.
Rather, we focus on the action of usual Dehn twists on loops in the surface.

For definiteness and for simplicity, here and throughout, we only consider the case where the surface is a compact oriented connected surface 
$\Sigma := \Sigma_{g,1}$ of genus~$g$ with one boundary component.
Let $\pi := \pi_1(\Sigma,*)$ be the fundamental group of $\Sigma$ with basepoint $*$ chosen from the boundary $\partial \Sigma$.
Let $\mathcal{M}$ be the \emph{mapping class group} of $\Sigma$, namely the group of diffeomorphisms of $\Sigma$ 
fixing the boundary pointwise,
modulo isotopies fixing the boundary.
By the Dehn--Nielsen theorem, the action of $\mathcal{M}$ on $\pi$ induces a canonical isomorphism
\[
\mathcal{M} \overset{\cong}{\longrightarrow} {\rm Aut}_{\partial}(\pi),
\]
where ${\rm Aut}_{\partial}(\pi)$ is the group of automorphisms of $\pi$ fixing the based loop that is parallel to the boundary $\partial \Sigma$.

The action of a usual Dehn twist on $\pi$ is described explicitly as follows.
Let $C\subset {\rm Int}(\Sigma)$ be a simple closed curve.
If $\ell \colon [0,1] \to \Sigma$ is a based loop in $\Sigma$ which intersects $C$ in general position, then $t_C(\ell)$ is obtained by inserting a copy of $C$ (with a suitable orientation) at every intersection of $\ell$ with $C$.
To be more precise, give an orientation to $C$ and let $\ell \cap C = \{ p_1,\ldots,p_n\}$ be the intersection of $\ell$ and $C$, where $0< \ell^{-1}(p_1) < \ell^{-1}(p_2) < \cdots < \ell^{-1}(p_n) < 1$.
We denote by $\varepsilon_i \in \{ \pm 1\}$ the local intersection number at $p_i$ of the two oriented curves $\ell$ and $C$.
Then,
\begin{equation} \label{eq:tCl}
t_C(\ell) = \ell_{*p_1} (C_{p_1})^{\varepsilon_1} \ell_{p_1p_2} (C_{p_2})^{\varepsilon_2} \cdots \ell_{p_{n-1}p_n}(C_{p_n})^{\varepsilon_n} \ell_{p_n*} \in \pi
\end{equation}
where $\ell_{*p_1}$ is the subpath of $\ell$ from $*$ to $p_1$, 
$C_{p_1}$ is the closed curve $C$ based at $p_1$, and so on.
(Concatenation of paths is read from left to right.) 
Note that we need the orientation of $\Sigma$ to define $t_C$, but 
it is easily seen that the right hand side of \eqref{eq:tCl} is independent of the orientation of $C$.

Now let $\gamma \subset {\rm Int}(\Sigma) $ be any closed curve, \emph{with or without} self-intersection. 
A first naive trial to define $t_{\gamma}$ would be to use formula \eqref{eq:tCl} with replacing~$C$ by $\gamma$, but this does not work.
In fact, the right hand side of \eqref{eq:tCl} is not invariant under homotopy of $\ell$ 
when $\gamma$ has self-intersections.
Even if we neglect this fact and pretend that formula \eqref{eq:tCl} works, 
we cannot expect it to define a diffeomorphism of $\Sigma$.
This is because when $\ell$ is simple, the right hand side of \eqref{eq:tCl} may have non-trivial self-intersection
(arising from~$\gamma$), while any diffeomorphism of $\Sigma$ must preserve simple paths in $\Sigma$.

One outcome of this apparently hopeless situation is to consider, instead, the formal linear combination
\begin{equation} \label{eq:sigmaCl}
\sigma(C)(\ell) := \sum_{i=1}^n \varepsilon_i\, \ell_{*p_i} C_{p_i} \ell_{p_i*} \in \mathbb{Z} \pi,
\end{equation}
which one may view as a ``linearization'' of \eqref{eq:tCl}.
A key fact is that the right hand side is homotopy invariant even if we replace $C$ with any closed (oriented) curve $\gamma$. 
By linearity, any linear combination $u$ of 
free loops in $\Sigma$ gives an endomorphism $\sigma(u) \colon \mathbb{Z}\pi \to \mathbb{Z}\pi$, 
which turns out to be a derivation of the group ring $\mathbb{Z}\pi$.

The result in \cite{KK14, MT13, KK16} describes the action of $t_C$ on the group ring $\mathbb{Z}\pi$ 
as the exponential of a derivation of $\mathbb{Z}\pi$, 
which depends only on the homotopy class of $C$ and is built from the action \eqref{eq:sigmaCl}.
To be more specific, in order to work with the exponential, one has to replace $\Z \pi$ with the $I$-adic completion $\widehat{\Q \pi}$ of the group algebra $\Q \pi$, where $I$ denotes the augmentation ideal. 
Then the formula is 
\begin{equation} \label{eq:tCexp}
t_C = \exp \left( \sigma \big( \frac{1}{2} (\log C)^2 \big) \right)
\in {\rm Aut}(\widehat{\mathbb{Q}\pi}).
\end{equation}

Of course, the right hand side of \eqref{eq:tCexp}
makes sense if we replace $C$ with any closed curve $\gamma$.
Thus we define the \emph{generalized Dehn twist} along $\gamma$ to be
\begin{equation} \label{eq:tCexpgen}
t_{\gamma} := \exp \left( \sigma \big( \frac{1}{2} (\log \gamma)^2 \big) \right).
\end{equation}
In general, $t_{\gamma}$ does no longer preserve the fundamental group $\pi \subset \widehat{\mathbb{Q}\pi}$, but turns out to preserve the \emph{Malcev completion} $\widehat{\pi}$ of $\pi$.

\subsection{Organization}

This expository paper is organized as follows.

In Section~\ref{sec:Dtf}, we give the precise definition of a generalized Dehn twist~$t_{\gamma}$ along a closed curve $\gamma$.
It is defined as an element of the generalized mapping class group $\widehat{\calM} := {\rm Aut}_{\partial}(\widehat{\pi})$, i.e$.$ the group of automorphisms of $\widehat{\pi}$ that preserves the element corresponding to the boundary of $\Sigma$.
The group~$\widehat{\calM}$ naturally contains the mapping class group $\calM$ as a subgroup.

Section~\ref{sec:realize} is concerned with the question whether $t_{\gamma} \in \widehat{\calM}$ 
is induced by a diffeomorphism of $\Sigma$ or not, i.e$.$ $t_{\gamma} \in \calM$ or not.
In fact, as we will see, 
there are many examples of closed curves $\gamma$ for which the answer to this question is negative.

In Section~\ref{sec:diagram}, we give a diagrammatic description of generalized Dehn twists in terms of decorated trees 
whose leaves are colored by the first homology group of $\Sigma$.
Although we do not explicitly review this here, 
that point of view is closely related to the theory of the Johnson homomorphisms for the mapping class group \cite{Joh83, Mo93}.
Using this diagrammatic description, we show the following analogue of the Dehn--Lickorish theorem, which seems to be new:
the generalized mapping class group $\widehat{\calM}$ is topologically generated by (rational powers of) generalized Dehn twists.

In Section \ref{sec:Fox}, we present an approach of \cite{MT13} 
to generalize Dehn twists which is based on the notion of ``Fox pairing''.
This approach enables us to consider ``twists'' in a more algebraic setting.
We mention some examples arising elsewhere in topology, and which could be interesting for further investigation.

In Section~\ref{sec:hc}, we review recent results from \cite{KM19} 
about a (partial) topological interpretation of $t_{\gamma}$ in terms of $3$-dimensional surgery.
In more detail, given a closed curve $\gamma \subset {\rm Int}(\Sigma)$, 
we choose a knot resolution $K$  of $\gamma$ in $U:= \Sigma \times [-1, + 1]$, and perform a surgery along $K$ with a suitable framing.
We compare the resulting homology cobordism, denoted by $U_K$, 
with $t_{\gamma}$ through their respective actions on the Malcev completion $\widehat{\pi}$.

Finally, in Section~\ref{sec:skein}, we summarize some results from the series of papers 
\cite{TsujiCSAI, TsujiCSApure, TsujiTorelli, Tsujihom3, TsujiHOMFLY-PTskein}. 
We explain two variations of formula \eqref{eq:tCexp} 
that exist for the Kauffman bracket skein algebra, on the one hand, and for the HOMFLY-PT skein algebra, on the other hand. 
The resulting skein versions of generalized Dehn twists are related to \eqref{eq:tCexpgen} via some commutative diagrams.
We conclude by reviewing applications of this skein approach of  mapping class groups to the construction of topological invariants of homology 3-spheres.

\subsection{Acknowledgments}

The main tools used in the study of generalized Dehn twists are algebraic operations on loop spaces in an oriented surface
which are defined by intersection, or, self-intersection.
We would like to point out that Vladimir Turaev  already introduced in 1978 such kind of operations \cite{Tur78}:
these should be viewed as precursors of the so-called ``Goldman bracket'' \cite{Gol86} and ``Turaev cobracket'' \cite{Tur91},
as well as the loop action \eqref{eq:sigmaCl} that is discussed above.
Thus, we hope that the reader will be convinced that the framework of  generalized Dehn twists borrows much
to the influential works of Turaev in low-dimensional topology, from the late 1970's to nowadays.

This paper was finalized while the two first-named authors were visiting the CRM in Montr\'eal:
they thank the Simons Foundation for support.
Y.K. is supported by JSPS KAKENHI 18K03308;  G.M. is supported in part  by the  project ITIQ-3D, funded by the
``R\'egion Bourgogne Franche-Comt\'e.''
S.T. is supported by JSPS KAKENHI 18J00305 and the Research  Institute for Mathematical
Sciences, an International Joint Usage/Research Center located in Kyoto University.

\section{The Dehn twist formula and generalized Dehn twists}
\label{sec:Dtf}

In this section, we review the definition of generalized Dehn twists.
Recall that $\pi$ is the fundamental group of the surface $\Sigma = \Sigma_{g,1}$ with basepoint $*\in \partial \Sigma$.

\subsection{Malcev completion}

Let $I$ be the augmentation ideal of the group algebra $\Q \pi$.
An element of $I$ is a formal $\Q$-linear combination $\sum_{x\in \pi} a_x x$, where $a_x=0$ for all but finite $x$ and $\sum_{x\in \pi} a_x =0$.
The powers $\{ I^m \}_m$ define a multiplicative filtration of $\Q \pi$.
The \emph{$I$-adic completion} of $\Q \pi$ is the projective limit
\[
\widehat{\Q \pi} := \varprojlim_{m} \Q \pi /I^m.
\]
The powers of $I$ induce a natural filtration of $\widehat{\Q \pi}$ which we denote by $\{ \widehat{I^{m}}\}_m$.

There is also a canonical Hopf algebra structure on $\Q \pi$ whose coproduct is given by $\Delta(x) = x \otimes x$ 
for any $x\in \pi$, and this induces a complete Hopf algebra structure on $\widehat{\Q \pi}$.
The \emph{Malcev completion} $\widehat{\pi}$ of $\pi$ is defined to be the set of group-like elements in $\widehat{\Q\pi}$: 
\[
\widehat{\pi} := \big\{  x\in \widehat{\Q\pi} \mid x\neq 0, \Delta(x) = x \widehat{\otimes} x \big\}.
\]
There is a filtration of the group $\widehat{\pi}$ whose $m$th term is $\widehat{\pi}_m := {\widehat{\pi} \cap (1 + \widehat{I^m})}$.
Since $\pi$ is a free group of finite rank, the natural map $\Q \pi \to \widehat{\Q \pi}$ is injective, and so is the natural map $\pi \to \widehat{\pi}$.
Let $\pi ~=~\Gamma_1\pi~\supset~\Gamma_2\pi~\supset~\Gamma_3\pi\supset \cdots$ be the lower central series of the group $\pi$.
For each $m \ge 0$, the map $\pi \to \widehat{\pi}$ induces an injective group homomorphism
\begin{equation} \label{eq:Mal_k}
\pi/\Gamma_m \pi \longrightarrow \widehat{\pi}/\widehat{\pi}_m.
\end{equation}

\subsection{The action of free loops on based loops}

Let us recall the operation~$\sigma$ introduced in \cite{KK14}.
Let $\alpha$ be an (oriented) free loop in $\Sigma$ and $\beta$ a based loop, and assume that they intersect in transverse double points.
Set
\begin{equation} \label{eq:sigma}
\sigma(\alpha)(\beta) := \sum_{p\in \alpha \cap \beta}
\varepsilon_p\, \beta_{*p} \alpha_p \beta_{p*}
\in \Z \pi.
\end{equation}
Here, the sum is taken over all the intersections of $\alpha$ and $\beta$, $\varepsilon_p \in \{ \pm 1\}$ is the local intersection number of $\alpha$ and $\beta$ at $p$, and $\beta_{*p}$, $\alpha_p$, and $\beta_{p*}$ have the same meaning as in formula \eqref{eq:tCl}.
Extending by linearity, we obtain a $\Q$-linear map
\[
\sigma(u)\colon \Q \pi \longrightarrow \Q \pi
\]
for any $\Q$-linear combination $u$ of (homotopy classes of) free loops in $\Sigma$.
It~is in fact a derivation of $\Q \pi$: for any $v_1,v_2\in \Q \pi$,
\[
\sigma(u) (v_1v_2) = (\sigma(u)(v_1))\, v_2 + v_1\, (\sigma(u)(v_2)).
\]

Let $\alpha$ be a free loop in $\Sigma$.
For each $m\in \Z$, let $\alpha^m$ be the $m$th power of $\alpha$.
For any polynomial $f(x) \in \Q[x]$, the expression $f(\alpha)$ makes sense as a $\Q$-linear combination of free loops in $\Sigma$, 
so that the derivation $\sigma(f(\alpha)) \colon \Q\pi \to \Q\pi$ is defined.
As is proved in \cite{KKgroupoid,KK15,MT13}, it holds that 
\begin{equation} \label{eq:sigma-mn}
\sigma( (\alpha-1)^m ) (I^n) \subset I^{m+n-2}
\quad \text{for any $m,n \ge 0$}
\end{equation}
(with the convention that $\Q\pi = I^0=I^{-1}=I^{-2}$).
Therefore, for any power series $f(x) \in \Q[[x-1]]$, one can consider the derivation
\[
\sigma(f(\alpha)) \colon \widehat{\Q \pi} \longrightarrow \widehat{\Q \pi},
\]
which is continuous with respect to the filtration $\{ \widehat{I^m} \}_m$.

\begin{rem} \label{rem:operations}
\begin{enumerate}
\item
The operation $\sigma$ is a refinement of the Goldman bra\-cket \cite{Gol86} of two (homotopy classes of) free loops in $\Sigma$.
In order to give a precise statement, for a based loop $\alpha$ in $\Sigma$ denote by $|\alpha|$ the free homotopy class of $\alpha$.
(We also apply this convention to $\Q$-linear combinations of based loops in $\Sigma$.)
Let $\alpha$ and $\beta$ be based loops in~$\Sigma$.
Then the Goldman bracket $\big[ |\alpha|, |\beta| \big]$ of the two free loops $|\alpha|$ and $|\beta|$ 
is equal to $\big| \sigma( |\alpha|) (\beta) \big|$.
\item
The operation $\sigma$ itself has a refinement, 
which is called the \emph{homotopy intersection form} and denoted by $\eta\colon \Q\pi \times \Q\pi \to \Q\pi$.
This form, which can be regarded as a ``universal'' version of Reidemeister's equivariant  intersection pairings, 
is explicit in \cite{Tur78} and implicit in~\cite{Papa75}; see also~\cite{Per06,MT13}.
Let $\alpha$ and $\beta$ be immersed based loops in $\Sigma$ such that their intersections in the interior of $\Sigma$ consists of finitely many transverse double points, and in a neighborhood of the basepoint of $\Sigma$, they are arranged as shown in the following figure:
\begin{center}
\vskip 0.5em
{\unitlength 0.1in%
\begin{picture}(10.0000,3.7200)(4.0000,-8.0000)%
%
\special{pn 13}%
\special{pa 400 800}%
\special{pa 1400 800}%
\special{fp}%
%
\special{pn 4}%
\special{sh 1}%
\special{ar 900 800 16 16 0 6.2831853}%
%
\special{pn 8}%
\special{pa 900 800}%
\special{pa 700 640}%
\special{fp}%
%
\special{pn 8}%
\special{pa 900 800}%
\special{pa 1100 640}%
\special{fp}%
%
\special{pn 8}%
\special{pa 900 800}%
\special{pa 820 600}%
\special{fp}%
%
\special{pn 8}%
\special{pa 900 800}%
\special{pa 980 600}%
\special{fp}%
%
\special{pn 8}%
\special{pn 8}%
\special{pa 820 600}%
\special{pa 816 593}%
\special{fp}%
\special{pa 797 559}%
\special{pa 794 551}%
\special{fp}%
\special{pa 779 516}%
\special{pa 776 508}%
\special{fp}%
\special{pa 762 472}%
\special{pa 759 465}%
\special{fp}%
\special{pa 733 437}%
\special{pa 726 433}%
\special{fp}%
\special{pa 688 429}%
\special{pa 679 430}%
\special{fp}%
\special{pa 642 440}%
\special{pa 634 442}%
\special{fp}%
\special{pa 600 460}%
\special{pa 593 465}%
\special{fp}%
\special{pa 575 498}%
\special{pa 574 506}%
\special{fp}%
\special{pa 589 541}%
\special{pa 594 548}%
\special{fp}%
\special{pa 621 576}%
\special{pa 627 581}%
\special{fp}%
\special{pa 656 606}%
\special{pa 663 611}%
\special{fp}%
\special{pa 693 635}%
\special{pa 700 640}%
\special{fp}%
%
\special{pn 8}%
\special{pn 8}%
\special{pa 980 600}%
\special{pa 984 593}%
\special{fp}%
\special{pa 1003 559}%
\special{pa 1006 551}%
\special{fp}%
\special{pa 1021 516}%
\special{pa 1024 508}%
\special{fp}%
\special{pa 1038 472}%
\special{pa 1041 465}%
\special{fp}%
\special{pa 1067 437}%
\special{pa 1074 433}%
\special{fp}%
\special{pa 1112 429}%
\special{pa 1121 430}%
\special{fp}%
\special{pa 1158 440}%
\special{pa 1166 442}%
\special{fp}%
\special{pa 1200 460}%
\special{pa 1207 465}%
\special{fp}%
\special{pa 1225 498}%
\special{pa 1226 506}%
\special{fp}%
\special{pa 1211 541}%
\special{pa 1206 548}%
\special{fp}%
\special{pa 1179 576}%
\special{pa 1173 581}%
\special{fp}%
\special{pa 1144 606}%
\special{pa 1137 611}%
\special{fp}%
\special{pa 1107 635}%
\special{pa 1100 640}%
\special{fp}%
\put(8.6000,-9.2000){\makebox(0,0)[lb]{$*$}}%
\put(5.6000,-7.2000){\makebox(0,0)[lb]{$\alpha$}}%
\put(11.2000,-7.6000){\makebox(0,0)[lb]{$\beta$}}%
\end{picture}}%
\vskip 0.5em
\end{center}
Then,
\begin{equation} \label{eq:eta}
\eta(\alpha,\beta) := \sum_{p\in \alpha \cap \beta} \varepsilon_p\, \alpha_{*p}\beta_{p*}.
\end{equation}
How to reconstruct $\sigma$ from $\eta$ will be explained in Example \ref{ex:hif}.
\end{enumerate}
\end{rem}

\subsection{Logarithms of Dehn twists}

Consider now the power series
\[
L(x) := \frac{1}{2} (\log x)^2 \in \Q[[x-1]],
\]
where
$$
\log x = \sum_{n=1}^{\infty} \frac{(-1)^{n-1}}{n} (x-1)^n.
$$
Let $\gamma$ be an (unoriented) closed curve in $\Sigma$.
We pick an orientation of  
$\gamma$ but, for simplicity, we use the same letter $\gamma$ for the resulting free loop.
Since $L(x^{-1}) = L(x)$, the derivation $\sigma(L(\gamma)) \colon \widehat{\Q \pi} \to \widehat{\Q \pi}$ does not depend 
on this choice of orientation.

\begin{rem} \label{rem:piconj}
The set of conjugacy classes in $\pi$, denoted by $\vert \pi \vert$, is naturally identified 
with the set of free homotopy classes of loops in $\Sigma$.
Thus the $\Q$-linear span $\Q \vert \pi \vert$ is the underlying vector space for the Goldman bracket.
As was shown in \cite{KKgroupoid, KK15, MT13}, 
there is a natural filtration of $\Q \vert \pi \vert$ defined by using the filtration $\{ I^m \}_m$ of $\Q \pi$ 
and the canonical projection $\zettaiti{ -}\colon  \Q \pi \to \Q \vert \pi \vert$.
Then, the expression $L(\gamma)$ makes sense as an element of the completion of~$\Q \vert \pi \vert$ with respect to this filtration.
\end{rem}

\begin{thm}[\cite{KK14,MT13}] \label{thm:t_Cformula}
Let $C$ be a simple closed curve in ${\rm Int}(\Sigma)$. 
Then the exponential of the derivation $\sigma(L(C))$ converges, and it coincides with the action of the Dehn twist $t_C$ on $\widehat{\Q \pi}$:
\[
t_C = \exp \big(  \sigma(L(C)) \big)  \colon \widehat{\Q \pi} \longrightarrow \widehat{\Q \pi}.
\] 
\end{thm}

\subsection{Definition of a generalized Dehn twist} \label{subsec:gDt}

Let $\gamma$ be a closed curve in~$\Sigma$.
As was shown in \cite{KKgroupoid, KK15, MT13}, for any $\gamma$ the exponential of the derivation $\sigma(L(\gamma))$ converges and defines a filtration-preserving algebra automorphism of $\widehat{\Q \pi}$.
Theorem \ref{thm:t_Cformula} shows that when $\gamma$ is simple, $\exp(\sigma(L(\gamma)))$ is the action on $\widehat{\Q \pi}$ of the usual Dehn twist along $\gamma$.
In this case, $\exp(\sigma(L(\gamma)))\colon \widehat{\Q \pi} \to \widehat{\Q \pi}$ 
is clearly a Hopf algebra automorphism since it preserves $\pi$.
In fact, $\exp(\sigma(L(\gamma)))$ is a Hopf algebra automorphism for any closed curve $\gamma$ \cite[~\S 5]{MT13}, see also \cite[\S 5]{KK16}.
Thus, by restriction, $\exp(\sigma(L(\gamma)))$ can be regarded as an automorphism of the Malcev completion $\widehat{\pi}$ of $\pi$.
Furthermore, $t_{\gamma}$ preserves the boundary loop $\zeta \in \pi \subset \widehat{\pi}$ 
since $\sigma(\alpha)(\zeta) = 0$ for any free loop $\alpha$.

\begin{dfn}
The \emph{generalized Dehn twist along $\gamma$} is the automorphism of the complete Hopf algebra $\widehat{\Q \pi}$ defined by
\[
t_{\gamma}:= \exp \big(  \sigma(L(\gamma)) \big) 
\colon \widehat{\Q \pi} \longrightarrow \widehat{\Q \pi},
\] 
or equivalently, its restriction to the Malcev completion $\widehat{\pi}$.
\end{dfn}

\noindent
Note that, 
for any integer $n\ge 0$, it holds that $t_{\gamma^n} = (t_{\gamma})^{n^2}$.

\subsection{Action on the nilpotent quotients of $\pi$}
\label{subsec:anq}

We say that a closed curve  $\gamma$ in $\Sigma$ 
is of \emph{nilpotency class} $\ge k$
if its homotopy class in $\pi$ (after some arbitrary choices of orientation and basing arc) lies in $\Gamma_k \pi$.
The generalized Dehn twists act on $\widehat{\pi}$ and hence on the quotient $\widehat{\pi}/\widehat{\pi}_m$ for each $m\ge 0$.
It~will turn out in  Section \ref{sec:hc} that,
if $\gamma$ is of nilpotency class $\ge k$, the action of $t_{\gamma}$ on $\widehat{\pi}/\widehat{\pi}_{2k+1}$ preserves the image of the map \eqref{eq:Mal_k} with $m=2k+1$, and therefore, 
$t_{\gamma}$ acts on the free nilpotent group $\pi/\Gamma_{2k+1}\pi$.
In particular, $t_{\gamma}$~acts on $\pi/\Gamma_{2k}\pi$.
The following result describes this action in the manner of formula \eqref{eq:tCl}.

\begin{prop} \label{prop:quotient}
Let $\gamma \subset {\rm Int}(\Sigma)$ be a closed curve of nilpotency class~${\ge k}$ and give an orientation to $\gamma$.
Let $\ell\colon [0,1]\to \Sigma$ be a based loop which intersects $\gamma$ in general position, and let $\ell \cap \gamma = \{ p_1,\ldots,p_n \}$ with $\ell^{-1}(p_1) < \ell^{-1}(p_2)< \cdots < \ell^{-1}(p_n)$.
For each $i\in \{1,\ldots,n\}$, let $\varepsilon_i \in \{\pm 1\}$ be the sign of intersection of $\ell$ and $\gamma$ at $p_i$.
Then, the class $\{ \ell \}_{2k-1} \in \pi/\Gamma_{2k} \pi$ of $\ell \in \pi$
is mapped by $t_\gamma$ to
\begin{equation*} 
t_{\gamma}(\{ \ell\}_{2k-1} )
=
\ell_{*p_1} (\gamma_{p_1})^{\varepsilon_1} \ell_{p_1p_2} (\gamma_{p_2})^{\varepsilon_2}
 \cdots \ell_{p_{n-1}p_n}( \gamma_{p_n})^{\varepsilon_n} \ell_{p_n*}
 \in \frac{\pi}{\Gamma_{2k} \pi}.
\end{equation*}
\end{prop}

\begin{proof}
Let $\ell'$ be the based loop in the right hand side of the above formula.
We need to prove that $t_{\gamma}(\ell)\, (\ell')^{-1} \in \widehat{\pi}_{2k}$.
This is equivalent to showing that $t_{\gamma}(\ell)\, \ell^{-1} \equiv \ell' \ell^{-1}$ modulo $\widehat{I^{2k}}$.
By assumption on $\gamma$, we have $\gamma-1 \in I^k$.
Using \eqref{eq:sigma-mn} and the fact that the leading term of $L(x)$ is $(x-1)^2/2$, we obtain
\[
t_{\gamma}(\ell)\, \ell^{-1} \equiv
1 + \sigma \left( \frac{1}{2}(\gamma-1)^2 \right)\! (\ell) \, \ell^{-1} \mod \widehat{I^{2k}}.
\]
Thus, the proof is reduced to showing that
\begin{equation} \label{eq:ell'ell}
\ell' \ell^{-1} -1 \equiv
\sigma\left( \frac{1}{2}(\gamma-1)^2 \right)\! (\ell) \, \ell^{-1}
\mod I^{2k}.
\end{equation}

To simplify notation, for each $j\in \{1,\ldots, n\}$ put $\delta_j:= \ell_{*p_j} (\gamma_{p_j}) \overline{\ell_{*p_j}}$
where $\overline{\ell_{*p_j}}$ denotes the reverse of the path $\ell_{*p_j}$.
Note that $\delta_j \in \Gamma_k \pi$ by our assumption on $\gamma$.
Then, on the one hand, we compute
\[
\ell' \ell^{-1} - 1 = \left( \prod_{j=1}^n \delta_j^{\varepsilon_j} \right) -1
= \sum_{j=1}^n \delta_1^{\varepsilon_1}\cdots \delta_{j-1}^{\varepsilon_{j-1}} (\delta_j^{\varepsilon_j} -1 )
\equiv \sum_{j=1}^n (\delta_j^{\varepsilon_j} -1),
\]
where the last equivalence is modulo $I^{2k}$.
On the other hand, we compute
\begin{align*}
\sigma\left( \frac{1}{2}(\gamma-1)^2 \right)\! (\ell) \, \ell^{-1}
&= \sigma\left( \frac{1}{2} \gamma^2 - \gamma \right)\! (\ell) \, \ell^{-1} \\
&= \left( \sum_{j=1}^n \varepsilon_j\, \ell_{*p_j} (\gamma_{p_j}^2 - \gamma_{p_j} ) \ell_{p_j *} \right) \ell^{-1} \\
&= \sum_{j=1}^n \varepsilon_j (\delta_j^2 - \delta_j).
\end{align*}
Finally, using again the fact that $\delta_j \in \Gamma_k \pi$, we have
$\varepsilon_j (\delta_j^2 - \delta_j) \equiv \varepsilon_j (\delta_j-1) \equiv (\delta_j^{\varepsilon_j}-1)$ modulo $I^{2k}$.
This proves \eqref{eq:ell'ell}.
\end{proof}

\section{Realizability as diffeomorphisms}
\label{sec:realize}

Let ${\rm Aut}_{\partial}(\pi)$ be the group of automorphisms of $\pi$ that fix the boundary element $\zeta$.
Then, by the Dehn--Nielsen isomorphism $\mathcal{M} \cong {\rm Aut}_{\partial}(\pi)$,
we can regard the mapping class group as a subgroup 
of the group of  automorphisms of  the filtered group $\hpi$ that fix the boundary element $\zeta \in \hpi$:
\begin{equation} \label{eq:form1}
 \widehat{\calM} := \Aut_\partial (\hpi) .
\end{equation}
We shall refer to $\widehat{\calM}$ as the \emph{generalized mapping class group} of $\Sigma$.
It can be equivalently defined as the group of automorphisms of the complete Hopf algebra $\widehat{\Q \pi}$ that preserve
the  homotopy intersection form $\eta\colon \widehat{\Q \pi} \times \widehat{\Q \pi} \to \widehat{\Q \pi}$:
\begin{equation} \label{eq:form2}
 \widehat{\calM} = \Aut_{\eta}\big(\widehat{\Q \pi} \big) .
\end{equation}
(See \cite[\S 8.1 \&  \S 10.3]{MT13} for the equivalence between the two definitions.)

As we have seen in Section \ref{subsec:gDt}, the generalized Dehn twists $t_{\gamma}$ are defined as elements in $\widehat{\calM}$.
We say that $t_{\gamma}$ is \emph{realizable as a diffeomorphism} if $t_{\gamma} \in \mathcal{M}$.

\begin{prob}
Given a closed curve $\gamma$ in $\Sigma$, determine whether $t_{\gamma}$ is realizable as a diffeomorphism or not.
\end{prob}

The following result generalizes the fact that the support of the usual Dehn twist $t_C$ is in an annulus neighborhood of $C$.

\begin{thm}[\cite{Kuno13,KKgroupoid}] \label{thm:locsupp}
Let $\gamma$ be an immersed closed curve in ${\rm Int}(\Sigma)$, and
suppose that $t_{\gamma}$ is realizable as a diffeomorphism.
Then there is an orientation-preserving diffeomorphism of $\Sigma$ which represents $t_{\gamma}$ and whose support lies in a regular neighborhood of $\gamma$.
\end{thm}

It is conjectured that $t_{\gamma}$ is \emph{not} realizable as a diffeomorphism
 unless $\gamma$ is homotopic to a power of a simple closed curve \cite{Kuno13, KK16}.
The following result produces many examples of closed curves $\gamma$ such that $t_{\gamma} \notin \calM$.

\begin{thm}[\cite{KK15}] \label{thm:pioneinj}
Let $\gamma$ be an immersed non-simple closed curve in $ {\rm Int}(\Sigma)$ whose self-intersections consist of transverse double points.
If the inclusion homomorphism $\pi_1(N(\gamma)) \to \pi_1(\Sigma)$ is injective, where $N(\gamma)$ is a closed regular neighborhood of $\gamma$, then $t_{\gamma}$ is not realizable as a diffeomorphism.
\end{thm}

The proof of this theorem uses Theorem \ref{thm:locsupp} and a certain operation measuring self-intersections of loops in $\Sigma$ which is essentially equivalent to the operation introduced by Turaev \cite{Tur78}.

\begin{example} 
Figure \ref{fig:two_examples} shows two examples of a closed curve $\gamma$ in the surface of genus two.
These examples are easily seen to satisfy the assumption of Theorem~\ref{thm:pioneinj}, and thus $t_{\gamma} \not\in \calM$.
The example in the left part is a \emph{figure eight}, i.e$.$ a closed curve with a single transverse double point.
In fact, if a figure eight~$\gamma$ is not homotopic to a simple closed curve, 
 nor to the square of a simple closed curve, then we can use Theorem \ref{thm:pioneinj} to conclude that $t_{\gamma} \not\in \calM$.
\begin{figure}[htb]
\begin{center}
{\unitlength 0.1in%
\begin{picture}(39.0000,8.0000)(2.0000,-13.0000)%
%
\special{pn 13}%
\special{ar 1700 900 100 400 0.0000000 6.2831853}%
%
\special{pn 13}%
\special{ar 500 900 300 400 1.5707963 4.7123890}%
%
\special{pn 13}%
\special{ar 4000 900 100 400 0.0000000 6.2831853}%
%
\special{pn 13}%
\special{ar 2800 900 300 400 1.5707963 4.7123890}%
%
\special{pn 13}%
\special{ar 600 900 150 150 0.0000000 6.2831853}%
%
\special{pn 13}%
\special{ar 1200 900 150 150 0.0000000 6.2831853}%
%
\special{pn 13}%
\special{ar 2900 900 150 150 0.0000000 6.2831853}%
%
\special{pn 13}%
\special{ar 3500 900 150 150 0.0000000 6.2831853}%
%
\special{pn 13}%
\special{pa 1700 1300}%
\special{pa 500 1300}%
\special{fp}%
%
\special{pn 13}%
\special{pa 1700 500}%
\special{pa 500 500}%
\special{fp}%
%
\special{pn 13}%
\special{pa 4000 1300}%
\special{pa 2800 1300}%
\special{fp}%
%
\special{pn 13}%
\special{pa 4000 500}%
\special{pa 2800 500}%
\special{fp}%
%
\special{pn 4}%
\special{sh 1}%
\special{ar 900 900 16 16 0 6.2831853}%
%
\special{pn 8}%
\special{ar 600 900 280 280 1.5707963 4.7123890}%
%
\special{pn 8}%
\special{pa 600 620}%
\special{pa 632 620}%
\special{pa 663 623}%
\special{pa 691 632}%
\special{pa 717 645}%
\special{pa 742 662}%
\special{pa 765 683}%
\special{pa 786 707}%
\special{pa 806 734}%
\special{pa 826 763}%
\special{pa 844 794}%
\special{pa 862 826}%
\special{pa 880 860}%
\special{pa 900 900}%
\special{fp}%
%
\special{pn 8}%
\special{pa 600 1180}%
\special{pa 632 1180}%
\special{pa 663 1177}%
\special{pa 691 1168}%
\special{pa 717 1155}%
\special{pa 742 1138}%
\special{pa 765 1117}%
\special{pa 786 1093}%
\special{pa 806 1066}%
\special{pa 826 1037}%
\special{pa 844 1006}%
\special{pa 862 974}%
\special{pa 880 940}%
\special{pa 900 900}%
\special{fp}%
%
\special{pn 8}%
\special{pa 1200 1180}%
\special{pa 1168 1180}%
\special{pa 1137 1177}%
\special{pa 1109 1168}%
\special{pa 1083 1155}%
\special{pa 1058 1138}%
\special{pa 1035 1117}%
\special{pa 1014 1093}%
\special{pa 994 1066}%
\special{pa 974 1037}%
\special{pa 956 1006}%
\special{pa 938 974}%
\special{pa 920 940}%
\special{pa 900 900}%
\special{fp}%
%
\special{pn 8}%
\special{pa 1200 620}%
\special{pa 1168 620}%
\special{pa 1137 623}%
\special{pa 1109 632}%
\special{pa 1083 645}%
\special{pa 1058 662}%
\special{pa 1035 683}%
\special{pa 1014 707}%
\special{pa 994 734}%
\special{pa 974 763}%
\special{pa 956 794}%
\special{pa 938 826}%
\special{pa 920 860}%
\special{pa 900 900}%
\special{fp}%
%
\special{pn 8}%
\special{ar 1200 900 280 280 4.7123890 1.5707963}%
%
\special{pn 8}%
\special{ar 2900 900 240 240 1.5707963 4.7123890}%
%
\special{pn 8}%
\special{pa 2900 1140}%
\special{pa 3500 1140}%
\special{fp}%
%
\special{pn 8}%
\special{pn 8}%
\special{pa 3650 900}%
\special{pa 3650 908}%
\special{fp}%
\special{pa 3649 944}%
\special{pa 3648 952}%
\special{fp}%
\special{pa 3644 987}%
\special{pa 3644 994}%
\special{fp}%
\special{pa 3638 1029}%
\special{pa 3636 1036}%
\special{fp}%
\special{pa 3628 1070}%
\special{pa 3626 1077}%
\special{fp}%
\special{pa 3616 1110}%
\special{pa 3613 1118}%
\special{fp}%
\special{pa 3600 1149}%
\special{pa 3597 1156}%
\special{fp}%
\special{pa 3580 1187}%
\special{pa 3577 1193}%
\special{fp}%
\special{pa 3558 1221}%
\special{pa 3553 1226}%
\special{fp}%
\special{pa 3532 1249}%
\special{pa 3528 1253}%
\special{fp}%
\special{pa 3503 1273}%
\special{pa 3497 1277}%
\special{fp}%
\special{pa 3468 1291}%
\special{pa 3461 1294}%
\special{fp}%
\special{pa 3428 1300}%
\special{pa 3420 1300}%
\special{fp}%
%
\special{pn 8}%
\special{ar 3420 900 230 400 1.5707963 3.1415927}%
%
\special{pn 8}%
\special{ar 3580 900 70 240 4.7123890 6.2831853}%
%
\special{pn 8}%
\special{pa 3580 660}%
\special{pa 2900 660}%
\special{fp}%
%
\special{pn 8}%
\special{ar 3500 860 280 280 4.7123890 1.5707963}%
%
\special{pn 8}%
\special{ar 3500 900 310 320 3.1415927 4.7123890}%
%
\special{pn 4}%
\special{sh 1}%
\special{ar 3300 660 16 16 0 6.2831853}%
%
\special{pn 4}%
\special{sh 1}%
\special{ar 3238 1140 16 16 0 6.2831853}%
\end{picture}}%
\end{center}
\caption{Closed curves $\gamma$ such that $t_{\gamma} \notin \calM$}
\label{fig:two_examples}
\end{figure}
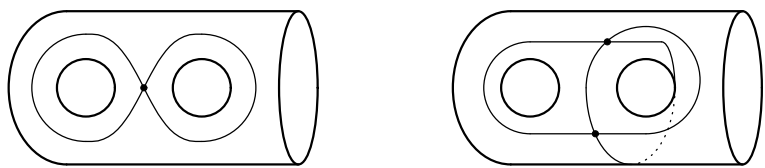
\end{example}

\section{Diagrammatic formulation of  generalized Dehn twists}
\label{sec:diagram}

Generalized Dehn twists have a useful description in terms of so-called ``Jacobi diagrams''.
In this section, we review this description and draw some consequences.

\subsection{Generalized Dehn twists as Lie  automorphisms}

Generalized Dehn twists have been defined in Section \ref{subsec:gDt} as automorphisms of the complete Hopf algebra $\widehat{\Q\pi}$
or, by restriction, as automorphisms of the Malcev completion $\widehat{\pi}$.
In some situations, however, it is appropriate to swap $\widehat{\pi}$ for its ``infinitesimal'' analogue,
namely the \emph{Malcev Lie algebra} $\frakM(\pi)$ of~$\pi$. 
The latter can be defined as the primitive part of $\widehat{\Q\pi}$:
$$
\frakM(\pi) := \big\{ x \in \widehat{\Q\pi}\, \vert\,  \Delta(x) =  x \widehat{\otimes} 1 + 1 \widehat{\otimes} x \big\}.
$$
Indeed,  the Malcev completion and the Malcev Lie algebra correspond one to the other
through the exponential and logarithm series
\begin{equation} \label{eq:correspondence}
1 + \widehat{ I^1  } \supset \widehat{\pi}
\xymatrix{
\ar@/^0.5pc/[rr]^-\log_-{\cong} &&  \ar@/^0.5pc/[ll]^-\exp
}
\frakM({\pi}) \subset  \widehat{ I^1  }
\end{equation}
which are defined, for all $u\in \widehat{ I^1  }$, by 
$$
\exp(u) =  \sum_{k=0}^{\infty}  \frac{u^k}{k!}
\quad \hbox{and} \quad
\log(1+u) =  \sum_{k=1}^{\infty}
(-1)^{k+1}\frac{u^k}{k},
$$
respectively.
For any closed curve $\gamma$ in $\Sigma$, the generalized Dehn twist $t_\gamma \in \Aut(\widehat{\Q\pi})$  
restricts to an automorphism of $\frakM(\pi)$, which preserves the element~$\log(\zeta)$.
We denote this restriction in the same way:
$$
 t_\gamma \in \Aut(\frakM(\pi)).
$$

\begin{rem}
We are using the same notation for three different kinds of automorphisms:
$$
\text{(i)} \ t_\gamma  \in \Aut(\widehat{\Q\pi}), \quad \text{(ii)} \ t_\gamma  \in \Aut(\widehat{\pi}), \quad 
\text{(iii)} \  t_\gamma  \in \Aut(\frakM({\pi})).
$$
To sum up: each of (ii) and (iii) are restrictions of (i); furthermore,
(ii) and (iii) correspond each other through the correspondence \eqref{eq:correspondence}.
\end{rem}

Next, to make this Lie algebra automorphism $t_\gamma$ more concrete,
one can swap the Malcev Lie algebra $\frakM(\pi)$ for the free Lie algebra 
$$
\frakL:= \frakL(H)
$$
generated by $H:=H_1(\Sigma;\Q)$. But, this can not be done in a canonical way 
and requires a notion that has been coined in \cite{Ma12}
as a ``symplectic expansion'' of the free group $\pi$.
Recall that $\frakL$ is the primitive part of the tensor algebra $T(H)$ generated by $H$,
and that $\frakL$ is a graded Lie algebra:
$$
\frakL = \bigoplus_{j=1}^{\infty} \frakL_j
\quad \hbox{ where } \frakL_1=H.
$$
Since the homology intersection form $\omega \colon H \times H \to \Z$ of the oriented surface~$\Sigma$ is skew-symmetric and non-degenerate,
it defines a duality $H \cong H^*$  by $x\mapsto \omega(x,-)$: hence we regard $\omega$ as an element of 
$$
\Lambda^2 H^* \cong \Lambda^2 H \cong  \frakL_2.
$$
Then, a \emph{symplectic expansion} of $\pi$ is defined as a map
$$
\theta\colon \pi \longrightarrow \widehat{T}(H)
$$
with values in the degree-completion of $T(H)$,
which is multiplicative, maps the boundary element $\zeta$ to $\exp(-\omega)$  and satisfies
$$
 \theta(x) = \underbrace{1 + [x]+ (\deg\geq 2)}_{\hbox{\small group-like}}
$$
for all $x\in \pi$ with homology class  $[x]\in H$.
 The condition that $\theta(x)$ is group-like is equivalent to requiring that $\log \theta(x)$ lies in the degree-completion of~$\mathfrak{L}$.
Symplectic expansions  are easily proved to exist \cite[Lemma 2.16]{Ma12},
and some instances can be constructed in an explicit combinatorial way  \cite{Ku12}.
\emph{In this section, we  fix a symplectic expansion~$\theta$.} 

This map $\theta\colon \pi \to \widehat{T}(H)$ can be extended, by linearity and continuity, 
to a complete Hopf algebra isomorphism $\theta\colon \widehat{\Q \pi} \to \widehat{T}(H)$, 
which restricts to a complete Lie algebra  isomorphism $\theta\colon \frakM(\pi) \to \widehat{\frakL}$. 
Hence, for any closed curve $\gamma$ in $\Sigma$,
the generalized Dehn twist $t_\gamma$ can be equivalently considered~as 
$$
t_\gamma^\theta :=   \theta  \circ t_\gamma \circ  \theta^{-1}  \in \Aut(\widehat{\frakL}).
$$

\subsection{The Lie algebra of symplectic derivations}

Let $\gamma$ be a closed  curve in $\Sigma$ and consider now the logarithm
$$
\log(t_\gamma^\theta) =  \theta  \circ \log(t_\gamma) \circ  \theta^{-1}
=  \theta  \circ \sigma(L(\gamma)) \circ  \theta^{-1} 
$$
which is a derivation of $\widehat{\frakL}$  vanishing on $\omega \in \frakL_2$.
The set consisting of all derivations of $\frakL$ 
that  vanish on $\omega \in \frakL_2$ is stable under the usual Lie bracket of derivations. 
It is called the \emph{Lie algebra of symplectic derivations} and denoted by  $$\Der_\omega( \frakL).$$
This Lie algebra  can also be regarded as a subspace of $\Hom(H,\frakL)$
(since any derivation is determined by its restriction to $H$)
and, therefore, it can be regarded as a subspace  of  $H \otimes \frakL$ (using the  duality $H \cong H^*$).
From this viewpoint, 
$\Der_\omega(\frakL)$ is the kernel of the Lie bracket $H \otimes \frakL \to \frakL_{\geq 2}$.
Note that $\Der_\omega(\frakL)$ is a graded Lie algebra, where a derivation is homogeneous of degree $k$ 
if and only if it increases the degree of $\frakL$ by $k$.

It is well-known that the above-mentioned graded Lie algebras have diagrammatic descriptions, which we now recall.
First of all, note that a linear combination~$T$
of planar binary rooted trees  with $j$ leaves colored by $H$ defines  an element $\operatorname{comm}(T)\in \frakL_j$. For instance:\\[0.1cm]
$$
\operatorname{comm}\Bigg(\begin{array}{c}
\labellist \small \hair 2pt
\pinlabel {$h_1$} [b] at 2 162
\pinlabel {$h_2$} [b] at 111 162
\pinlabel {$h_3$} [b] at 167 162
\pinlabel {$h_4$} [b] at 222 162
\pinlabel {root} [t] at 108 5
\endlabellist
\includegraphics[scale=0.3]{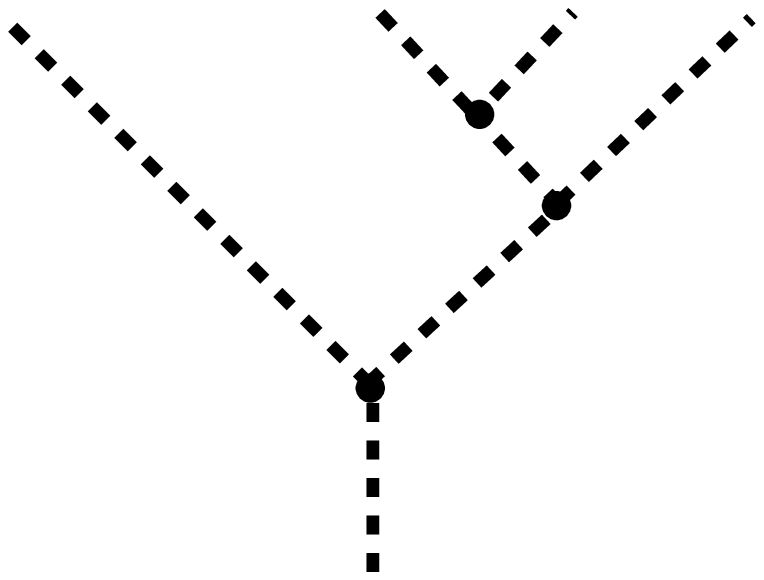} 
\end{array}\Bigg) =  [h_1,[[h_2,h_3],h_4]]
\quad \quad  (\hbox{with } h_1,\dots,h_4 \in H).
$$
A \emph{Jacobi diagram} is a graph with only univalent  and trivalent vertices, the latter being assumed to be  oriented
(i.e$.$ half-edges are cyclically ordered around each trivalent vertex); it is said  \emph{$H$-colored} if
all  its univalent vertices are colored by~$H$. Its \emph{degree} is the number of its trivalent vertices.
In the sequel, we always assume that Jacobi diagrams are finite, tree-shaped and connected. 
For example, here is an $H$-colored Jacobi diagram of degree~$3$
(where, by convention, vertex orientations are given by the trigonometric orientation of the plan):\\[0.cm]
$$
{\labellist \small \hair 2pt
\pinlabel {$h_1$} [tr] at 0 0
\pinlabel {$h_2$} [br] at 0 59
\pinlabel {$h_3$} [b] at 69 76
\pinlabel {$h_4$} [bl] at 138 63
\pinlabel {$h_5$} [tl] at 135 0
\endlabellist}
\includegraphics[scale=0.3]{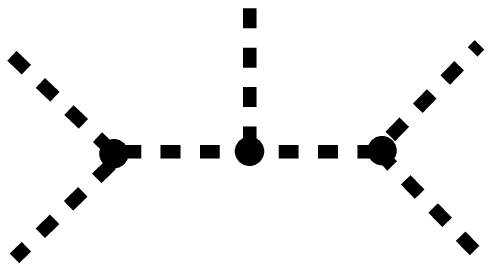}
\quad \quad \quad \quad  (\hbox{with } h_1,\dots,h_5 \in H)
$$
\vspace{-0.3cm}

\noindent
Let
$$
\calT = \bigoplus_{d=0}^{\infty} \calT_d
$$ 
be the graded $\Q$-vector space of $H$-colored Jacobi diagrams modulo the \emph{AS}, 
\emph{IHX} and \emph{multilinearity} relations,
which are the local relations shown below: \\[0.2cm]

\begin{center}
\labellist \small \hair 2pt
\pinlabel {AS} [t] at 102 -5
\pinlabel {IHX} [t] at 543 -5
\pinlabel {multilinearity} [t] at 1036 -5
\pinlabel {$= \ -$}  at 102 46
\pinlabel {$-$} at 484 46
\pinlabel {$+$} at 606 46
\pinlabel {$=0$} at 721 46 
\pinlabel {$+$} at 1106 46
\pinlabel {$=$} at 961 46
\pinlabel{$h_1+h_2$} [b] at 881 89
\pinlabel{$h_1$} [b] at 1042 89
\pinlabel{$h_2$} [b] at 1170 89
\endlabellist
\centering
\includegraphics[scale=0.3]{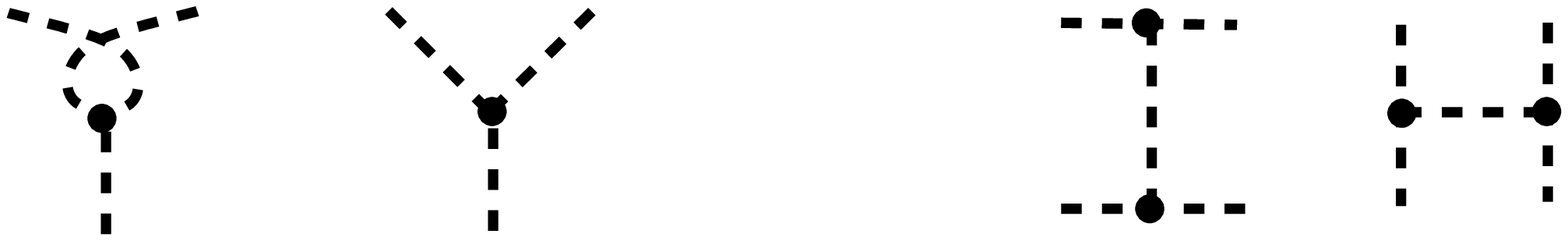}
\end{center}
\vspace{0.5cm}
There is a graded linear map
$
\Xi \colon \calT  \rightarrow 
\Der_\omega( \frakL) \subset H\otimes \frakL
$
which is defined, for any $H$-colored Jacobi diagram $T$, by
$$
\Xi(T) := \sum_{v} \operatorname{col}(v) \otimes \operatorname{comm}(T_v).
$$
Here the sum is over all univalent vertices $v$ of $T$, $ \operatorname{col}(v) $ denotes the element of $H$ carried by $v$
and $T_v$ is the  tree $T$ rooted at $v$.
The  map $\Xi$ is known to be an isomorphism
(see, for instance, \cite{HP}).

\begin{example}
In degree $0$, $\Xi$ maps any diagram of the form $h\hbox{\textbf{\,-\! -\! -}\,} k$ 
to $h \otimes k + k \otimes h$.
Indeed, as a subspace of $\Hom(H,\frakL)$, 
the degree $0$ part of $\Der_\omega( \frakL)$ corresponds to 
$$
\big\{u \in \Hom(H,H) \mid 
(u \otimes \operatorname{id} + \operatorname{id} \otimes u)(\omega) =0 \in \Lambda^2 H \subset H^{\otimes 2} \big\};
$$
therefore, as a subspace of $H \otimes \frakL$, this  corresponds to the kernel of the canonical projection $H \otimes H \to \Lambda^2H$,
i.e$.$ to the symmetric part  of $H\otimes H$.
\end{example}

The isomorphism $\Xi^{-1}$ transports the Lie bracket of derivations to the following Lie bracket in~$\calT$:
for any $H$-colored Jacobi diagrams $T$ and $T'$, 
\begin{equation}
\label{eq:Lie}
[T,T'] = \sum_{v,v'} \omega(v,v')\cdot T_v \hbox{\textbf{\,-\! -\! -}\,}  T_{v'}
\end{equation}
where the sum is over all univalent vertices $v$ and $v'$ of $T$ and $T'$, respectively,
and $T_v \hbox{\textbf{\,-\! -\! -}\,}  T_{v'}$ is obtained by gluing ``root-to-root'' $T_v$ and $T_{v'}$.

\subsection{A diagrammatic formula for generalized Dehn twists}

Recall that, in this section, we have fixed a symplectic expansion $\theta$ of $\pi$.

\begin{thm}[\cite{KM19}] \label{thm:diagram}
For any  closed curve $\gamma$ in $\Sigma$, we have
\begin{equation} \label{eq:symmetric}
\Xi^{-1}\big(\log(t_\gamma^\theta)\big)= \frac{1}{2} \cdot\log \theta ([\gamma]) \hbox{\textbf{\,-\! -\! -}\,}   \log \theta ([\gamma])
\end{equation}
meaning that $\Xi^{-1}\big(\log(t_\gamma^\theta)\big)$  is half the series of Jacobi diagrams
obtained by  gluing ``root-to-root''  two copies of the series of planar  rooted trees $\log \theta ([\gamma])$.
\end{thm}

Formula \eqref{eq:symmetric} is shown in \cite[Proof of Theorem 5.1]{KM19}  for a null-homologous closed curve $\gamma$, 
but the proof extends verbatim to any closed curve.
This proof of  \eqref{eq:symmetric} is based on a formal description of the loop operation $\sigma$ using the symplectic expansion $\theta$.
(See \cite[Theorem 1.2.2]{KK14} and also \cite[~\S 10]{MT13}.) 

We now mention some consequences of \eqref{eq:symmetric} as for the ``leading terms'' of generalized Dehn twists.
Let $\frakL^\Z$ be the Lie ring freely generated by $H_1(\Sigma;\Z)$ and regard $\frakL^\Z$ as a subset of~$\frakL$.
Let ${k\geq 1}$ be an integer.
 Recall from Section \ref{subsec:anq} that a   closed curve~$\gamma$ in $\Sigma$ is  \emph{of nilpotency class $\geq k$} if, 
when it is oriented and connected to  $*$  by an arc (in an arbitrary way),
its homotopy class $[\gamma] \in \pi$  belongs to $\Gamma_k \pi$.
Then,
$$
\theta([\gamma])= 1 + \{\gamma\}_{k} + (\deg\geq k+1)
$$
where $\{\gamma\}_{k}\in \frakL_k\subset H^{\otimes k}$ corresponds to the class of $\gamma$ modulo $\Gamma_{k+1} \pi$
through the canonical isomorphism $ \frakL_k^\Z \cong \Gamma_k \pi/ \Gamma_{k+1} \pi$.

\begin{prop} \label{prop:one_two} 
\begin{enumerate}
\item[(i)]
Let $\gamma$ be a closed curve  of nilpotency class~$\geq k$ in~$\Sigma$. Then
$$ \qquad \qquad 
\Xi^{-1}\big(\log(t_\gamma^\theta)\big) = 
\frac{1}{2}\{\gamma\}_{k} \hbox{\textbf{\,-\! -\! -}\,} \{\gamma\}_{k} +(\deg \geq 2k-1).
$$
\item[(ii)]
Let $\gamma_-,\gamma_+$ be closed curves in $\Sigma$ of nilpotency class $\geq k$ such that $\{\gamma_-\}_{k}=\{\gamma_+\}_{k}$.
Then 
$$
\qquad \quad  \Xi^{-1}\big(\log\big( (t_{\gamma_-}^\theta)^{-1}\, t_{\gamma_+}^\theta\big) \big) = 
\{\gamma_\pm\}_{k} \hbox{\textbf{\,-\! -\! -}\,} \{\gamma_+ \gamma_-^{-1}\}_{k+1} +(\deg \geq 2k)
$$
where each curve $\gamma_\pm$ is oriented and connected to $*$ 
by an arc (in an arbitrary way)
to define $\gamma_+ \gamma_-^{-1} \in \Gamma_{k+1} \pi$.
\end{enumerate}
\end{prop}

\begin{proof}[About the proof]
Statement (i) is a direct consequence of formula \eqref{eq:symmetric}.
As for statement~(ii), note that the product $(t_{\gamma_-}^\theta)^{-1}\, t_{\gamma_+}^\theta$ 
does have a logarithm since it acts trivially on $I/I^2 \cong H$; 
then (ii) is a less immediate consequence of \eqref{eq:symmetric} using the Baker--Campbell--Hausdorff formula
and the Lie bracket \eqref{eq:Lie} in $\calT$. 
\end{proof}

\subsection{Generation of the generalized mapping class group}

Recall from  Section \ref{sec:realize} that the mapping class group $\calM$ can be regarded
 as a subgroup of the generalized mapping class group $\widehat{\calM}$.
Think of the latter  in the form~\eqref{eq:form2}.
Then, for any $k\geq 0$, let $\widehat{\calM}[k]$ be the subgroup of $\widehat{\calM}$ 
that acts trivially on $\widehat{\Q\pi}/\widehat{I^{k+1}}$: the sequence of nested subgroups
$$
\widehat{\calM}=\widehat{\calM}[0] \supset \widehat{\calM}[1] \supset \cdots
\supset \widehat{\calM}[k] \supset \widehat{\calM}[k+1] \supset \cdots
$$
is the analogue for $\widehat{\calM}$ of the {\it Johnson filtration} of the mapping class group~${\calM}$,
which has been introduced by Johnson in \cite{Joh83} and studied by Morita in \cite{Mo93}.

It can be verified that the Johnson filtration is ``strongly'' central in the sense that 
$\big[\widehat{\calM}[j], \widehat{\calM}[k]\big] \subset \widehat{\calM}[j+k]$ for all  $j,k\geq 0$. 
In particular, it consists of normal subgroups of $\widehat{\calM}$. Furthermore, 
the fact that the filtration  $\{ \widehat{I^k} \}_k$ of $\widehat{\Q\pi}$ has a trivial intersection easily implies that
$$
 \bigcap_{k=0}^{\infty} \widehat{\calM}[k]=\{1\}.
$$
In the sequel, we consider the Hausdorff topology on the group $\widehat{\calM}$ defined by the Johnson filtration.

Any closed curve $\gamma$ in $\Sigma$ defines a one-parameter family $\{ t_{r,\gamma} \}_{r\in \Q}$
in  $\widehat{\calM}$  by setting
\begin{equation} \label{eq:trC}
t_{r,\gamma} := \exp\Big(r \sigma\big((\log \gamma)^2\big) \Big) .  
\end{equation}
Note that $t_{1/2,\gamma}$ is  the generalized Dehn twist $t_\gamma$ along $\gamma$:
hence $\{ t_{r,\gamma} \}_{r\in \Q}$ consists of all rational roots of $t_\gamma$. 
The following result, which seems to be new,  
is an algebraic analogue of the fact that $\calM$ is generated by usual Dehn twists.

\begin{thm}
The group $\widehat{\calM}$ is topologically generated by the elements $t_{r,\gamma}$ for all $r\in \Q$
and any closed curve $\gamma$ in $\Sigma$.
\end{thm}

\begin{proof}
Let $\mathcal{TW}^\Q$ be the subgroup of $\widehat{\calM}$  generated by the $t_{r,\gamma}$ for all $r\in \Q$
and any $\gamma \subset \Sigma$. To prove that $\mathcal{TW}^\Q$  is dense in $\widehat{\calM}$, 
we shall prove  the following statement for any  $u \in \widehat{\calM}$
and any integer $n\geq 0$:
\begin{quote}
{\it there exist elements $u_k \in {\mathcal{TW}^\Q \cap \widehat{\calM}[k]}$ for $k\in\{0,\dots,n\}$ such that}
$$ 
(\mathcal{H}_n) \qquad u \equiv \prod_{k=0}^n u_k \mod \widehat{\calM}[n+1].
$$
\end{quote}
 
The proof is by induction on $n$.
We firstly prove $(\mathcal{H}_0)$. Consider the automorphism $\tilde u_0$ of $\widehat{I^1}/\widehat{I^2} \cong H$ 
induced by $u$: it preserves $\omega$ since $u$ preserves the homotopy intersection form $\eta$. 
Therefore, $\tilde u_0$ is a finite product of symplectic transvections. 
Observe the following general fact about a symplectic transvection
\begin{equation} \label{eq:transvection}
 H \longrightarrow H, \ h \longmapsto h+ \omega(c,h) \cdot c,
\end{equation}
that is defined by an element $c\in H$: one can find a closed curve $C$ in $\Sigma$ and an integer $m\geq 1$ 
such that (for some orientation of $C$) we have $c=[C]/m \in H$; then $t_{1/2m,C}$ induces \eqref{eq:transvection} 
at the level of $\widehat{I^1}/\widehat{I^2} \cong H$. We deduce that
there exists $u_0\in \mathcal{TW}^\Q$ such that $u_0$ induces $\tilde u_0 \in \Aut(H)$:
therefore $u\equiv u_0 \mod \widehat{\calM}[1]$.

Assuming $(\mathcal{H}_{n-1})$ for $n\geq 1$, we shall now prove $(\mathcal{H}_{n})$. 
Set $v:= \prod_{k=0}^{n-1} u_k$.
Choose a symplectic expansion $\theta$ of $\pi$.  Since $v^{-1}u$ belongs to  $\widehat{\calM}[n]$, the derivation
$$
\delta:=\log(\theta\circ (v^{-1}u)\circ \theta^{-1}) \ \in \Der_{\omega}(\widehat\frakL)
$$
increases degrees by at least $n$. 
Therefore, the series of Jacobi diagrams $\Xi^{-1}(\delta)\in \widehat{\calT}$
starts in degree $n$.

Assume that $n=2m$ is even. 
As any element of the vector space $\calT_n$, the leading term 
of $\Xi^{-1}(\delta)$ can be written in the form
$$
\sum_{i=1}^p r_i\cdot  {x}_i \hbox{\textbf{\,-\! -\! -}\,} {x_i} \ \in \calT_n
$$
where $r_i\in \Q$ and ${x}_i \in \frakL_{m+1}^\Z$. For every $i\in \{1,\dots,p\}$, we choose
a closed curve $C_i$ in $\Sigma$ of nilpotency class $\geq m+1$ that represents $x_i \in \frakL_{m+1}^\Z \cong \Gamma_{m+1} \pi/\Gamma_{m+2} \pi$.
Proposition \ref{prop:one_two}.(i) has the following generalization:
for any closed curve $\gamma$ in $\Sigma$ of nilpotency class $\geq k$ and for any $r\in \Q$, we have
$$
\Xi^{-1}\big(\log( \theta \circ t_{r,\gamma} \circ \theta^{-1} )\big) = 
r \cdot \{\gamma\}_{k} \hbox{\textbf{\,-\! -\! -}\,} \{\gamma\}_{k} +(\deg \geq 2k-1)
$$
and, in particular, $t_{r,\gamma}$ belongs to $\widehat{\calM}[2k-2]$.
Therefore,
\begin{eqnarray*}
&&\Xi^{-1}\log\Big(\theta \circ \Big( \prod_{i=1}^p t_{r_i,C_i} \Big) \circ \theta^{-1} \Big) \\
&= & \sum_{i=1}^p\, \Xi^{-1}\log\big(\theta \circ t_{r_i,C_i} \circ \theta^{-1} \big) + (\deg \geq n+1)\\
& =& \sum_{i=1}^p r_i\cdot  {x}_i \hbox{\textbf{\,-\! -\! -}\,} {x_i}+ (\deg \geq n+1)
\end{eqnarray*}
and then $v^{-1}u \equiv u_n \mod \widehat{\calM}[n+1]$ where
$$
u_n  := \prod_{i=1}^p t_{r_i,C_i} \in \mathcal{TW}^\Q \cap \widehat{\calM}[n].
$$

Assume now that $n=2m+1$ is odd. As any element of the vector space~$\calT_n$, the leading term
of $\Xi^{-1}(\delta)$ can be written in the form
$$
\sum_{i=1}^p 2r_i\cdot  {x}_i \hbox{\textbf{\,-\! -\! -}\,} {y_i} \ \in \calT_n
$$
where $r_i\in \Q$, ${x}_i \in \frakL_{m+1}^\Z$ and  ${y}_i \in \frakL_{m+2}^\Z$ . For every $i\in \{1,\dots,p\}$, we choose
two closed curves $C_i$ and $D_i$ in $\Sigma$ of nilpotency class $\geq m+1$
that both represent $x_i \in \frakL_{m+1}^\Z \cong \Gamma_{m+1} \pi/\Gamma_{m+2} \pi$
 and such that $D_i C_i^{-1}$ represents~$y_i$.
Then, by a similar argument to the case where $n$ is even and using a generalized version of Proposition \ref{prop:one_two}.(ii), we obtain
\[
\Xi^{-1}\log\Big(\theta \circ \Big( \prod_{i=1}^p \big(t_{r_i,C_i}\big)^{-1}\, t_{r_i,D_i}\Big) \circ \theta^{-1} \Big)
=
\sum_{i=1}^p 2r_i\cdot  {x}_i \hbox{\textbf{\,-\! -\! -}\,} {y_i} +(\deg \geq n+1),
\]
and then $v^{-1}u \equiv u_n \mod \widehat{\calM}[n+1]$ where
$$
u_n  := \prod_{i=1}^p (t_{r_i,C_i})^{-1}\, t_{r_i,D_i} \in \mathcal{TW}^\Q \cap \widehat{\calM}[n].
$$
\end{proof}

\section{Algebraic formulation of generalized Dehn twists}
\label{sec:Fox}

We review from \cite{MT13} a group-algebraic framework for generalized Dehn twists. 
In fact, (at least) part of this framework can be extended to Hopf algebras~\cite{MT14}.
Hence we consider a Hopf algebra $A$: let $\Delta\colon A \to A \otimes A$ be the coproduct,  $\varepsilon\colon A \to \Q$ the counit
and $S\colon A \to A$ the antipode. We assume that $A$ is involutive, i.e$.$ $S^2= \operatorname{id}_A$,
and we denote by $I:=\ker \varepsilon$ the augmentation ideal of $A$.

A \emph{Fox pairing} in $A$ is a $\Q$-bilinear map $\eta\colon A \times A \to A$ such that
\begin{equation} \label{eq:Fox}
\eta(ab,c) = a\,\eta(b,c) + \varepsilon(b)\, \eta(a,c), \quad 
\eta(a,bc) =  \eta(a,b)\,c + \varepsilon(b)\, \eta(a,c)
\end{equation}
for all $a,b,c\in A$. 
It follows  that 
\begin{equation} \label{eq:II}
\eta(I^m,I^n)\subset I^{m+n-2}
\end{equation} for any integers $m,n\geq 1$
(with the convention that $I^{-2}=I^{-1}=I^0=A$).

A Fox pairing $\eta$ induces two other bilinear forms.
First, the \emph{homological form} induced by $\eta$ is the bilinear map
$$
(-\bullet_\eta -)\colon I/I^2 \times I/I^2 \longrightarrow \Q
$$ 
defined by $\{a\} \bullet_\eta \{b\} = \varepsilon \eta(a,b)$.
Second, the \emph{derived form} of $\eta$ is the bilinear map $\sigma_{\eta}\colon A \times A \to A$ defined by
$$
\sigma_{\eta} (a,b) := \sum_{(a)}\sum_{(b)}\sum_{ (\eta(a'',b''))}
 b'\, S(\eta(a'',b'')')\, a'\, \eta(a'',b'')''
$$ 
for any $a,b\in A$. (Here we have used Sweedler's notation 
$\Delta(c) = \sum_{(c)} c'\otimes c''$
to write down the coproduct of any element $c\in A$.) 
It can be checked that $\sigma_\eta$ induces a map
$$
\sigma_{\eta}\colon A \longrightarrow \operatorname{Der}(A,A), \ a \longmapsto \sigma_{\eta} (a,-)
$$
with values in the Lie algebra of derivations of $A$ and
which vanishes on commutators of $A$.

We now recall the most fundamental example of Fox pairings that arises from topology,
and which we have already mentioned in previous sections.

\begin{example} \label{ex:hif}
Assume that $A:=\Q\pi$ is the algebra of a group $\pi$.
Then, the notion of Fox pairing  appears under different names (but equivalent forms) in Papakyriakopoulos's work  \cite{Papa75} 
and in Turaev's paper~\cite{Tur78}.
These works mostly considered the fundamental group $\pi$ of an oriented surface with non-empty boundary,
and the \emph{homotopy intersection form}
$$
\eta\colon \Q \pi \times \Q \pi \longrightarrow \Q\pi
$$ 
that we have reviewed in \eqref{eq:eta}.
It turns out that $\eta$ is a Fox pairing,
whose homological form $\bullet_\eta$  in 
$$I/I^2 \cong H_1(\pi;\Q) \cong H_1(\Sigma;\Q)$$ 
is the intersection form $\omega$,
and whose derived form $\sigma=\sigma_{\eta}$ is the action~\eqref{eq:sigma}.
This Fox pairing $\eta$ has some remarkable properties: it is ``skew-symmetric'' in  some weak sense,
and the corresponding ``double bracket'' is ``quasi-Poisson'' in the sense of \cite{VdB08}. (See \cite{MT13,MT14}.)
These properties generalize the fact that the Goldman bracket (which is induced by $\sigma$) is a Lie bracket~\cite{Gol86}.
\end{example}

Here are two other examples of Fox pairings that still arise from topology,
and which need to consider Hopf algebras in a broader sense.

\begin{example}
Assume that, instead of a Hopf algebra $A$, we are given a \emph{complete Hopf algebra} $\widehat A$
(i.e$.$ a ``Hopf  monoid'' in the symmetric monoidal category of complete $\Q$-vector spaces:
see \cite[Appendix A]{Qui69} for the definition).
Then the notion of ``Fox pairing'' extends verbatim to $\widehat A$.
Let $\pi$ be a group, and  let 
$$
\widehat A:= \varprojlim_{k} A/I^k
$$
be the $I$-adic completion of the group algebra $A:=\Q \pi$. 
By \eqref{eq:II}, any Fox pairing in $A$ induces by continuity a Fox pairing in $\widehat A$, 
but there also exist Fox pairings in $\widehat A$ that (a priori) do not arise from $A$. 
For  the commutator subgroup $\pi$ of a knot group in the standard $3$-sphere,
Turaev constructed such a Fox pairing in $\widehat A$ using the homotopy intersection form of a Seifert surface of the knot:
see \cite[Supplement 3]{Tur78} and \cite[\S 5]{Tur81} for further details.
\end{example}

\begin{example}
Assume that $A$ is a \emph{graded Hopf algebra}
(i.e$.$ a ``Hopf  monoid'' in the symmetric monoidal category of graded $\Q$-vector spaces):
then the theory of Fox pairings can be easily adapted to this setting.
Let $M$ be a smooth oriented manifold of dimension $d>2$ with non-empty boundary,
and let $A:= H_*(\Omega(M,\star);\Q)$ be the homology of its loop space $\Omega(M,\star)$ based at a point $\star \in \partial M$.
Then, by intersecting families of loops in $M$, one  defines in the graded Hopf algebra $A$ 
a  Fox pairing of degree $2-d$: it is ``skew-symmetric'' in some sense, 
and the corresponding ``double bracket'' is Gerstenhaber of degree $2-d$ in the sense of \cite{VdB08}. 
We refer to \cite{MT17} for the construction of this intersection double bracket which uses the ideas of string topology, 
and to \cite[Appendix B]{MT18} for the correspondence between Fox pairings and double brackets.
\end{example}

We now restrict ourselves to the case where the Hopf algebra $A=\Q\pi$ is a group algebra.
Let $\widehat A= \widehat{I^{0}} \supset \widehat{I^{1}} \supset \widehat{I^{2}} \supset \cdots$ be the canonical filtration
on the $I$-adic completion $\widehat A$ of $A$.
Let $\eta$ be a Fox pairing in $A$ and let $C$ be a group-like element of $\widehat A$ such that
\begin{equation} \label{eq:isotropic}
\{C-1\} \bullet_\eta \{C-1\}=0
\end{equation}
where $\{C-1\} \in \widehat{I}/\widehat{I^2} \cong I/I^2$ denotes the class of $C-1$.
It follows from~\eqref{eq:II} that $\sigma_\eta$ induces a map
$\sigma_{\eta}: \widehat{I^2} \to\operatorname{Der}(\widehat A,\widehat A)$
with values in the Lie algebra of filtration-preserving derivations of $\widehat A$.
Furthermore, the hypothesis~\eqref{eq:isotropic} implies that 
$$
 \sigma_{\eta}\big(\log(C)^2\big) \in  \operatorname{Der}(\widehat A, \widehat A)
$$
is ``weakly nilpotent'' in the sense that, for any integer $m\geq 1$, 
it maps $\widehat A$ to~$\widehat{I^m}$ 
after sufficiently enough iterations  \cite[Lemma 4.1]{MT13}.
Hence, we can consider the \emph{twist}~map
\begin{equation} \label{eq:twist}
t_{r,C} := \exp\Big(r\sigma_{\eta}\big(\log(C)^2\big)  \Big)
\end{equation}
for any scalar $r\in \Q$ and any group-like element $C\in  \widehat A$. 
It turns out that $t_{r,C}$ is an automorphism of the complete Hopf algebra $\widehat A$ \cite[Theorem 5.1]{MT13}:
thus, $t_{r,C}$ can be equivalently regarded as an automorphism of the Malcev completion $\widehat \pi$.

\begin{example}
If $\pi$ is the fundamental group of an oriented surface with non-empty boundary
and  $\eta$ is the homotopy intersection form \eqref{eq:eta},
then $t_{r,C}$ is the generalized Dehn twist  \eqref{eq:trC}.
\end{example}

\begin{rem}
As it has been overviewed above, the theory  of Fox pairings also exists for (graded) Hopf algebras \cite{MT14,MT18}.
Thus, it seems possible to define twists in the Hopf-algebraic setting as well. 
In particular, it would be interesting to understand the topological origin of the twists
that arise in  higher dimension  from the intersection operation of \cite{MT13}.
\end{rem}

\section{Generalized Dehn twists and homology cylinders}
\label{sec:hc}

There is another field of low-dimensional topology where automorphisms of the Malcev completion $\hpi$  naturally appear:
this is the study of homology cylinders.
The latter has started with the work of Garoufalidis and Levine~\cite{GL05};
see also \cite[\S 8.5]{Habiro00} and \cite{Habegger00}.

A \emph{homology cobordism} of $\Sigma$ is a pair $(M,m)$
consisting of a compact oriented 3-manifold   $M$ and  a diffeomorphism
$m\colon {\partial (\Sigma \times [-1,+1])} \to \partial M$ preserving the orientations,
such that the inclusion maps $m_{\pm}\colon \Sigma \to M$ defined by $m_\pm (x) := m(x,\pm 1)$ give isomorphisms in integral  homology.
Thus $M$ is a cobordism (with corners) whose ``input'' surface $\partial_+ M :=m_+(\Sigma)$
and ``output'' surface  $\partial_- M :=m_-(\Sigma)$ are both parametrized by the ``standard'' surface $\Sigma$:
$$
\labellist
\scriptsize\hair 2pt
\pinlabel {$\partial_+ M$} [r] at 1 186
 \pinlabel {$\partial_- M$} [r] at 0 75
 \pinlabel {$M$}  at 92 128
 \pinlabel {$m_+$} [r] at 91 222
 \pinlabel {$m_-$} [r] at 90 38
 \pinlabel {$\Sigma$} [r] at 2 255
 \pinlabel {$\Sigma$} [r] at 3 4
\endlabellist
\centering
\includegraphics[scale=0.29]{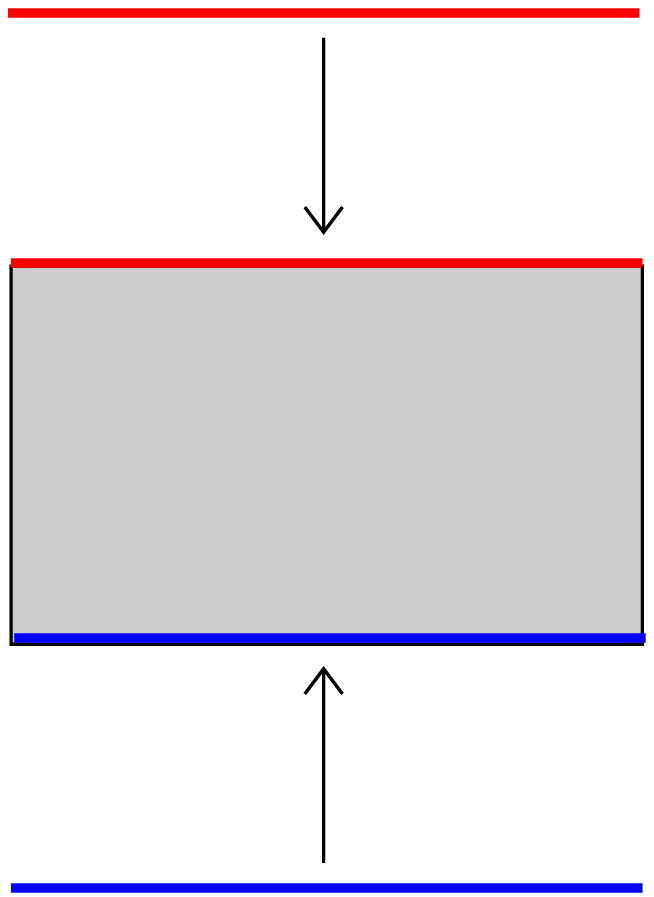}
$$
Homology cobordisms can be multiplied in the usual fashion by gluing ``output'' to ``input'' surfaces. Hence the set 
$$
\calC := \big\{  \hbox{\small homology cobordisms of $\Sigma$}  \big\}  / \hbox{\small diffeomorphism}
$$ 
is a monoid, whose neutral element is the usual cylinder $U:= \Sigma \times [-1,+1]$
(taking the identity as boundary parametrization $u\colon \partial(\Sigma \times [-1,+1]) \to \partial U$).
It is well-known that the group of invertible elements of $\calC$ is   the mapping class group $\calM$,
which is viewed as a subset of $\calC$ by the ``mapping cylinder'' construction.
(See, for instance, \cite[\S 2]{HM} for a survey.)

Let $(M,m)\in \calC$ and  $k\ge 1$.
It follows from a result of Stallings~\cite{Sta65} that
 the maps $m_+$ and $m_-$ induce isomorphisms $\pi/\Gamma_{k+1} \pi \to \pi_1(M)/\Gamma_{k+1} \pi_1(M)$.
Thus there is a monoid homomorphism
\[
\rho_k \colon \calC \longrightarrow \Aut(\pi/\Gamma_{k+1}\pi),
\quad (M,m)  \longmapsto \big( {(m_-)}^{-1} \circ m_+\big).
\]
Since $\hpi$ is the projective limit of the Malcev completions of $\pi/\Gamma_{k+1} \pi$ as ${k\to \infty}$,
we obtain  a homomorphism $\hat\rho \colon \calC \to \Aut(\hpi)$, which  generalizes the Dehn--Nielsen representation
of the mapping class group:
$$
\xymatrix{
\calM \ar[r]^-{\rho} \ar@{^{(}->}[d]_-{\stackrel{\hbox{\scriptsize mapping }}{\hbox{\scriptsize  cylinder}}}  & \Aut_\partial(\pi) \ar[d]^-{\ \widehat{\centerdot}}& \\
\calC\, \ar[r]^-{\hat \rho} & \Aut_\partial(\hpi) &\!\!\!\!\!\!\!\!\!\!\!\!\!\!\!\! 
= \widehat{\calM}
}
$$

We now explain how the representation $\hat \rho$ is related to generalized Dehn twists.
Let $\gamma\subset \Sigma$ be a closed curve. A \emph{resolution} of $\gamma$ is a knot $K$ in the usual cylinder $U$
which projects onto $\gamma$ in $\Sigma = \Sigma \times \{+1\}$. 

\begin{thm}[\cite{KM19}] \label{thm:curve_knot}
Let $\gamma\subset \Sigma$ whose class $[\gamma] \in \pi$ (for some arbitrary choices of orientation and connecting arc to $*$) 
 belongs to $\Gamma_{k} \pi$ for some $k\geq 2$. Then, for any resolution $K$ of $\gamma$, we have a commutative diagram
$$
\xymatrix{
{\pi}/{\Gamma_{2k+1} \pi} \ar[d]_-{\rho_{2k}(U_K)}  \ar[r] & \hpi / \hpi_{2k+1} \ar[d]^-{(t_\gamma)^{\varepsilon}} \\
{\pi}/{\Gamma_{2k+1} \pi}   \ar[r] & \hpi / \hpi_{2k+1} 
}
$$
where $U_K \in \calC$ is obtained from $U$ by surgery along $K$ taking $\varepsilon \in \{-1,+1\}$ as ``framing number''.
\end{thm}

The proof of Theorem \ref{thm:curve_knot} is based on some explicit formulas for 
$\rho_{2k}(U_K)\in \Aut({\pi}/{\Gamma_{2k+1} \pi})$ and
$t_\gamma \in \Aut(\hpi / \hpi_{2k+1} )$:
the one for $t_\gamma$ 
is obtained by commutator calculus and refines Proposition \ref{prop:quotient};
the one for $\rho_{2k}(U_K)$  is proved by surgery calculus.
Both formulas are expressed in terms of the class of $[\gamma]$ modulo $\Gamma_{k+2}\pi$
using the homotopy intersection form $\eta$, and they turn out to be the same.

\begin{rem}
Of course, Theorem \ref{thm:curve_knot} reduces to a  basic fact if $\gamma$ is simple:  in that case,
there is only one way to ``resolve'' $\gamma$ to $K$ and, by Lickorish's trick \cite[Proof of Theorem 2]{Lic62},
the cobordism $U_K$ is merely the mapping cylinder of the usual Dehn twist $t_\gamma$.
\end{rem}

Considering arbitrary closed curves $\gamma \subset \Sigma$, one  derives from
Theorem~\ref{thm:curve_knot} a few applications to the study of homology cylinders.
A \emph{homology cylinder} is an $M\in \calC$ such that
$m_+=m_-\colon H_1(\Sigma) \to H_1(M)$ or, equivalently, such that $\rho_1(M)$ is the identity of $\pi/\Gamma_2\pi$. 
Examples of homology cylinders include all the cobordims $U_K$ of the type considered in Theorem~\ref{thm:curve_knot}.
Homology cylinders constitute a submonoid 
$
\calI\calC \subset \calC
$
on which it is important to compute the generalized Dehn--Nielsen representation
$$
\hat\rho : \calI\calC \longrightarrow \operatorname{IAut}_\partial(\hpi).
$$
Here $\operatorname{IAut}_\partial(\hpi)$ is the subgroup of $\operatorname{Aut}_\partial(\hpi)$
that gives the identity on the associated graded of $\hpi$. 
Through the correspondence \eqref{eq:correspondence},
$\operatorname{IAut}_\partial(\hpi)$ is identified with the group  $\operatorname{IAut}_\partial( \frakM(\pi) )$
of automorphisms of the Malcev Lie algebra $\frakM(\pi)$ that induce the identity on the associated graded and fix  $\log \zeta$. 
Furthermore, for  any symplectic expansion $\theta$ of $\pi$, 
the latter group is identified with the Lie algebra $\operatorname{Der}_\omega(\widehat \frakL)$ through
the map  $\psi \mapsto \log( \theta \circ \psi \circ \theta^{-1})$.
Thus, we can consider the composition
$$
\xymatrix{
\calI\calC \ar[r]^-{\hat \rho}  \ar@{-->}@/_2pc/[rr]^-{\varrho^\theta} & \operatorname{IAut}_\partial(\hpi)
\cong \operatorname{IAut}_\partial( \frakM(\pi) ) \cong
\operatorname{Der}_\omega(\widehat \frakL)  \ar[r]^-{\Xi^{-1}} & \widehat{\calT}}
$$

\vspace{0.1cm}
\noindent
as a diagrammatic version of the generalized Dehn--Nielsen representation.

\begin{rem}
For some instances of symplectic expansions $\theta$, the composition $\varrho^\theta$ gives
the ``tree-reduction'' of the LMO homomorphism \cite{Ma12}, 
which is a fundamental invariant of homology cylinders in quantum topology.
\end{rem}

Consequently, by combining Theorem \ref{thm:diagram} and Theorem \ref{thm:curve_knot}, 
we obtain a partial, but explicit, computation of $\hat \rho(U_K)$ for any knot $K \subset U$
whose homotopy class $[K]$ belongs to $\Gamma_{k} \pi_1(U) \cong \Gamma_k \pi$:
$$
\varrho^\theta(U_K) \equiv \frac{1}{2} \cdot\log \theta ([K]) \hbox{\textbf{\,-\! -\! -}\,}   \log \theta ([K])
+ \hbox{(trees of degree $\geq 2k$)}
$$
The authors do not know whether this identity holds true in higher degrees. 
In particular, the possibility that $\varrho^\theta(U_K)$  depends only 
on the homotopy class $[K] \in \pi$ is not excluded yet.

To conclude, we mention that Theorem \ref{thm:curve_knot} also implies an analogue of Proposition \ref{prop:one_two} where generalized Dehn twists are replaced by surgeries.
This provides new surgery formulas for the Johnson homomorphisms, 
from which an alternative  proof of the surjectivity of the Johnson homomorphisms \cite{GL05,Habegger00} can be derived.
We refer to \cite{KM19} for further details.

\section{Skein versions of generalized Dehn twists}
\label{sec:skein}

In this section, we present formulas giving the action of the (usual) Dehn twists
on skein algebras of the thickened surface $\Sigma \times [-1,+1]$ in terms of commutators in these algebras.
These formulas lead to some ``skein versions'' of the generalized Dehn twists. 
Recall that, for simplicity, the surface $\Sigma = \Sigma_{g,1}$ is  assumed 
to be compact oriented connected of genus $g$ with one boundary component.
In this section, in addition to the previously fixed basepoint $*\in \partial \Sigma$, 
we take a second basepoint $\bullet \neq *$ in $\partial \Sigma$.

\subsection{Skein algebras and skein modules}
\label{subsec:skein}

We start by recalling the two versions of ``skein algebras'' with which we will work. 

Let $\skein (\Sigma)$ be the \emph{Kauffman bracket skein module} of the thickened surface $\Sigma \times [-1,+1]$. 
This is the quotient of the free $\Q [[A+1]]$-module 
generated by the (isotopy classes of) framed unoriented links in $\Sigma \times [-1,+1]$
modulo the relations of  
Figure \ref{fig:Kauffman_bracket_skein_relation}.
Similarly, we define the Kauffman bracket skein module $\skein(\Sigma, \bullet, *)$ 
by considering the framed unoriented tangles \emph{with endpoints in} $\{ \bullet,*\}$. 
Here tangles have (in addition to closed components) 
a unique component homeomorphic to the interval, called the \emph{string}, 
such that one of its endpoints is in $\{ \bullet \} \times (-1,+1)$ and the other in $\{ * \} \times (-1,+1)$;
furthermore, the framing of the string is given at its endpoints by the positive direction 
of the $(-1,+1)$ factor. 
Note that $A= -1+(A+1)$ and 
$$
A^{-1}= -\frac{1}{1-(A+1)} = -1 - \sum_{m\geq 1} (A+1)^m
$$ 
are viewed here as power series in $(A+1)$.

\begin{rem}
 The Kauffman bracket skein module was introduced by Przytycki \cite{Przy91} for links in oriented $3$-manifolds.
There are several versions for tangles in a thickened surface with boundary.
For a full detail of the version that we work with here, see \cite[Definition 3.2]{TsujiCSAI}.
Note that our version is different from Muller's one \cite{Muller}, which involves additional skein relations.
\end{rem}

\begin{figure}[ht]
\begin{minipage}[b]{0.45\hsize}
  \centering
 \input{Kauffman_bracket_skein_relation.tex}
\caption{$\skein (\Sigma)$}
\label{fig:Kauffman_bracket_skein_relation}
\end{minipage}
\hspace{0.5cm}
\begin{minipage}[b]{0.45\hsize}
  \centering
  \input{HOMFLY-PT_skein_relation.tex}
\caption{$\tskein (\Sigma)$}
\label{fig:HOMFLY-PT_skein_relation}
\end{minipage}
\end{figure}

Let $\tskein (\Sigma)$ be the \emph{HOMFLY-PT skein module} of $\Sigma \times [-1,+1]$.
This  is the quotient of  the free  $\Q [ \rho ] [[h]]$-module  generated by the (isotopy classes of)
framed oriented links in $\Sigma \times [-1,+1]$
modulo the relations of  Figure \ref{fig:HOMFLY-PT_skein_relation}.
Similarly, we define the HOMFLY-PT skein module $\tskein(\Sigma, \bullet, *)$ by considering framed oriented tangles with a unique string, which is oriented from ${\{ \bullet \} \times (-1,+1)}$ to $\{ * \} \times (-1,+1)$
and framed at its endpoints  by the positive direction  of the $(-1,+1)$ factor. 

\begin{rem}
The HOMFLY-PT skein module for links in oriented $3$-manifolds was introduced in the works of Przytycki \cite{Przy91} and Turaev \cite{Tur91}.
Turaev shows that his version of the skein module
gives a quantization of the Goldman bracket.
Here we consider a version for tangles in a thickened surface with boundary.
For more detail, \cite[Definition 3.2]{TsujiHOMFLY-PTskein}.
\end{rem}

\begin{rem} \label{rem:skein_functorial} 
The definitions above of the skein modules work for any compact oriented surface (with non-empty boundary).
Furthermore, they are functorial in the following sense:
suppose that $\Sigma$ and $\Sigma'$ are compact oriented surfaces with non-empty boundary and assume given an orientation-preserving embedding $e \colon \Sigma \times [-1,+1] \to \Sigma' \times [-1,+1]$ of $3$-manifolds, then $e$ induces a $\Q[[A+1]]$-linear map $e_* \colon \skein(\Sigma) \to \skein(\Sigma')$ and a $\Q[\rho][[h]]$-linear map $e_*\colon \tskein(\Sigma) \to \tskein(\Sigma')$.
\end{rem}

In the sequel, we sometimes use the notation $(\calG, \calV)$ in order to refer either 
to  the pair $(\skein ( \Sigma ), \skein (\Sigma, \bullet, * ))$ or to the pair $(\tskein (\Sigma),  \tskein (\Sigma, \bullet, * ))$;
in such a situation, we set
$$
\epsilon_\calG :=\left\{\begin{array}{ll} -A+\gyaku{A} & \hbox{if } \calG = \skein(\Sigma),\\
h  & \hbox{if } \calG = \tskein(\Sigma).
\end{array}\right.
$$
The ``stacking'' operation $ab$ (with $a$ ``over'' $b$) is defined whenever $(a,b)$ is an element of 
$\calG\times \calG$, $\calG\times \calV$ or $ \calV \times \calG$. 
With this operation, $\calG$ is an associative algebra and called the \emph{skein algebra}.
There is also a Lie bracket on $\calG$ defined by renormalizing the algebra commutator operation:
\begin{equation} \label{eq:[-,-]}
[x,x']:= \frac{xx'-x'x}{\epsilon_{\calG}},
\end{equation}
and there is  an action $\sigma$ of the Lie algebra $\calG$ on  $\calV$ defined by
\begin{equation} \label{eq:sigma_ter}
\sigma(x)(y) := \frac{xy-yx}{\epsilon_{\calG}}.
\end{equation}
The operations \eqref{eq:[-,-]} and \eqref{eq:sigma_ter} are well-defined 
(although $\epsilon_{\calG}$ is not invertible in the ground ring): 
 see \cite[\S 3.2]{TsujiCSAI} and \cite[\S 3.3]{TsujiHOMFLY-PTskein}.

\subsection{Comparison of algebras and modules}
\label{subsec:comparison_ma}

We now explain how the skein algebras/modules are related one to the other,
and how they are connected to the constructions of the previous sections like Goldman's Lie bracket. 
We shall define four types of homomorphisms --- for algebras as well as for modules;
these homomorphisms will fit all together into diagrams
\begin{equation} \label{eq:comparison_abs}
\xymatrix{
\qquad\qquad \mathcal{A}(\Sigma)  \supset \tzeroskein (\Sigma) 
\ar@{>}[rr]^-{\psi} \ar@{>>}[d]^{\varpi} & & \mathcal{S}(\Sigma) \ar@{>>}[d]^{\varpi} \\
\qquad\qquad S'(\GL  ) \supset S'( \bGL  ) \ar@{>>}[rr]^-{\psi} & &  \hskein (\Sigma ). \\
}
\end{equation}
and
\begin{equation} \label{eq:comparison_rel}
\xymatrix{
\quad \qquad\qquad \mathcal{A}(\Sigma,\bullet,*)  \supset \tzeroskein (\Sigma,\bullet,*) 
\ar@{>}[rr]^-{\psi} \ar@{>>}[d]^{\varpi} & & \mathcal{S}(\Sigma,\bullet,*) \ar@{>>}[d]^{\varpi} \\
\quad \qquad\qquad S'(\GL  )  \otimes \Q \pi_{\bullet,*}\supset S'( \bGL  )  \otimes \Q \pi_{\bullet,*}
\ar@{>>}[rr]^-{\psi} & &  \hskein (\Sigma,\bullet,*) \\
}
\end{equation}
whose notations will be explained one by one. 
Those diagrams will be commutative in the following sense:
for any $x\in \tzeroskein (\Sigma)$ such that $\varpi(x)\in S'( \bGL  )$
(resp., any $x\in \tzeroskein (\Sigma,\bullet,*)$ such that $\varpi(x)\in S'( \bGL  ) \otimes \Q\pi_{\bullet,*}$),
we have $\psi(\varpi(x))= \varpi(\psi(x))$.\\

\noindent 
\textbf{$\star$ Maps $\varpi\colon \tskein ( \Sigma ) \to S' (\GL  )$ and
$\varpi \colon \tskein (\Sigma, \bullet,*) \to  S' (\GL  ) \otimes \Q \pi_{\bullet,*}$.}\\[0.2cm] 
Recall that $\vert \pi \vert$ is the set of conjugacy classes of $\pi$,
with natural projection $\zettaiti{ -}\colon \Q \pi \to \GL$,
and that $[ - , - ]\colon  \GL   \times \GL   \to \GL  $ denotes Goldman's Lie bracket
(see  Remarks \ref{rem:operations} and \ref{rem:piconj}). 
Let $ \pi_{\bullet,*}:= \pi_{1}(\Sigma,\bullet,*)$ be the set of homotopy classes of paths connecting $\bullet$ to $*$.
There is a groupoid version of the action \eqref{eq:sigma} 
\begin{equation} \label{eq:sigma_bis}
\sigma\colon \GL  \otimes  \Q \pi_{\bullet,*} \longrightarrow \Q \pi_{\bullet,*}
\end{equation}
which is defined in  the same manner (see \cite{KKgroupoid,KK16}) and makes $\Q\pi_{\bullet,*}$ into a $\GL$-module. 

Let $S'(\GL  )$ be the quotient of the symmetric algebra of $\GL$ where the trivial loop is identified with the constant $1\in \Q$.
Since the trivial loop is central in $\GL$, one can extend by the Leibniz rule the Goldman bracket 
to $S'(\GL  )$ which turns into a Poisson algebra.
The resulting Lie bracket of $S'(\GL  )$ and the action \eqref{eq:sigma_bis} merge 
to give an action $\sigma$ of the Lie algebra $S'(\GL  )$ on the $\Q$-vector space $S' (\GL  ) \otimes \Q \pi_{\bullet,*}$.
This action is characterized by the following two properties:
\begin{itemize} 
\item[(i)] for any degree one element $w \in \GL$, for all $v\in S'(\GL)$ and $y\in \Q\pi_{\bullet,*}$,
\[
\sigma(w)(v\otimes y) = [w,v] \otimes y + v \otimes \sigma(w) (y);
\]
\item[(ii)] for any $w, w' \in S'(\GL)$, for all $v\in S'(\GL)$ and $y\in \Q\pi_{\bullet,*}$,
\[
\sigma(ww')(v\otimes y) =
w' \big( \sigma(w)(v\otimes y) \big) +  w \big( \sigma(w')(v \otimes y) \big)
\]
where $w'$ acts on the first factor of $\sigma(w)(v\otimes y)$ by multiplication.
\end{itemize}

There is a surjective $\Q$-linear map 
$\varpi\colon \tskein (\Sigma) \to S' (\GL  )$
 which maps any  framed oriented link $L=L_1 \cup \cdots \cup L_j$ 
 to the product $[L_1]\cdots[L_j] \in  S' (\GL  )$ of (the free homotopy classes of) its components projected onto $\Sigma$.
By convention, the empty link is mapped to the unit $1\in  S' (\GL  )$.
Furthermore, assigning $0$ to the variable $h$ and $1/2$ to the variable $\rho$ ensures that $\varpi$ is well-defined.
Besides, there is a surjective $\Q$-linear map 
$
\varpi\colon \tskein (\Sigma, \bullet,*) \to S' (\GL  ) \otimes \Q \pi_{\bullet,*}
$
which is defined in a similar way by projecting the string of a tangle  to (the homotopy class of) its projection.
It  is easily  verified 
(using the idea of Turaev \cite[proof of Theorem 3.3]{Tur91}) that the diagrams 
$$
\xymatrix{ \tskein (\Sigma)  \otimes \tskein (\Sigma)  \ar[r]^-{[-,-]} \ar[d]_-{\varpi \otimes \varpi} 
& \tskein (\Sigma)  \ar[d]^-{\varpi}   \\
 S' (\GL  )  \otimes  S' (\GL  ) \ar[r]_-{[-,-]} &  S' (\GL  ) 
}
$$
and 
$$
\xymatrix{ \tskein (\Sigma)  \otimes \tskein (\Sigma, \bullet,*) \ar[r]^-\sigma \ar[d]_-{\varpi \otimes \varpi} 
& \tskein (\Sigma, \bullet,*) \ar[d]^-{\varpi}   \\
 S' (\GL  )  \otimes \big( S' (\GL  ) \otimes \Q \pi_{\bullet,*} \big) \ar[r]_-\sigma &  S' (\GL  ) \otimes \Q \pi_{\bullet,*}
}
$$
are commutative: in other words, 
 $\varpi\colon \tskein ( \Sigma ) \to S' (\GL  )$ is a Lie algebra homomorphism and the map 
$\varpi\colon \tskein (\Sigma, \bullet,*) \to  S' (\GL  ) \otimes \Q \pi_{\bullet,*}$ of Lie modules is equivariant over this homomorphism.\\

\noindent
\textbf{$\star$ Maps  $\psi  \colon \tzeroskein (\Sigma) \to \skein (\Sigma)$ 
and $\psi \colon \tskein (\Sigma, \bullet,*) \to \skein (\Sigma, \bullet,*) $.}\\[0.2cm] 
It is easily verified from its defining relations that the  $\Q[\rho][[h]]$-module
$\tskein (\Sigma)$ has a direct sum decomposition 
\begin{equation*}
\tskein (\Sigma)= \bigoplus_{x \in H_1 (\Sigma; \Z)}	 \tskein_x (\Sigma)
\end{equation*}
where $\tskein_x (\Sigma)$ denotes the  $\Q[\rho][[h]]$-submodule 
generated by links with total homology class equal to $x$.
Then, for any $x, y \in  H_1 (\Sigma;\Z)$, we have 
$ \tskein_x (\Sigma) \cdot \tskein_y (\Sigma)  \subset \tskein_{x+y} (\Sigma)$
and $[\tskein_x (\Sigma), \tskein_y (\Sigma)] \subset \tskein_{x+y}(\Sigma)$.

Let $\psi' \colon \tskein (\Sigma) \to \skein (\Sigma)$ be the $\Q$-linear map 
defined  by $\psi'(L):= (-A)^{w ( L )}\,  L$ for any framed oriented link $L$, 
 while assigning  $-A^2+A^{-2} $ to $h$
and 
\begin{eqnarray*}
\frac{\log A^4}{-A^2+A^{-2}} &=& \frac{4\log\big((1-(A+1)\big)}{-\big(1-(A+1)\big)^2 + \big(1-(A+1)\big)^{-2}} \\
&=& -1 + \frac{2}{3}(A+1)^2 + \cdots \ \in \Q[[A+1]]
\end{eqnarray*}
to $\rho$ so that $\exp (\rho h)=A^4$.
Here  $w(L)$ is the \emph{total framing number} of~$L$,
which can be computed as the difference between the total number of positive crossings 
and  the total number of negative crossings in a projection diagram of $L$.  
(See \cite[Proposition 7.15]{TsujiHOMFLY-PTskein}.)
Similarly, we define a $\Q$-linear map $\psi' \colon \tskein (\Sigma,\bullet,*) \to \skein (\Sigma,\bullet,*)$. 
We remark that, for any framed oriented link $L$ and any framed  oriented tangle $T$,
$$
\psi' (L T)=(-A)^{w(L,T)}\, 
\psi' (L)\, \psi' (T)
$$
and
$$
\psi' (TL)=(-A)^{w(T,L)}\,
\psi'(T)\, \psi'(L)
$$
where  $w(L,T)=-w(T,L)$ is the intersection number $\omega([L],[T])$
of the homology classes of $L$ and $T$ projected onto $\Sigma$. 
Hence, for any $x \in \tskein_0 (\Sigma)$ and $y \in \tskein (\Sigma)$, we have 
\begin{equation} \label{eq:psi'1}
 \psi'(xy) = \psi'(x)\, \psi'(y), \quad
\psi'(yx) = \psi'(y)\, \psi'(x) 
\end{equation}
and, similarly, for any $x \in \tskein_0 (\Sigma)$ and $z \in \tskein (\Sigma, \bullet,*)$, we have
\begin{equation} \label{eq:psi'2}
 \psi'(xz) = \psi' (x)\, \psi'(z), \quad
\psi'(zx) = \psi' (z)\, \psi'(x). 
\end{equation}
Next, we renormalize the above two maps $\psi'$ by setting $\psi := \frac{1}{A+A^{-1}} \psi'$ where 
$$
\frac{1}{A+A^{-1}} = \frac{1}{-2 - \sum_{m\geq 2}(A+1)^m} 
= - \frac{1}{2}+ \frac{1}{4} (A+1)^2+ \cdots \ \in \Q[[A+1]].
$$ 
It follows from~\eqref{eq:psi'1} that 
$\psi \colon \tzeroskein  (\Sigma) \to \skein (\Sigma) $ is a Lie algebra homomorphism,
and it follows from  \eqref{eq:psi'2} that 
the map  $\psi \colon \tskein (\Sigma,\bullet,*) \to \skein (\Sigma, \bullet,*) $ of Lie modules 
is equivariant over this homomorphism. \\
 
 \noindent
\textbf{$\star$ Maps  $\varpi \colon \skein (\Sigma) \to \hskein (\Sigma)$ 
and $\varpi \colon \skein (\Sigma,\bullet,*) \to \hskein (\Sigma,\bullet,*)$.}\\[0.2cm]
Let $\hskein ( \Sigma)$  be the \emph{Kauffman bracket skein algebra  ``at $A:=-1$''},
namely the quotient  of $\skein(\Sigma)$  by $(A+1)\, \skein (\Sigma)$,
and let $\varpi \colon \skein (\Sigma) \to \hskein (\Sigma)$  be the canonical projection. 
Define the quotient space $\hskein (\Sigma,\bullet,*)$ in a similar way, and
let  $\varpi \colon \skein (\Sigma,\bullet,*) \to \hskein (\Sigma,\bullet,*)$ be the canonical projection.
In those quotients, the class of a tangle  depends only on its homotopy class,  and hence $\hskein (\Sigma)$ is a commutative algebra.
The Lie bracket of $\skein (\Sigma)$
and the Lie action of $\skein (\Sigma)$ on $\skein (\Sigma, \bullet,*)$ 
descend to $\hskein (\Sigma)$ and $\hskein (\Sigma,\bullet,*)$, respectively.
 Hence, just like $\skein (\Sigma)$, the quotient $\hskein (\Sigma)$ is a Poisson algebra.\\

\noindent
\textbf{$\star$ Maps  $\psi \colon S' (\bGL  ) \to \hskein (\Sigma)$ and 
$\psi \colon S'(\bGL  ) \otimes  \Q \pi_{\bullet,*} \to \hskein (\Sigma , \bullet, *)$.}\\[0.2cm]
 Since the works of Goldman \cite{Gol86}, Turaev \cite{Tur91}, Bullock \cite{Bullock97} and others, 
there is a well-known relationship between
the symmetric algebra  of the Goldman Lie algebra, the $\hbox{SL}_2$-representation algebra of $\pi$ 
and the Kauffman bracket skein algebra  ``at $A:=-1$''. 
Let us recall this relationship in our setting.

Goldman \cite{Gol86} introduced the subspace $\bGL  $ of $\GL  $ generated by the elements 
\begin{equation*}
\zettaisq{x} :=
 \zettaiti{x} + \zettaiti{x^{-1}}
\end{equation*} 
for all $x\in \pi$, which correspond to 
homotopy classes of unoriented loops in~$\Sigma$.
It turns out that $\bGL  $ is a Lie subalgebra of $\GL  $.
We denote by $S'(\bGL  )$ the quotient of the symmetric algebra of $\bGL$ by the relation $\zettaisq{1} = 2 \in \Q$.
Let $\psi \colon S'(\bGL  ) \to \hskein (\Sigma)$ be the algebra homomorphism 
mapping, for all $x\in \pi$, the element $\zettaisq{x}$ 
to  \emph{minus} the class of any knot 
that projects onto  the unoriented loop corresponding 
to $\zettaisq{x}$. In a similar way, we  have  
a $\Q$-linear map $\psi \colon S'(\bGL  ) \otimes  \Q \pi_{\bullet,*} \to \hskein (\Sigma , \bullet, *)$.
It can be verified that the maps $\psi$ have the following properties.

\begin{prop} \label{prop:psi}
\begin{enumerate} 
\item[(i)] For all $v,v'\in S'(\bGL  )$, $\psi ([v,v'])= [\psi (v), \psi (v')]$.
\item[(ii)] The map $\psi \colon S' (\bGL  ) \to \hskein (\Sigma)$ is surjective and its kernel is the ideal generated by 
$\zettaisq{xx'}+ \zettaisq{x^{-1}x'}  -  \zettaisq{x}\zettaisq{x'}$ for all $x,x' \in \pi$.
\item[(iii)] For all $v\in S'(\bGL  )$ and $w\in \Q \pi_{\bullet,*}$, 
$\psi \big(\sigma (v)(w)\big)= \sigma \big(\psi(v)) (\psi (w)\big)$.
\item[(iv)] The map $\psi \colon S'(\bGL  ) \otimes  \Q \pi_{\bullet,*} \to \hskein (\Sigma , \bullet, *)$ is surjective
and its kernel  is  generated as an $S' (\bGL)$-module by the following two types of elements:
\begin{itemize} 
\item $1 \otimes  (rxr' + rx^{-1}r')-
\zettaisq{x} \otimes (rr')$, where $r\in \pi_{\bullet,*}$ and $r',x \in \pi$; 
\item $(\zettaisq{xx'} +\zettaisq{x^{-1}x'}-
\zettaisq{x}\zettaisq{x'}) \otimes r$, where $x,x' \in \pi$ and $r\in \pi_{\bullet,*}$.
\end{itemize}
\end{enumerate}
\end{prop}

Statement (i) is the well-known expression of  the Goldman bracket as a commutator in the skein algebra:
thus  $\psi \colon S' (\bGL  ) \to \hskein (\Sigma)$ is a Lie algebra homomorphism and, by (iii),
the map $\psi \colon S'(\bGL  ) \otimes  \Q \pi_{\bullet,*} \to \hskein (\Sigma , \bullet, *)$  of Lie modules
is equivariant over this homomorphism.

Statement (ii) is the usual description of the Kauffman bracket skein algebra ``at $A:=-1$'' 
in terms of the fundamental group $\pi$.  
This commutative algebra is known to be isomorphic to the $\operatorname{GL_2}$-invariant part of 
the  coordinate algebra of the affine scheme $X_\pi$ of $\operatorname{SL_2}$-representations  of the group $\pi$.
Specifically, $X_\pi$ is the functor that assigns to any commutative $\Q$-algebra $B$ 
the set of group homomorphisms $\Hom(\pi,\operatorname{SL}_2(B))$.
The action of the group scheme $\operatorname{GL_2}$ on $X_\pi$ given by conjugation of matrices
induces an action of $\operatorname{GL_2}$ on the coordinate algebra $\Q[X_\pi]$ of $X_\pi$: 
we denote by $\Q[X_\pi]^{\operatorname{GL_2}}$ 
the $\operatorname{GL_2}$-invariant part. The \emph{trace map}
$$
\operatorname{tr}\colon S'(\bGL  ) \longrightarrow \Q[X_\pi]^{\operatorname{GL_2}}
$$
is the algebra homomorphism  mapping $\zettaisq{x}$ 
to the regular function that, for any commutative $\Q$-algebra $B$,  
is defined by $\Hom(\pi, \operatorname{SL_2}(B))\ni \rho \mapsto \operatorname{tr}(\rho(x))\in B$. 
Besides, there is another \emph{trace map}
$$
\operatorname{tr}\colon \hskein (\Sigma) \longrightarrow \Q[X_\pi]^{\operatorname{GL_2}}
$$
which is the algebra homomorphism mapping a knot $K$ to the regular function defined by 
$\Hom(\pi, \operatorname{SL_2}(B))\ni \rho \mapsto -\operatorname{tr}(\rho(K))\in B$.
(Here $\pi$ is identified with $\pi_1(\Sigma \times [-1,+1])$.)
Then, we have a commutative diagram
$$
\xymatrix{
S'(\bGL  ) \ar@{->>}[d]_-\psi \ar[r]^-{\operatorname{tr}} & \Q[X_\pi]^{\operatorname{GL_2}}. \\
\hskein (\Sigma) \ar[ru]^-{\operatorname{tr}}_-{\  \cong }}
$$
 The trace map defined on $\hskein (\Sigma)$ is known to be an isomorphism
(see \cite[Proposition 9.1]{BH95} and \cite[Theorem  7.1]{PS00}):
therefore the kernel of the trace map defined on $S'(\bGL )$ coincides with $\ker(\psi)$,
which is described by statement (ii) of Proposition \ref{prop:psi}.
Finally, note that the Poisson bracket in $\Q[X_\pi]^{\operatorname{GL_2}}$ corresponding
to the Poisson bracket in $\hskein (\Sigma)$ via the trace map is an algebraic counterpart
of the Atiyah--Bott Poisson structure on the moduli space of $\operatorname{SL}_2(\mathbb{C})$-representation of $\pi$ \cite{Gol86}.

\subsection{Filtrations}

For both $\calG = \skein(\Sigma)$ and $\calG = \tskein(\Sigma)$, 
the skein algebra~$\calG$ and the corresponding skein module $\calV$ can be endowed with decreasing filtrations.
Denoted by $\{ F^n \calG \}_n$ and $\{ F^n \calV \}_n$, respectively, these filtrations have the following properties.

\begin{enumerate}
\item[(i)]
The stacking operations of $(\calG,\calV)$ are filtration-preserving.

\item[(ii)] 
The Lie bracket $[-,-]$ of $\calG$ and the Lie action $\sigma$ of $\calG$ on $\calV$
are maps of degree $(-2)$: for instance, one has $[F^m \calG, F^n \calG] \subset F^{m+n-2} \calG$. 

\item[(iii)]
The image of the filtration of $\tskein (\Sigma)$ by the surjective map $\varpi \colon \tskein(\Sigma) \to S' (\GL)$ 
is the filtration of $S' (\GL)$ inherited from the $I$-adic filtration of $\Q \pi$ (see Remark \ref{rem:piconj}).

\item[(iv)]
All the maps in the diagrams \eqref{eq:comparison_abs} and \eqref{eq:comparison_rel} are filtration-preserving.

\item[(v)]
Let $e \colon \Sigma \times [-1,+1] \to
\Sigma' \times [-1, +1]$ be an embedding between thickened surfaces as in Remark \ref{rem:skein_functorial}.
Then the induced map $e_* \colon \skein (\Sigma) 
\to \skein (\Sigma')$ and $e_* \colon \tskein (\Sigma) 
\to \tskein (\Sigma')$ are filtration-preserving.
\end{enumerate}

We do not give details of the construction of these filtrations
which are defined by explicit systems of generators;
see \cite[\S 5]{TsujiCSAI} and \cite[\S 5]{TsujiCSApure} for ${\calG = \skein(\Sigma)}$, 
and \cite[\S 4]{TsujiHOMFLY-PTskein} for $\calG = \tskein(\Sigma)$. Instead, we give a few examples.

\begin{example}
\begin{enumerate}
\item[(1)] 
Consider an embedding $e$ of the oriented fatgraph 
$$
\Gamma := \begin{array}{c} 
{\unitlength 0.1in%
\begin{picture}(14.0500,6.0000)(2.0000,-7.0000)%
%
\special{pn 8}%
\special{pa 200 350}%
\special{pa 200 350}%
\special{fp}%
\special{pa 200 350}%
\special{pa 400 350}%
\special{fp}%
%
\special{pn 8}%
\special{ar 600 350 200 200 0.0000000 6.2831853}%
%
\special{pn 8}%
\special{pa 800 350}%
\special{pa 800 350}%
\special{fp}%
\special{pa 800 350}%
\special{pa 1000 350}%
\special{fp}%
%
\special{pn 8}%
\special{ar 1200 350 200 200 0.0000000 6.2831853}%
%
\special{pn 8}%
\special{pa 1400 350}%
\special{pa 1605 350}%
\special{fp}%
%
\special{pn 8}%
\special{pa 1150 100}%
\special{pa 1200 150}%
\special{fp}%
\special{pa 1200 150}%
\special{pa 1150 200}%
\special{fp}%
\special{pa 550 100}%
\special{pa 600 150}%
\special{fp}%
\special{pa 600 150}%
\special{pa 550 200}%
\special{fp}%
%
\special{pn 8}%
\special{pa 200 350}%
\special{pa 200 350}%
\special{fp}%
\special{pa 200 350}%
\special{pa 200 350}%
\special{fp}%
\special{pa 200 350}%
\special{pa 200 650}%
\special{fp}%
\special{pa 200 650}%
\special{pa 1600 650}%
\special{fp}%
\special{pa 1600 650}%
\special{pa 1600 650}%
\special{fp}%
\special{pa 1600 650}%
\special{pa 1600 350}%
\special{fp}%
%
\special{pn 8}%
\special{pa 850 300}%
\special{pa 850 300}%
\special{fp}%
\special{pa 850 300}%
\special{pa 900 350}%
\special{fp}%
\special{pa 900 350}%
\special{pa 850 400}%
\special{fp}%
%
\special{pn 8}%
\special{pa 1150 500}%
\special{pa 1200 550}%
\special{fp}%
\special{pa 1200 550}%
\special{pa 1150 600}%
\special{fp}%
\special{pa 550 500}%
\special{pa 600 550}%
\special{fp}%
\special{pa 600 550}%
\special{pa 550 600}%
\special{fp}%
%
\special{pn 8}%
\special{pa 955 605}%
\special{pa 955 605}%
\special{fp}%
\special{pa 955 605}%
\special{pa 905 650}%
\special{fp}%
%
\special{pn 8}%
\special{pa 910 650}%
\special{pa 950 700}%
\special{fp}%
\end{picture}}%
 \end{array}
$$  into $\Sigma \times [-1,+1]$.
Then, the linear combination of knots
\[
{\unitlength 0.1in%
\begin{picture}(41.1500,2.7400)(3.6400,-6.2000)%
\put(3.6400,-6.2000){\makebox(0,0)[lb]{$e($}}%
%
\special{pn 8}%
\special{ar 678 470 100 100 3.1415927 6.2831853}%
%
\special{pn 8}%
\special{pa 778 470}%
\special{pa 778 470}%
\special{fp}%
\special{pa 878 470}%
\special{pa 778 470}%
\special{fp}%
%
\special{pn 8}%
\special{ar 978 470 100 100 3.1415927 6.2831853}%
%
\special{pn 8}%
\special{pa 1078 470}%
\special{pa 1128 470}%
\special{fp}%
\special{pa 1128 470}%
\special{pa 1128 620}%
\special{fp}%
\special{pa 1128 620}%
\special{pa 528 620}%
\special{fp}%
\special{pa 528 470}%
\special{pa 528 620}%
\special{fp}%
\special{pa 528 470}%
\special{pa 578 470}%
\special{fp}%
\put(11.6500,-6.2000){\makebox(0,0)[lb]{$)$}}%
\put(12.9800,-6.2000){\makebox(0,0)[lb]{$-\, \, e($}}%
%
\special{pn 8}%
\special{ar 1776 470 100 100 6.2831853 3.1415927}%
%
\special{pn 8}%
\special{pa 1876 470}%
\special{pa 1976 470}%
\special{fp}%
\special{pa 2176 470}%
\special{pa 2226 470}%
\special{fp}%
\special{pa 2226 470}%
\special{pa 2226 620}%
\special{fp}%
\special{pa 2226 620}%
\special{pa 1626 620}%
\special{fp}%
\special{pa 1626 620}%
\special{pa 1626 470}%
\special{fp}%
\special{pa 1676 470}%
\special{pa 1626 470}%
\special{fp}%
%
\special{pn 8}%
\special{ar 2076 470 100 100 3.1415927 6.2831853}%
\put(22.6400,-6.2000){\makebox(0,0)[lb]{$)$}}%
\put(23.9900,-6.2000){\makebox(0,0)[lb]{$-\, \, e($}}%
%
\special{pn 8}%
\special{pa 2726 470}%
\special{pa 2776 470}%
\special{fp}%
\special{pa 2726 470}%
\special{pa 2726 620}%
\special{fp}%
\special{pa 2726 620}%
\special{pa 3326 620}%
\special{fp}%
\special{pa 3326 620}%
\special{pa 3326 470}%
\special{fp}%
\special{pa 3276 470}%
\special{pa 3326 470}%
\special{fp}%
\special{pa 3076 470}%
\special{pa 2976 470}%
\special{fp}%
%
\special{pn 8}%
\special{ar 2876 470 100 100 3.1415927 6.2831853}%
%
\special{pn 8}%
\special{ar 3176 470 100 100 6.2831853 3.1415927}%
\put(33.6400,-6.2000){\makebox(0,0)[lb]{$)$}}%
\put(35.0700,-6.2000){\makebox(0,0)[lb]{$+\, \, e($}}%
%
\special{pn 8}%
\special{pa 3891 470}%
\special{pa 3841 470}%
\special{fp}%
\special{pa 3841 470}%
\special{pa 3841 620}%
\special{fp}%
\special{pa 3841 620}%
\special{pa 4441 620}%
\special{fp}%
\special{pa 4441 620}%
\special{pa 4441 470}%
\special{fp}%
\special{pa 4441 470}%
\special{pa 4391 470}%
\special{fp}%
\special{pa 4191 470}%
\special{pa 4091 470}%
\special{fp}%
%
\special{pn 8}%
\special{ar 3991 470 100 100 6.2831853 3.1415927}%
%
\special{pn 8}%
\special{ar 4291 470 100 100 6.2831853 3.1415927}%
\put(44.7900,-6.2000){\makebox(0,0)[lb]{$)$}}%
%
\special{pn 8}%
\special{pa 2098 372}%
\special{pa 2050 348}%
\special{fp}%
\special{pa 2098 372}%
\special{pa 2050 396}%
\special{fp}%
%
\special{pn 8}%
\special{pa 2892 370}%
\special{pa 2844 346}%
\special{fp}%
\special{pa 2892 370}%
\special{pa 2844 394}%
\special{fp}%
%
\special{pn 8}%
\special{pa 1792 568}%
\special{pa 1744 544}%
\special{fp}%
\special{pa 1792 568}%
\special{pa 1744 592}%
\special{fp}%
%
\special{pn 8}%
\special{pa 3196 568}%
\special{pa 3148 544}%
\special{fp}%
\special{pa 3196 568}%
\special{pa 3148 592}%
\special{fp}%
%
\special{pn 8}%
\special{pa 4306 568}%
\special{pa 4258 544}%
\special{fp}%
\special{pa 4306 568}%
\special{pa 4258 592}%
\special{fp}%
%
\special{pn 8}%
\special{pa 4006 568}%
\special{pa 3958 544}%
\special{fp}%
\special{pa 4006 568}%
\special{pa 3958 592}%
\special{fp}%
%
\special{pn 8}%
\special{pa 697 372}%
\special{pa 649 348}%
\special{fp}%
\special{pa 697 372}%
\special{pa 649 396}%
\special{fp}%
%
\special{pn 8}%
\special{pa 997 372}%
\special{pa 949 348}%
\special{fp}%
\special{pa 997 372}%
\special{pa 949 396}%
\special{fp}%
\end{picture}}%
\]
obtained by restriction of $e$ to four cycles in $\Gamma$ is one of the generators of $F^2 \tskein (\Sigma)$.

\item[(2)]
If $K$ is a framed unoriented knot in $\Sigma \times [-1,+1]$, then $K+2$ is one of the generators of $F^2 \skein (\Sigma)$.
\end{enumerate}
\end{example}

\begin{rem}
As the above discussion indicates, the filtrations on $\calG$ and~$\calV$ could be thought of as analogues of filtrations in other contexts:
one is the $I$-adic filtration of the group algebra $\Q \pi$, and the other is the Vassiliev--Goussarov filtration 
in the theory of finite-type invariants of links.
In fact, when the surface $\Sigma$ is a disk, the filtrations on $\calG$ and~$\calV$ are indeed induced  by the Vassiliev--Goussarov filtration.
(See \cite[\S 5.1]{TsujiCSApure} and \cite[\S 4.1]{TsujiHOMFLY-PTskein}.) 
However, it should be remarked that
when $\Sigma$ has non-trivial fundamental group, 
the filtrations on $\calG$ and~$\calV$
can not be obtained by simply ``resolving double points of singular links''
but needs the more general constructions of Goussarov \cite{Goussarov98}  and Habiro \cite{Habiro00}.
\end{rem}

We shall need the completions of the skein algebra and skein module resulting from those filtrations:
\[
\widehat{\calG} := \varprojlim_m \calG /F^m \calG, \hskip 2em
\widehat{\calV} := \varprojlim_m \calV /F^m \calV.
\]
Thus we obtain 
$\cskein(\Sigma)$, $\cskein(\Sigma,\bullet,*)$,  $\ctskein(\Sigma)$ and  $\ctskein(\Sigma,\bullet,*)$.
From properties (i) to (iv) of the filtrations listed above, 
it follows that the stacking operation, the Lie bracket $[-,-]$, the Lie action $\sigma$, 
and all the maps in diagrams \eqref{eq:comparison_abs} and \eqref{eq:comparison_rel} descend to the completions.
Also, property (v) shows that the maps $e_*$ descend to completions.

\subsection{Dehn twists on skein modules}
\label{subsection:skein_Dehn}

We are now ready to state two variations of the formula \eqref{eq:tCexp} for the  
skein modules $\calV = \skein(\Sigma,\bullet,*)$ and $\calV = \tskein(\Sigma,\bullet, * )$.
The relationship between these and the original version \eqref{eq:tCexp} will be explained in the next subsection. \\

\noindent
\textbf{$\star$ The case of the Kauffman bracket skein module.}\\[0.2cm] 
For a simple closed curve $C$ in $\Sigma$, we define the element $L_{\skein}(C) \in \cskein (\Sigma)$ by
\begin{equation} \label{eq:LKC}
L_{\skein} (C)
:=
\frac{-A+A^{-1}}{4 \log (-A)}
\left(\arccosh \left(\frac{-C}{2} \right) \right)^2
-(-A+A^{-1}) \log (-A).
\end{equation}
Here, we regard $C$ as a framed knot in $\Sigma \times [-1, +1]$ 
using the framing defined by the surface,
and we use the following power series:  
\begin{eqnarray*}
 \left(\arccosh \left(\frac{-x}{2} \right) \right)^2
&= &  -\frac{1}{4} \sum_{i=0}^{\infty}
\frac{i!i!}{(i+1)(2i+1)!} \left(4-x^2\right)^{i+1}  \\
&=& -(x+2)- \frac{1}{12}(x+2)^2 + \cdots \ \in \Q[[x+2]] 
\end{eqnarray*} 
\begin{eqnarray}
\label{eq:fraction} \frac{-A+A^{-1}}{4 \log (-A)} 
&=& \frac{1-(A+1) - \sum_{n\geq 0} (A+1)^n}{-4\sum_{n\geq 1} (A+1)^n/n } \\
\notag &=&\frac{1}{2} + \frac{1}{12} (A+1)^2 + \cdots \ \in \Q[[A+1]]
\end{eqnarray}
\begin{eqnarray*}
\notag (-A+A^{-1})\log (-A)  &=& -  \Big(1-(A+1) - \sum_{n\geq 0} (A+1)^n\Big) \cdot 
 \sum_{n\geq 1} \frac{(A+1)^n}{n} \\
&=& 2(A+1)^2 + 2(A+1)^3 + \cdots  \  \in \Q[[A+1]]
\end{eqnarray*}

\begin{thm}[{\cite[Theorem 4.1]{TsujiCSAI}}]
\label{thm_main_Dehn_skein}
The action of the Dehn twist along $C$ on the completed skein module $\cskein (\Sigma, \bullet, *)$ coincides with the exponential of $\sigma ( L_{\skein} (C))$:
\begin{equation*}
t_C =\exp \big( \sigma (L_{\skein} (C))\big) \colon \cskein (\Sigma,\bullet,*) \longrightarrow \cskein(\Sigma,\bullet,*).
\end{equation*}
\end{thm}

\begin{rem} \label{rem:second_term}
Of course, the power series $(-A+A^{-1})\log (-A)$ is in the annihilator of the Lie action $\sigma$, and Theorem \ref{thm_main_Dehn_skein} still holds if we omit this term in \eqref{eq:LKC}.
However, it plays some role in a relationship between $\skein (\Sigma)$ and the Torelli group of $\Sigma$.
See \cite[Remark 6.5]{TsujiCSApure} and Remark \ref{rem:BSCC} below.
\end{rem}

\vskip 1em

\noindent
\textbf{$\star$ The case of the HOMFLY-PT skein module.}\\[0.2cm] 
Let $\pi^+_{\bullet,*}:= \pi^+(\Sigma,\bullet,*)$
be the set of  regular homotopy classes of properly immersed paths from $\bullet$ to $*$.
Let $\nu$ be the simple path from $\bullet$ to $*$ 
that traverses the oriented boundary of~$\Sigma$: 
$$
{\unitlength 0.1in%
\begin{picture}(10.0000,5.0000)(3.2000,-6.0000)%
\put(3.2000,-3.2000){\makebox(0,0)[lb]{$\Sigma$}}%
%
\special{pn 13}%
\special{pa 520 560}%
\special{pa 1320 560}%
\special{fp}%
%
\special{pn 4}%
\special{sh 1}%
\special{ar 680 560 16 16 0 6.2831853}%
\special{sh 1}%
\special{ar 1160 560 16 16 0 6.2831853}%
\special{sh 1}%
\special{ar 1160 560 16 16 0 6.2831853}%
\put(6.6600,-6.9200){\makebox(0,0)[lb]{$\bullet$}}%
\put(11.4600,-6.9200){\makebox(0,0)[lb]{$*$}}%
\put(8.8000,-7.0000){\makebox(0,0)[lb]{$\nu$}}%
%
\special{pn 8}%
\special{pa 960 560}%
\special{pa 920 520}%
\special{fp}%
\special{pa 960 560}%
\special{pa 920 600}%
\special{fp}%
%
\special{pn 0}%
\special{sh 0.300}%
\special{pa 520 540}%
\special{pa 900 540}%
\special{pa 900 100}%
\special{pa 520 100}%
\special{pa 520 540}%
\special{ip}%
\special{pn 8}%
\special{pa 520 540}%
\special{pa 900 540}%
\special{pa 900 100}%
\special{pa 520 100}%
\special{pa 520 540}%
\special{ip}%
%
\special{pn 0}%
\special{sh 0.300}%
\special{pa 980 540}%
\special{pa 1320 540}%
\special{pa 1320 100}%
\special{pa 980 100}%
\special{pa 980 540}%
\special{ip}%
\special{pn 8}%
\special{pa 980 540}%
\special{pa 1320 540}%
\special{pa 1320 100}%
\special{pa 980 100}%
\special{pa 980 540}%
\special{ip}%
%
\special{pn 0}%
\special{sh 0.300}%
\special{pa 1000 100}%
\special{pa 880 100}%
\special{pa 880 500}%
\special{pa 1000 500}%
\special{pa 1000 100}%
\special{ip}%
\special{pn 8}%
\special{pa 1000 100}%
\special{pa 880 100}%
\special{pa 880 500}%
\special{pa 1000 500}%
\special{pa 1000 100}%
\special{ip}%
\end{picture}}%
$$
\vspace{0.cm}

\noindent We define a  group  law on  $\pi^+_{\bullet,*}$  in the following way.
Given two properly immersed paths $r_1$ and $r_2$ from $\bullet$ to $*$,
consider the concatenation $r_1\,\nu^{-1}\,r_2$,
smooth its corners and insert a small counterclockwise curl;
then the resulting path represents the multiplication  of $r_1$ and $r_2$:
$$
{\unitlength 0.1in%
\begin{picture}(32.6000,5.3200)(1.2000,-8.2000)%
\put(4.6000,-5.7000){\makebox(0,0)[lb]{$r_1$}}%
%
\special{pn 13}%
\special{pa 120 800}%
\special{pa 920 800}%
\special{fp}%
%
\special{pn 4}%
\special{sh 1}%
\special{ar 280 800 16 16 0 6.2831853}%
\special{sh 1}%
\special{ar 760 800 16 16 0 6.2831853}%
\special{sh 1}%
\special{ar 760 800 16 16 0 6.2831853}%
\put(2.6600,-9.7200){\makebox(0,0)[lb]{$\bullet$}}%
\put(7.4600,-9.7200){\makebox(0,0)[lb]{$*$}}%
%
\special{pn 8}%
\special{pa 760 800}%
\special{pa 760 640}%
\special{fp}%
\special{pa 280 800}%
\special{pa 280 640}%
\special{fp}%
%
\special{pn 8}%
\special{pn 8}%
\special{pa 280 640}%
\special{pa 280 632}%
\special{fp}%
\special{pa 284 594}%
\special{pa 287 586}%
\special{fp}%
\special{pa 298 549}%
\special{pa 301 542}%
\special{fp}%
\special{pa 320 508}%
\special{pa 324 501}%
\special{fp}%
\special{pa 349 472}%
\special{pa 355 466}%
\special{fp}%
\special{pa 385 442}%
\special{pa 391 437}%
\special{fp}%
\special{pa 425 419}%
\special{pa 433 416}%
\special{fp}%
\special{pa 470 405}%
\special{pa 478 404}%
\special{fp}%
\special{pa 516 400}%
\special{pa 524 400}%
\special{fp}%
\special{pa 562 404}%
\special{pa 571 405}%
\special{fp}%
\special{pa 607 416}%
\special{pa 615 420}%
\special{fp}%
\special{pa 649 438}%
\special{pa 656 442}%
\special{fp}%
\special{pa 685 466}%
\special{pa 691 472}%
\special{fp}%
\special{pa 716 501}%
\special{pa 720 508}%
\special{fp}%
\special{pa 739 542}%
\special{pa 742 549}%
\special{fp}%
\special{pa 753 586}%
\special{pa 756 594}%
\special{fp}%
\special{pa 760 632}%
\special{pa 760 640}%
\special{fp}%
%
\special{pn 13}%
\special{pa 1000 800}%
\special{pa 1800 800}%
\special{fp}%
%
\special{pn 4}%
\special{sh 1}%
\special{ar 1160 800 16 16 0 6.2831853}%
\special{sh 1}%
\special{ar 1640 800 16 16 0 6.2831853}%
\special{sh 1}%
\special{ar 1640 800 16 16 0 6.2831853}%
\put(11.0600,-9.7200){\makebox(0,0)[lb]{$\bullet$}}%
\put(16.2600,-9.7200){\makebox(0,0)[lb]{$*$}}%
%
\special{pn 8}%
\special{pa 1640 800}%
\special{pa 1640 640}%
\special{fp}%
\special{pa 1160 800}%
\special{pa 1160 640}%
\special{fp}%
%
\special{pn 8}%
\special{pn 8}%
\special{pa 1160 640}%
\special{pa 1160 632}%
\special{fp}%
\special{pa 1164 596}%
\special{pa 1165 589}%
\special{fp}%
\special{pa 1176 554}%
\special{pa 1179 547}%
\special{fp}%
\special{pa 1195 515}%
\special{pa 1199 508}%
\special{fp}%
\special{pa 1221 480}%
\special{pa 1227 474}%
\special{fp}%
\special{pa 1254 450}%
\special{pa 1260 445}%
\special{fp}%
\special{pa 1291 426}%
\special{pa 1298 423}%
\special{fp}%
\special{pa 1331 410}%
\special{pa 1339 408}%
\special{fp}%
\special{pa 1374 401}%
\special{pa 1382 401}%
\special{fp}%
\special{pa 1418 401}%
\special{pa 1426 401}%
\special{fp}%
\special{pa 1461 408}%
\special{pa 1469 410}%
\special{fp}%
\special{pa 1502 423}%
\special{pa 1509 426}%
\special{fp}%
\special{pa 1540 445}%
\special{pa 1546 450}%
\special{fp}%
\special{pa 1573 474}%
\special{pa 1578 479}%
\special{fp}%
\special{pa 1600 508}%
\special{pa 1605 515}%
\special{fp}%
\special{pa 1621 547}%
\special{pa 1624 554}%
\special{fp}%
\special{pa 1635 589}%
\special{pa 1636 596}%
\special{fp}%
\special{pa 1640 632}%
\special{pa 1640 640}%
\special{fp}%
\put(13.5000,-5.7000){\makebox(0,0)[lb]{$r_2$}}%
%
\special{pn 13}%
\special{pa 2580 800}%
\special{pa 3380 800}%
\special{fp}%
%
\special{pn 4}%
\special{sh 1}%
\special{ar 2740 800 16 16 0 6.2831853}%
\special{sh 1}%
\special{ar 3220 800 16 16 0 6.2831853}%
\special{sh 1}%
\special{ar 3220 800 16 16 0 6.2831853}%
%
\special{pn 8}%
\special{pa 3220 800}%
\special{pa 3220 640}%
\special{fp}%
\special{pa 2740 800}%
\special{pa 2740 640}%
\special{fp}%
%
\special{pn 8}%
\special{ar 3100 720 40 40 4.7123890 3.1415927}%
%
\special{pn 8}%
\special{pn 8}%
\special{pa 2660 640}%
\special{pa 2660 632}%
\special{fp}%
\special{pa 2663 597}%
\special{pa 2665 590}%
\special{fp}%
\special{pa 2673 557}%
\special{pa 2675 550}%
\special{fp}%
\special{pa 2688 518}%
\special{pa 2692 511}%
\special{fp}%
\special{pa 2709 482}%
\special{pa 2713 477}%
\special{fp}%
\special{pa 2733 452}%
\special{pa 2738 446}%
\special{fp}%
\special{pa 2763 423}%
\special{pa 2769 418}%
\special{fp}%
\special{pa 2798 399}%
\special{pa 2804 395}%
\special{fp}%
\special{pa 2834 381}%
\special{pa 2841 378}%
\special{fp}%
\special{pa 2875 368}%
\special{pa 2882 366}%
\special{fp}%
\special{pa 2917 361}%
\special{pa 2925 360}%
\special{fp}%
\special{pa 2961 361}%
\special{pa 2969 361}%
\special{fp}%
\special{pa 3003 367}%
\special{pa 3011 369}%
\special{fp}%
\special{pa 3044 380}%
\special{pa 3050 383}%
\special{fp}%
\special{pa 3079 397}%
\special{pa 3086 401}%
\special{fp}%
\special{pa 3113 420}%
\special{pa 3119 425}%
\special{fp}%
\special{pa 3143 447}%
\special{pa 3147 452}%
\special{fp}%
\special{pa 3168 477}%
\special{pa 3172 483}%
\special{fp}%
\special{pa 3189 512}%
\special{pa 3192 518}%
\special{fp}%
\special{pa 3205 550}%
\special{pa 3207 557}%
\special{fp}%
\special{pa 3215 589}%
\special{pa 3217 597}%
\special{fp}%
\special{pa 3220 632}%
\special{pa 3220 640}%
\special{fp}%
%
\special{pn 8}%
\special{pn 8}%
\special{pa 2740 640}%
\special{pa 2740 632}%
\special{fp}%
\special{pa 2746 596}%
\special{pa 2748 589}%
\special{fp}%
\special{pa 2763 558}%
\special{pa 2767 551}%
\special{fp}%
\special{pa 2787 526}%
\special{pa 2793 521}%
\special{fp}%
\special{pa 2818 503}%
\special{pa 2824 499}%
\special{fp}%
\special{pa 2853 487}%
\special{pa 2860 485}%
\special{fp}%
\special{pa 2895 480}%
\special{pa 2904 480}%
\special{fp}%
\special{pa 2939 485}%
\special{pa 2946 487}%
\special{fp}%
\special{pa 2976 499}%
\special{pa 2982 502}%
\special{fp}%
\special{pa 3007 521}%
\special{pa 3012 526}%
\special{fp}%
\special{pa 3033 551}%
\special{pa 3037 557}%
\special{fp}%
\special{pa 3052 589}%
\special{pa 3054 596}%
\special{fp}%
\special{pa 3060 632}%
\special{pa 3060 640}%
\special{fp}%
%
\special{pn 8}%
\special{pa 520 400}%
\special{pa 480 360}%
\special{fp}%
\special{pa 520 400}%
\special{pa 480 440}%
\special{fp}%
%
\special{pn 8}%
\special{pa 1400 400}%
\special{pa 1360 360}%
\special{fp}%
\special{pa 1400 400}%
\special{pa 1360 440}%
\special{fp}%
%
\special{pn 8}%
\special{pa 2940 360}%
\special{pa 2900 320}%
\special{fp}%
\special{pa 2940 360}%
\special{pa 2900 400}%
\special{fp}%
%
\special{pn 8}%
\special{pa 2900 480}%
\special{pa 2860 440}%
\special{fp}%
\special{pa 2900 480}%
\special{pa 2860 520}%
\special{fp}%
\put(26.8600,-9.7200){\makebox(0,0)[lb]{$\bullet$}}%
\put(32.0600,-9.7200){\makebox(0,0)[lb]{$*$}}%
\put(28.0500,-6.4000){\makebox(0,0)[lb]{$r_1$}}%
\put(31.5000,-4.4000){\makebox(0,0)[lb]{$r_2$}}%
\put(20.7000,-7.2000){\makebox(0,0)[lb]{$r_1 r_2 :=$}}%
%
\special{pn 8}%
\special{pa 3100 680}%
\special{pa 2660 680}%
\special{fp}%
\special{pa 2660 640}%
\special{pa 2660 680}%
\special{fp}%
\special{pa 3060 640}%
\special{pa 3060 720}%
\special{fp}%
\end{picture}}%
$$
\vspace{0.cm}

\noindent
By abusing notation, let $\nu \in \pi^+_{\bullet,*}$ be represented by a simple path 
regularly homotopic to $\nu$: 
\[
{\unitlength 0.1in%
\begin{picture}(8.0000,2.6000)(1.2000,-8.2000)%
%
\special{pn 13}%
\special{pa 120 800}%
\special{pa 920 800}%
\special{fp}%
%
\special{pn 4}%
\special{sh 1}%
\special{ar 280 800 16 16 0 6.2831853}%
\special{sh 1}%
\special{ar 760 800 16 16 0 6.2831853}%
\special{sh 1}%
\special{ar 760 800 16 16 0 6.2831853}%
\put(2.6600,-9.7200){\makebox(0,0)[lb]{$\bullet$}}%
\put(7.4600,-9.7200){\makebox(0,0)[lb]{$*$}}%
%
\special{pn 8}%
\special{pa 520 600}%
\special{pa 480 560}%
\special{fp}%
\special{pa 520 600}%
\special{pa 480 640}%
\special{fp}%
%
\special{pn 8}%
\special{ar 520 760 240 160 3.1415927 6.2831853}%
%
\special{pn 8}%
\special{pa 280 800}%
\special{pa 280 760}%
\special{fp}%
%
\special{pn 8}%
\special{pa 760 760}%
\special{pa 760 800}%
\special{fp}%
\end{picture}}%
\]
\vspace{0.cm}

\noindent Note that $\nu$ is a unital element for the  group  law  on $\pi^+_{\bullet,*}$. 
  
 Let $\mathcal{P}(\Sigma,\bullet,*)$ be the free $\Q [\rho][[h]]$-module  generated by the set  $\pi^+_{\bullet,*}$ 
 modulo the relation that transforms a clockwise little curl to $\exp(\rho h)$. 
 For a properly immersed path $r\colon [0,1] \to \Sigma$ from $\bullet$ to $*$,
a framed oriented tangle $\varphi (r)$ is defined by
$$
\varphi (r)\colon [0,1] \times [-1,+1] \longrightarrow \Sigma \times [-1,+1], \quad 
(t,s) \longmapsto \big(r(t), {1/2-t+\epsilon s}\big)
$$
where $\epsilon>0$ is a number small enough.
This $\varphi$ induces a $\Q [\rho][[h]]$-linear map
$\varphi\colon \mathcal{P} (\Sigma,\bullet,*) \to \tskein (\Sigma, \bullet,*)$.

We also define a multiplication law on $\tskein (\Sigma,\bullet,*)$ in the following way. 
Given two tangles $T_1$ and $T_2$, stack $T_1$ over $T_2$ and, next, connect the final 
point of the string of $T_1$  with the  initial
point of the string of $T_2$ by a segment that monotonically goes down (between $T_1$ and $T_2$)
and projects onto $\nu^{-1}$ and, finally, insert a small negative kink: 
$$
{\unitlength 0.1in%
\begin{picture}(32.6000,5.3200)(1.2000,-8.2000)%
\put(4.3500,-5.9500){\makebox(0,0)[lb]{$T_1$}}%
%
\special{pn 13}%
\special{pa 120 800}%
\special{pa 920 800}%
\special{fp}%
%
\special{pn 4}%
\special{sh 1}%
\special{ar 280 800 16 16 0 6.2831853}%
\special{sh 1}%
\special{ar 760 800 16 16 0 6.2831853}%
\special{sh 1}%
\special{ar 760 800 16 16 0 6.2831853}%
\put(2.6600,-9.7200){\makebox(0,0)[lb]{$\bullet$}}%
\put(7.4600,-9.7200){\makebox(0,0)[lb]{$*$}}%
%
\special{pn 8}%
\special{pa 760 800}%
\special{pa 760 640}%
\special{fp}%
\special{pa 280 800}%
\special{pa 280 640}%
\special{fp}%
%
\special{pn 8}%
\special{pn 8}%
\special{pa 280 640}%
\special{pa 280 632}%
\special{fp}%
\special{pa 284 594}%
\special{pa 287 586}%
\special{fp}%
\special{pa 298 549}%
\special{pa 301 542}%
\special{fp}%
\special{pa 320 508}%
\special{pa 324 501}%
\special{fp}%
\special{pa 349 472}%
\special{pa 355 466}%
\special{fp}%
\special{pa 385 442}%
\special{pa 391 437}%
\special{fp}%
\special{pa 425 419}%
\special{pa 433 416}%
\special{fp}%
\special{pa 470 405}%
\special{pa 478 404}%
\special{fp}%
\special{pa 516 400}%
\special{pa 524 400}%
\special{fp}%
\special{pa 562 404}%
\special{pa 571 405}%
\special{fp}%
\special{pa 607 416}%
\special{pa 615 420}%
\special{fp}%
\special{pa 649 438}%
\special{pa 656 442}%
\special{fp}%
\special{pa 685 466}%
\special{pa 691 472}%
\special{fp}%
\special{pa 716 501}%
\special{pa 720 508}%
\special{fp}%
\special{pa 739 542}%
\special{pa 742 549}%
\special{fp}%
\special{pa 753 586}%
\special{pa 756 594}%
\special{fp}%
\special{pa 760 632}%
\special{pa 760 640}%
\special{fp}%
%
\special{pn 13}%
\special{pa 1000 800}%
\special{pa 1800 800}%
\special{fp}%
%
\special{pn 4}%
\special{sh 1}%
\special{ar 1160 800 16 16 0 6.2831853}%
\special{sh 1}%
\special{ar 1640 800 16 16 0 6.2831853}%
\special{sh 1}%
\special{ar 1640 800 16 16 0 6.2831853}%
\put(11.0600,-9.7200){\makebox(0,0)[lb]{$\bullet$}}%
\put(16.2600,-9.7200){\makebox(0,0)[lb]{$*$}}%
%
\special{pn 8}%
\special{pa 1640 800}%
\special{pa 1640 640}%
\special{fp}%
\special{pa 1160 800}%
\special{pa 1160 640}%
\special{fp}%
%
\special{pn 8}%
\special{pn 8}%
\special{pa 1160 640}%
\special{pa 1160 632}%
\special{fp}%
\special{pa 1164 596}%
\special{pa 1165 589}%
\special{fp}%
\special{pa 1176 554}%
\special{pa 1179 547}%
\special{fp}%
\special{pa 1195 515}%
\special{pa 1199 508}%
\special{fp}%
\special{pa 1221 480}%
\special{pa 1227 474}%
\special{fp}%
\special{pa 1254 450}%
\special{pa 1260 445}%
\special{fp}%
\special{pa 1291 426}%
\special{pa 1298 423}%
\special{fp}%
\special{pa 1331 410}%
\special{pa 1339 408}%
\special{fp}%
\special{pa 1374 401}%
\special{pa 1382 401}%
\special{fp}%
\special{pa 1418 401}%
\special{pa 1426 401}%
\special{fp}%
\special{pa 1461 408}%
\special{pa 1469 410}%
\special{fp}%
\special{pa 1502 423}%
\special{pa 1509 426}%
\special{fp}%
\special{pa 1540 445}%
\special{pa 1546 450}%
\special{fp}%
\special{pa 1573 474}%
\special{pa 1578 479}%
\special{fp}%
\special{pa 1600 508}%
\special{pa 1605 515}%
\special{fp}%
\special{pa 1621 547}%
\special{pa 1624 554}%
\special{fp}%
\special{pa 1635 589}%
\special{pa 1636 596}%
\special{fp}%
\special{pa 1640 632}%
\special{pa 1640 640}%
\special{fp}%
\put(13.2500,-5.9500){\makebox(0,0)[lb]{$T_2$}}%
%
\special{pn 13}%
\special{pa 2580 800}%
\special{pa 3380 800}%
\special{fp}%
%
\special{pn 4}%
\special{sh 1}%
\special{ar 2740 800 16 16 0 6.2831853}%
\special{sh 1}%
\special{ar 3220 800 16 16 0 6.2831853}%
\special{sh 1}%
\special{ar 3220 800 16 16 0 6.2831853}%
%
\special{pn 8}%
\special{pa 3220 800}%
\special{pa 3220 640}%
\special{fp}%
\special{pa 2740 800}%
\special{pa 2740 640}%
\special{fp}%
%
\special{pn 8}%
\special{ar 3100 720 40 40 4.7123890 3.1415927}%
%
\special{pn 8}%
\special{pn 8}%
\special{pa 2660 640}%
\special{pa 2660 632}%
\special{fp}%
\special{pa 2663 597}%
\special{pa 2665 590}%
\special{fp}%
\special{pa 2673 557}%
\special{pa 2675 550}%
\special{fp}%
\special{pa 2688 518}%
\special{pa 2692 511}%
\special{fp}%
\special{pa 2709 482}%
\special{pa 2713 477}%
\special{fp}%
\special{pa 2733 452}%
\special{pa 2738 446}%
\special{fp}%
\special{pa 2763 423}%
\special{pa 2769 418}%
\special{fp}%
\special{pa 2798 399}%
\special{pa 2804 395}%
\special{fp}%
\special{pa 2834 381}%
\special{pa 2841 378}%
\special{fp}%
\special{pa 2875 368}%
\special{pa 2882 366}%
\special{fp}%
\special{pa 2917 361}%
\special{pa 2925 360}%
\special{fp}%
\special{pa 2961 361}%
\special{pa 2969 361}%
\special{fp}%
\special{pa 3003 367}%
\special{pa 3011 369}%
\special{fp}%
\special{pa 3044 380}%
\special{pa 3050 383}%
\special{fp}%
\special{pa 3079 397}%
\special{pa 3086 401}%
\special{fp}%
\special{pa 3113 420}%
\special{pa 3119 425}%
\special{fp}%
\special{pa 3143 447}%
\special{pa 3147 452}%
\special{fp}%
\special{pa 3168 477}%
\special{pa 3172 483}%
\special{fp}%
\special{pa 3189 512}%
\special{pa 3192 518}%
\special{fp}%
\special{pa 3205 550}%
\special{pa 3207 557}%
\special{fp}%
\special{pa 3215 589}%
\special{pa 3217 597}%
\special{fp}%
\special{pa 3220 632}%
\special{pa 3220 640}%
\special{fp}%
%
\special{pn 8}%
\special{pn 8}%
\special{pa 2740 640}%
\special{pa 2740 632}%
\special{fp}%
\special{pa 2746 596}%
\special{pa 2748 589}%
\special{fp}%
\special{pa 2763 558}%
\special{pa 2767 551}%
\special{fp}%
\special{pa 2787 526}%
\special{pa 2793 521}%
\special{fp}%
\special{pa 2818 503}%
\special{pa 2824 499}%
\special{fp}%
\special{pa 2853 487}%
\special{pa 2860 485}%
\special{fp}%
\special{pa 2895 480}%
\special{pa 2904 480}%
\special{fp}%
\special{pa 2939 485}%
\special{pa 2946 487}%
\special{fp}%
\special{pa 2976 499}%
\special{pa 2982 502}%
\special{fp}%
\special{pa 3007 521}%
\special{pa 3012 526}%
\special{fp}%
\special{pa 3033 551}%
\special{pa 3037 557}%
\special{fp}%
\special{pa 3052 589}%
\special{pa 3054 596}%
\special{fp}%
\special{pa 3060 632}%
\special{pa 3060 640}%
\special{fp}%
%
\special{pn 8}%
\special{pa 520 400}%
\special{pa 480 360}%
\special{fp}%
\special{pa 520 400}%
\special{pa 480 440}%
\special{fp}%
%
\special{pn 8}%
\special{pa 1400 400}%
\special{pa 1360 360}%
\special{fp}%
\special{pa 1400 400}%
\special{pa 1360 440}%
\special{fp}%
%
\special{pn 8}%
\special{pa 2940 360}%
\special{pa 2900 320}%
\special{fp}%
\special{pa 2940 360}%
\special{pa 2900 400}%
\special{fp}%
%
\special{pn 8}%
\special{pa 2900 480}%
\special{pa 2860 440}%
\special{fp}%
\special{pa 2900 480}%
\special{pa 2860 520}%
\special{fp}%
\put(26.8600,-9.7200){\makebox(0,0)[lb]{$\bullet$}}%
\put(32.0600,-9.7200){\makebox(0,0)[lb]{$*$}}%
\put(28.0500,-6.6500){\makebox(0,0)[lb]{$T_1$}}%
\put(31.2500,-4.4000){\makebox(0,0)[lb]{$T_2$}}%
\put(20.2000,-7.2000){\makebox(0,0)[lb]{$T_1 T_2 :=$}}%
%
\special{pn 8}%
\special{pa 3060 640}%
\special{pa 3060 720}%
\special{fp}%
\special{pa 3100 680}%
\special{pa 3080 680}%
\special{fp}%
\special{pa 3040 680}%
\special{pa 2760 680}%
\special{fp}%
\special{pa 2720 680}%
\special{pa 2660 680}%
\special{fp}%
\special{pa 2660 680}%
\special{pa 2660 640}%
\special{fp}%
\end{picture}}%
$$
\vspace{0.cm}

\noindent
Then $\tskein(\Sigma,\bullet,*)$ becomes an associative algebra 
and the Lie action $\sigma$ of $\tskein(\Sigma)$ on $\tskein(\Sigma,\bullet,*)$ is an action by derivations.
Furthermore, the map $\varphi$ is an algebra homomorphism.

Finally, there is a closure operation $\vert- \vert\colon \tskein (\Sigma, \bullet, *) \to \tskein  (\Sigma )$  
which consists in connecting the two endpoints of the string of a tangle using an arc that projects onto $\nu^{-1}$: 
$$
{\unitlength 0.1in%
\begin{picture}(19.8000,4.4000)(1.2000,-5.6000)%
%
\special{pn 13}%
\special{pa 120 560}%
\special{pa 920 560}%
\special{fp}%
%
\special{pn 4}%
\special{sh 1}%
\special{ar 280 560 16 16 0 6.2831853}%
\special{sh 1}%
\special{ar 760 560 16 16 0 6.2831853}%
\special{sh 1}%
\special{ar 760 560 16 16 0 6.2831853}%
\put(2.6600,-6.9200){\makebox(0,0)[lb]{$\bullet$}}%
\put(7.4600,-6.9200){\makebox(0,0)[lb]{$*$}}%
%
\special{pn 13}%
\special{pa 1300 560}%
\special{pa 2100 560}%
\special{fp}%
%
\special{pn 4}%
\special{sh 1}%
\special{ar 1460 560 16 16 0 6.2831853}%
\special{sh 1}%
\special{ar 1940 560 16 16 0 6.2831853}%
\special{sh 1}%
\special{ar 1940 560 16 16 0 6.2831853}%
\put(14.4600,-6.9200){\makebox(0,0)[lb]{$\bullet$}}%
\put(19.2600,-6.9200){\makebox(0,0)[lb]{$*$}}%
%
\special{pn 8}%
\special{pa 280 560}%
\special{pa 280 400}%
\special{fp}%
\special{pa 760 560}%
\special{pa 760 400}%
\special{fp}%
%
\special{pn 8}%
\special{pn 8}%
\special{pa 280 400}%
\special{pa 280 392}%
\special{fp}%
\special{pa 284 354}%
\special{pa 286 346}%
\special{fp}%
\special{pa 298 309}%
\special{pa 301 301}%
\special{fp}%
\special{pa 320 268}%
\special{pa 324 261}%
\special{fp}%
\special{pa 349 232}%
\special{pa 355 226}%
\special{fp}%
\special{pa 384 202}%
\special{pa 391 197}%
\special{fp}%
\special{pa 425 180}%
\special{pa 433 176}%
\special{fp}%
\special{pa 470 165}%
\special{pa 478 164}%
\special{fp}%
\special{pa 516 160}%
\special{pa 524 160}%
\special{fp}%
\special{pa 562 164}%
\special{pa 570 166}%
\special{fp}%
\special{pa 607 176}%
\special{pa 615 179}%
\special{fp}%
\special{pa 649 197}%
\special{pa 656 202}%
\special{fp}%
\special{pa 685 226}%
\special{pa 691 232}%
\special{fp}%
\special{pa 716 261}%
\special{pa 720 268}%
\special{fp}%
\special{pa 739 302}%
\special{pa 742 309}%
\special{fp}%
\special{pa 754 346}%
\special{pa 755 354}%
\special{fp}%
\special{pa 760 392}%
\special{pa 760 400}%
\special{fp}%
%
\special{pn 8}%
\special{pn 8}%
\special{pa 1460 400}%
\special{pa 1460 392}%
\special{fp}%
\special{pa 1464 356}%
\special{pa 1466 348}%
\special{fp}%
\special{pa 1476 314}%
\special{pa 1479 307}%
\special{fp}%
\special{pa 1495 275}%
\special{pa 1499 269}%
\special{fp}%
\special{pa 1521 240}%
\special{pa 1527 234}%
\special{fp}%
\special{pa 1554 209}%
\special{pa 1560 205}%
\special{fp}%
\special{pa 1591 186}%
\special{pa 1597 183}%
\special{fp}%
\special{pa 1630 170}%
\special{pa 1638 168}%
\special{fp}%
\special{pa 1673 162}%
\special{pa 1681 161}%
\special{fp}%
\special{pa 1718 161}%
\special{pa 1726 161}%
\special{fp}%
\special{pa 1762 168}%
\special{pa 1769 170}%
\special{fp}%
\special{pa 1802 183}%
\special{pa 1809 186}%
\special{fp}%
\special{pa 1840 205}%
\special{pa 1846 209}%
\special{fp}%
\special{pa 1873 234}%
\special{pa 1879 240}%
\special{fp}%
\special{pa 1901 269}%
\special{pa 1905 275}%
\special{fp}%
\special{pa 1921 307}%
\special{pa 1924 314}%
\special{fp}%
\special{pa 1934 348}%
\special{pa 1936 356}%
\special{fp}%
\special{pa 1940 392}%
\special{pa 1940 400}%
\special{fp}%
%
\special{pn 8}%
\special{pa 1460 400}%
\special{pa 1460 440}%
\special{fp}%
\special{pa 1940 440}%
\special{pa 1940 400}%
\special{fp}%
\special{pa 1940 400}%
\special{pa 1940 400}%
\special{fp}%
%
\special{pn 8}%
\special{ar 1700 440 240 80 6.2831853 3.1415927}%
\put(4.0000,-4.0000){\makebox(0,0)[lb]{$T$}}%
\put(15.8000,-4.0000){\makebox(0,0)[lb]{$\lvert T \rvert$}}%
%
\special{pn 8}%
\special{pa 520 160}%
\special{pa 480 120}%
\special{fp}%
\special{pa 520 160}%
\special{pa 480 200}%
\special{fp}%
%
\special{pn 8}%
\special{pa 1720 160}%
\special{pa 1680 120}%
\special{fp}%
\special{pa 1720 160}%
\special{pa 1680 200}%
\special{fp}%
\end{picture}}%
$$
\vspace{0.cm}

For any integer $n\ge 1$, let $\varphi_n \colon \mathcal{P}(\Sigma,\bullet,*)^{\otimes n} \to \tskein (\Sigma)$
 be the $\Q[\rho][[h]]$-linear map defined by   $\varphi_n (r_1 \otimes \cdots \otimes r_n) :=\zettaiti{\varphi (r_1)} \cdots  \zettaiti{\varphi (r_n)}$.
There is a filtration of the module $\mathcal{P}(\Sigma,\bullet,*)$ defined in terms of the augmentation ideal of the group algebra of $\pi^+_{\bullet,*}$, and the map $\varphi_n$ naturally descends to completions: $\varphi_n \colon \widehat{\mathcal{P}}(\Sigma,\bullet,*)^{\widehat{\otimes} n} \to \ctskein (\Sigma)$.

We define the power series $\lambda^{[n]} (X_1, \cdots,X_n) \in \Q[[X_1-1,\ldots,X_n-1]]$ in $n$ variables,
 inductively, by setting $\lambda^{[1]} (X_1) := (1/2) (\log X_1)^2$ and
\begin{align*}
& \lambda^{[n+1]} (X_1 , \cdots, X_{n+1}) := \\
& \frac{X_1\cdots X_n\, \lambda^{[n]} (X_1, \cdots, X_n)-
X_2\cdots X_{n+1}\, \lambda^{[n]}(X_2, \cdots, X_{n+1})}{X_1-X_{n+1}}
\end{align*}
for $n\ge 1$.
Now, for any $r \in \pi^+_{\bullet,*}$, define the element $L_{\tskein}(r) \in \ctskein (\Sigma)$ by 
\begin{eqnarray}
\label{eq:LAr}  \quad L_{\tskein} (r) &:=& 
 \left( \frac{h/2}{\arcsinh (h/2)} \right)^2  \cdot \\
\nonumber && \left(\sum_{n=1}^\infty \frac{(-h)^{n-1}}{n \exp(n \rho h)} \varphi_n
(\lambda^{[n]} (r_{1,n}, \cdots, r_{n,n})) -\frac{1}{3}\rho^3 h^2 \right), 
\end{eqnarray}
where 
$r_{i,n} := {\nu}^{\otimes (i-1)} \otimes \exp (\rho h) r
\otimes  {\nu}^{\otimes n-i}$ and the evaluation $\lambda^{[n]} (r_{1,n}, \cdots, r_{n,n})$ 
is understood as substitution of $X_i - 1$ with $r_{i,n} - \nu^{\otimes n}$.

Finally, we associate to any simple closed curve $C$ in $\Sigma$ the element
\begin{equation} \label{LAC}
L_{\tskein} (C) := L_{\tskein} (r_C) \in \ctskein (\Sigma)
\end{equation}
where $r_C \in \pi^+_{\bullet,*}$ is a simple path such that  $\nu^{-1}r_C$
is homotopic to $C$.

\begin{thm}[{\cite[Theorem 5.2 \& Theorem 5.1]{TsujiHOMFLY-PTskein}}]
\label{thm_main_Dehn_tskein}
For any simple closed curve $C$ in $\Sigma$, 
the skein element $L_{\tskein} (C)$ is independent of the choice of $r_C$,
and the action of the Dehn twist along $C$ on the completed skein module $\ctskein (\Sigma, \bullet, *)$ 
coincides with the exponential of $\sigma ( L_{\tskein} (C))$:
\begin{equation*}
t_C =\exp \big(  \sigma (L_{\tskein} (C))  \big) 
\colon \ctskein (\Sigma,\bullet,*) \longrightarrow \ctskein(\Sigma,\bullet,*).
\end{equation*} 
\end{thm}

\subsection{Generalized Dehn twists on skein modules}

Based on Theorems \ref{thm_main_Dehn_skein} and \ref{thm_main_Dehn_tskein}, 
we define generalized Dehn twists on skein modules.
In both cases, there are two possible definitions 
which we shall refer to as ``geometric'' and ``algebraic'' versions. \\

\noindent
\textbf{$\star$ The case of the Kauffman bracket skein module.}\\

\noindent
Let $K$ be a framed unoriented knot in $\Sigma \times [-1,+1]$.
On the one hand, we can regard $K$ as an embedding of the annulus $R := S^1 \times [-1,+1]$ 
into $\Sigma \times [-1,+1]$ mapping  the core $C_R:= S^1 \times \{ 0\}$ to $K$, and we extend it 
to an embedding $e=e_K$ of the solid torus $R\times [-1,+1]$ 
into $\Sigma \times [-1,+1]$ whose image is a tubular neighborhood of $K$.
By Remark \ref{rem:skein_functorial}, there is an induced map $e_* \colon \skein(R) \to \skein(\Sigma)$ of skein algebras.
The core $C_R$ is a simple closed curve in $R$, and thus the element $L_{\skein}(C_R) \in \cskein (R)$ is defined.
Then we set
\[
L_{\skein}^{\rm geom}(K) := e_* (L_{\skein}(C_R)) \in \cskein (\Sigma).
\]

\begin{dfn}
Let 
$\ t^{\rm geom}_{\skein,K} := \exp \big( \sigma (L^{\rm geom}_{\skein} (K))\big)
\colon \cskein (\Sigma,\bullet,*) \to \cskein(\Sigma,\bullet,*).$
\end{dfn}

On the other hand, we can regard $K$ as an element of $\skein (\Sigma)$
and  use formula \eqref{eq:LKC} more directly to set 
\[
L^{\rm alg}_{\skein} (K) :=
\frac{-A+A^{-1}}{4 \log (-A)}
\left( \arccosh \left( \frac{-K}{2} \right) \right)^2
-(-A+A^{-1}) \log (-A).
\] 

\begin{dfn}
Let
$\ t^{\rm alg}_{\skein, K}:= \exp \big( \sigma (L^{\rm alg}_{\skein}(K))\big)
\colon \cskein (\Sigma,\bullet,*)
\to \cskein(\Sigma,\bullet,*).$
\end{dfn}

Both $t^{\rm geom}_{\skein,K}$ and $t^{\rm alg}_{\skein,K}$ 
are filtration-preserving automorphisms of $\cskein(\Sigma,\bullet,*)$.
Note that, if $K$ is given by a simple closed curve $C$ in $\Sigma$ as in Theorem \ref{thm_main_Dehn_skein}, 
then these automorphisms coincide and are induced by the classical Dehn twist along $C$.

We now discuss the relationship between the skein  versions of generalized Dehn twists 
and the original version \eqref{eq:tCexpgen}.
To compare them,  note the following two facts.
On the one hand, the automorphisms $t^{\rm geom}_{\skein,K}$ and $t^{\rm alg}_{\skein,K}$ 
induce automorphisms of $\cskein^{-1}(\Sigma,\bullet,*)$ in the natural way.
On the other hand, by using the Lie action $\sigma$ of $S'(\bGL)$ on $S'(\bGL) \otimes \Q \pi_{\bullet,*}$, one can define the automorphism $\exp(\sigma(L(\gamma)))$ of the $I$-adic completion of $S'(\bGL) \otimes \Q \pi_{\bullet,*}$ 
for any closed curve $\gamma$ in $\Sigma$.
By Proposition \ref{prop:psi}\,(iii) \& (iv), $\exp(\sigma(L(\gamma)))$ induces an automorphism of $\cskein^{-1}(\Sigma,\bullet,*)$.

\begin{prop}
\label{prop_gen_Dehn_twist_GS}
Let $K$ be a framed unoriented knot in $\Sigma \times [-1,+1]$ 
which projects onto a closed curve $\gamma$ in $\Sigma$.
Then,
\[
\psi(L(\gamma)) = \varpi( L^{\rm geom}_{\skein}(K)) = \varpi( L^{\rm alg}_{\skein}(K)) \in \cskein^{-1}(\Sigma)
\]
where the maps $\psi$ and $\varpi$ are as in \eqref{eq:comparison_abs}.
In particular, $t^{\rm geom}_{\skein,K}$ and $t^{\rm alg}_{\skein,K}$ induce the same automorphism of $\cskein^{-1}(\Sigma,\bullet,*)$ as the one induced by the automorphism $\exp(\sigma(L(\gamma)))$ 
of the $I$-adic completion of $S'(\GL) \otimes \Q\pi_{\bullet,*}$.
\end{prop}

\begin{proof}
In $\skein^{-1}(\Sigma)$, the class of a tangle depends only on its homotopy class.
Hence $\varpi( L^{\rm geom}_{\skein}(K)) = \varpi( L^{\rm alg}_{\skein}(K))$.
Note also, that for any $r\in \pi$, the element
$\zettaiti{(\log r)^2}= \zettaisq{(1/2) (\log r)^2}$
is in the completion of $\bGL$.
Since we have
\[
\left. \frac{-A+A^{-1}}{4 \log (-A)} \right|_{A=-1} = \frac{1}{2}
\]
by \eqref{eq:fraction}, it suffices to prove that
$\psi(\zettaisq{(1/2)(\log r)^2}) =  \big( \arccosh(\psi (\zettaisq{r})/2) \big)^2$. We compute
\[
\psi (\zettaisq{ (\log r)^2}) 
=
\psi \left(
\left\vert\! \left\vert
\left( \arccosh \big( \frac{r+r^{-1}}{2} \big) \right)^2
\right\vert\! \right\vert
\right)
= \psi \left( \left(
\arccosh \big( \frac{\zettaisq {r}}{2} \big)
\right)^2 \zettaisq{1} \right),
\]
where the second equality is a repeated use of the identity $\psi(\zettaisq{(r+r^{-1})x}) = \psi(\zettaisq{r} \zettaisq{x})$.
(See Proposition \ref{prop:psi}.(ii).)
Since $\psi$ is an algebra homomorphism and $\psi(\zettaisq{1}) = 2$, we obtain the assertion.
\end{proof}

\begin{rem} \label{rem:relation1}
The automorphism $\exp(\sigma(L(\gamma)))$ in Proposition \ref{prop_gen_Dehn_twist_GS} restricts to an automorphism 
of the $I$-adic completion of $\Q\pi_{\bullet,*}$.
If we identify $\pi_{\bullet,*}$ with the fundamental group $\pi$ by using the path $\nu$, this automorphism
(which is the one constructed from the exponential of the groupoid version \eqref{eq:sigma_bis} of the action $\sigma$) is identical with the generalized Dehn twist $t_{\gamma}$ on $\widehat{\Q \pi}$.\\
\end{rem}

\noindent
\textbf{$\star$ The case of the HOMFLY-PT skein module.}\\  

\noindent
On the one hand,  let $K$ be a framed unoriented knot  in $\Sigma \times [-1,+1]$.
As we did in  the case of the Kauffman bracket skein module, 
we consider a tubular neighborhood $e: {R \times [-1,+1] } \to \Sigma \times [-1,+1]$ of $K$ and set
\[
L_{\tskein}^{\rm geom}(K) := e_* (L_{\tskein}(C_R)) \in \ctskein (\Sigma),
\]
where $e_* \colon \ctskein(R) \to \ctskein(\Sigma)$ is the induced map of completed skein algebras
and $L_{\tskein}(C_R) \in \ctskein (R)$ is defined by \eqref{LAC}.

\begin{dfn}
Let $\ t^{\rm geom}_{\tskein,K} := \exp \big( \sigma (L^{\rm geom}_{\tskein} (K))\big)
\colon \ctskein (\Sigma,\bullet,*) \to\ctskein(\Sigma,\bullet,*)$.
\end{dfn}

On the other hand, let $r \in \pi^+_{\bullet,*}$.
Then, by \eqref{eq:LAr}, the skein element $L_{\tskein}(r) \in \ctskein (\Sigma)$ is defined.

\begin{dfn}
Let 
$\ t^{\rm alg}_{\tskein, r} :=\exp \big( \sigma (L_{\tskein} (r))\big)
\colon \ctskein (\Sigma,\bullet,*) \to \ctskein(\Sigma,\bullet,*)$.
\end{dfn}

Both $t^{\rm geom}_{\tskein,K}$ and $t^{\rm alg}_{\tskein, r}$ 
are filtration-preserving algebra automorphisms of $\ctskein(\Sigma,\bullet,*)$,
 where the algebra structure of $\ctskein(\Sigma,\bullet,*)$ is the one described in the second part of \S \ref{subsection:skein_Dehn}.
If $C$ is a simple closed curve in $\Sigma$ as in Theorem~\ref{thm_main_Dehn_tskein}, 
then both $t^{\rm geom}_{\tskein,K_C}$ and $t^{\rm alg}_{\tskein, r_C}$ coincide with
the automorphism induced by the classical Dehn twist along $C$, 
where  $K_C$ is  the framed knot in $\Sigma \times [-1,+1]$ 
corresponding to $C$ and $r_C \in \pi^+_{\bullet,*}$ is chosen as in \eqref{LAC}.

\begin{prop}
Let $K$ be a framed unoriented knot in $\Sigma \times [-1,+1]$ 
 which projects onto a closed curve $\gamma$ in $\Sigma$,
and let $r\in \pi^+_{\bullet,*}$ be such that $\gamma$ is homotopic to $\nu^{-1}r$.
Then,
\[
L(\gamma) = \varpi( L^{\rm geom}_{\tskein}(K)) = \varpi(L_{\tskein}(r)) \in \widehat{\GL}
\]
where the map $\varpi$ is as in \eqref{eq:comparison_abs}. 
In particular, the following diagram is commutative:
\[
\xymatrix@C=60pt{
\ctskein(\Sigma,\bullet,*) \ar[r]^{\text{$t^{\rm geom}_{\tskein,K}$ or $t^{\rm alg}_{\tskein, r}$}} \ar[d]^{\varpi} & \ctskein(\Sigma,\bullet,*) \ar[d]^{\varpi} \\
\widehat{S'}(\GL) \widehat{\otimes} \widehat{\Q \pi_{\bullet,*}} \ar[r]_{\exp(\sigma(L(\gamma)))} & \widehat{S'}(\GL) \widehat{\otimes} \widehat{\Q \pi_{\bullet,*}}
}
\]
\end{prop}

\begin{proof}
Since the image of a tangle by $\varpi$ depends only on its homotopy class, we have $\varpi( L^{\rm geom}_{\tskein}(K)) = \varpi(L_{\tskein}(r))$.
Furthermore, the map $\varpi$ maps $h$ to $0$:
hence, in the expression of $L_{\tskein}(r)$ in \eqref{eq:LAr}, all the terms for $n\ge 2$ do not contribute to $\varpi(L_{\tskein}(r))$.
Since we have 
\[
\left. \left( \frac{h/2}{\arcsinh (h/2)} \right)^2 \right|_{h=0} = 1
\]
and $\lambda^{[1]}(X_1) = (1/2)(\log X_1)^2$, we conclude that 
$\varpi(L_{\tskein}(r))$ is equal to $|(1/2)(\log \nu^{-1}r)^2| = L(\gamma)$.
\end{proof}

\subsection{The generalized Dehn twist along a figure eight}

In this subsection, we work with $\hskein (\Sigma)$, the Kauffman bracket skein algebra ``at $A=-1$'', and give an explicit description of the generalized Dehn twist along a figure eight for $\hskein (\Sigma)$.
Since all of the constructions in the previous subsections apply to any compact oriented surface with boundary  
(cf. Remark \ref{rem:skein_functorial}),
we localize the situation and consider a figure eight $\gamma$ in a pair of pants $\Sigma_{0,3}$, 
i.e$.$ a surface of genus $0$ with $3$ boundary components, as shown in Figure \ref{fig:paths_r5infty}.

We fix some notations:
take one point from each boundary component of $\Sigma_{0,3}$ 
and consider the fundamental groupoid based at those three points $\{ *_0, *_1, *_2\}$.
Then the paths $r_1$, $r_2$, $r_3$ and $r_4$ shown in Figure \ref{fig:paths_r1234} are generators for this groupoid.
Let $\pi_j := \pi_1(\Sigma_{0,3},*_j)$ and $\pi_{i,j} := \pi_1(\Sigma_{0,3},*_i,*_j)$ for any $i,j \in\{0,1,2\}$.
Diagrams \eqref{eq:comparison_abs} and \eqref{eq:comparison_rel} exist also in this setting:
in particular, we have a surjective map of Lie algebras   
$\psi\colon S'(\mathbb{Q}\, \vert\!\vert \pi_j \vert\!\vert) \to \hskein (\Sigma_{0,3})$ 
and a surjective map of Lie modules $\psi_{i,j} \colon S'(\mathbb{Q}\, \vert\!\vert \pi_j \vert\!\vert) \otimes \Q \pi_{i,j} 
\to \hskein(\Sigma_{0,3},*_i,*_j)$.

\begin{figure}[ht]
\begin{minipage}{0.45\hsize}
  \centering
{\unitlength 0.1in%
\begin{picture}(18.0000,12.0000)(2.0000,-14.0000)%
%
\special{pn 13}%
\special{ar 1100 800 900 600 0.0000000 6.2831853}%
%
\special{pn 13}%
\special{ar 650 800 225 225 0.0000000 6.2831853}%
%
\special{pn 13}%
\special{ar 1550 800 225 225 0.0000000 6.2831853}%
\put(9.7200,-11.9800){\makebox(0,0)[lb]{$\gamma$}}%
%
\special{pn 4}%
\special{sh 1}%
\special{ar 1100 1400 16 16 0 6.2831853}%
%
\special{pn 4}%
\special{sh 1}%
\special{ar 650 1025 16 16 0 6.2831853}%
%
\special{pn 4}%
\special{sh 1}%
\special{ar 1550 1025 16 16 0 6.2831853}%
%
\special{pn 8}%
\special{ar 650 800 300 375 1.5707963 4.7123890}%
%
\special{pn 8}%
\special{ar 1550 800 300 375 4.7123890 1.5707963}%
%
\special{pn 8}%
\special{pa 1550 1175}%
\special{pa 1518 1176}%
\special{pa 1486 1176}%
\special{pa 1455 1172}%
\special{pa 1426 1166}%
\special{pa 1398 1156}%
\special{pa 1371 1143}%
\special{pa 1346 1127}%
\special{pa 1322 1108}%
\special{pa 1298 1087}%
\special{pa 1276 1065}%
\special{pa 1254 1040}%
\special{pa 1233 1013}%
\special{pa 1213 985}%
\special{pa 1193 956}%
\special{pa 1173 925}%
\special{pa 1154 894}%
\special{pa 1136 862}%
\special{pa 1117 830}%
\special{pa 1100 800}%
\special{fp}%
%
\special{pn 8}%
\special{pa 1550 425}%
\special{pa 1518 424}%
\special{pa 1486 424}%
\special{pa 1455 428}%
\special{pa 1426 434}%
\special{pa 1398 444}%
\special{pa 1371 457}%
\special{pa 1346 473}%
\special{pa 1322 492}%
\special{pa 1298 513}%
\special{pa 1276 535}%
\special{pa 1254 560}%
\special{pa 1233 587}%
\special{pa 1213 615}%
\special{pa 1193 644}%
\special{pa 1173 675}%
\special{pa 1154 706}%
\special{pa 1136 738}%
\special{pa 1117 770}%
\special{pa 1100 800}%
\special{fp}%
%
\special{pn 8}%
\special{pa 650 425}%
\special{pa 682 424}%
\special{pa 714 424}%
\special{pa 745 428}%
\special{pa 774 434}%
\special{pa 802 444}%
\special{pa 829 457}%
\special{pa 854 473}%
\special{pa 878 492}%
\special{pa 902 513}%
\special{pa 924 535}%
\special{pa 946 560}%
\special{pa 967 587}%
\special{pa 987 615}%
\special{pa 1007 644}%
\special{pa 1027 675}%
\special{pa 1046 706}%
\special{pa 1064 738}%
\special{pa 1083 770}%
\special{pa 1100 800}%
\special{fp}%
%
\special{pn 8}%
\special{pa 650 1175}%
\special{pa 682 1176}%
\special{pa 714 1176}%
\special{pa 745 1172}%
\special{pa 774 1166}%
\special{pa 802 1156}%
\special{pa 829 1143}%
\special{pa 854 1127}%
\special{pa 878 1108}%
\special{pa 902 1087}%
\special{pa 924 1065}%
\special{pa 946 1040}%
\special{pa 967 1013}%
\special{pa 987 985}%
\special{pa 1007 956}%
\special{pa 1027 925}%
\special{pa 1046 894}%
\special{pa 1064 862}%
\special{pa 1083 830}%
\special{pa 1100 800}%
\special{fp}%
\end{picture}}%
\caption{$\gamma$}
\label{fig:paths_r5infty}
\end{minipage}
\begin{minipage}{0.45\hsize}
  \centering
{\unitlength 0.1in%
\begin{picture}(18.0000,12.1700)(2.0000,-14.1700)%
%
\special{pn 13}%
\special{ar 1100 800 900 600 0.0000000 6.2831853}%
%
\special{pn 13}%
\special{ar 650 800 225 225 0.0000000 6.2831853}%
%
\special{pn 13}%
\special{ar 1550 800 225 225 0.0000000 6.2831853}%
%
\special{pn 8}%
\special{pa 1100 1400}%
\special{pa 650 1025}%
\special{fp}%
%
\special{pn 8}%
\special{pa 1100 1400}%
\special{pa 1550 1025}%
\special{fp}%
\special{pa 1550 1025}%
\special{pa 1550 1025}%
\special{fp}%
%
\special{pn 8}%
\special{pa 1362 1175}%
\special{pa 1362 1175}%
\special{fp}%
\special{pa 1288 1175}%
\special{pa 1362 1175}%
\special{fp}%
\special{pa 1362 1250}%
\special{pa 1362 1250}%
\special{fp}%
%
\special{pn 8}%
\special{pa 1362 1250}%
\special{pa 1362 1175}%
\special{fp}%
%
\special{pn 8}%
\special{pa 838 1175}%
\special{pa 838 1258}%
\special{fp}%
\special{pa 838 1182}%
\special{pa 912 1182}%
\special{fp}%
\special{pa 912 1182}%
\special{pa 912 1182}%
\special{fp}%
\put(6.3800,-12.5000){\makebox(0,0)[lb]{$r_1$}}%
\put(14.3900,-12.5000){\makebox(0,0)[lb]{$r_2$}}%
\put(8.6400,-5.0000){\makebox(0,0)[lb]{$r_3$}}%
\put(12.1400,-5.0000){\makebox(0,0)[lb]{$r_4$}}%
%
\special{pn 4}%
\special{sh 1}%
\special{ar 1100 1400 16 16 0 6.2831853}%
%
\special{pn 4}%
\special{sh 1}%
\special{ar 650 1025 16 16 0 6.2831853}%
\special{sh 1}%
\special{ar 1550 1025 16 16 0 6.2831853}%
%
\special{pn 8}%
\special{ar 650 800 338 338 3.1415927 4.7123890}%
%
\special{pn 8}%
\special{ar 1550 800 338 338 3.1415927 4.7123890}%
%
\special{pn 8}%
\special{ar 1550 800 338 338 4.7123890 6.2831853}%
%
\special{pn 8}%
\special{ar 650 800 338 338 4.7123890 6.2831853}%
%
\special{pn 8}%
\special{pa 312 792}%
\special{pa 311 824}%
\special{pa 311 856}%
\special{pa 315 889}%
\special{pa 323 923}%
\special{pa 335 956}%
\special{pa 350 987}%
\special{pa 369 1015}%
\special{pa 391 1039}%
\special{pa 415 1057}%
\special{pa 441 1068}%
\special{pa 470 1073}%
\special{pa 500 1072}%
\special{pa 532 1066}%
\special{pa 565 1057}%
\special{pa 598 1045}%
\special{pa 632 1032}%
\special{pa 650 1025}%
\special{fp}%
%
\special{pn 8}%
\special{pa 1212 792}%
\special{pa 1211 824}%
\special{pa 1211 856}%
\special{pa 1215 889}%
\special{pa 1223 923}%
\special{pa 1235 956}%
\special{pa 1250 987}%
\special{pa 1269 1015}%
\special{pa 1291 1039}%
\special{pa 1315 1057}%
\special{pa 1341 1068}%
\special{pa 1370 1073}%
\special{pa 1400 1072}%
\special{pa 1432 1066}%
\special{pa 1465 1057}%
\special{pa 1498 1045}%
\special{pa 1532 1032}%
\special{pa 1550 1025}%
\special{fp}%
%
\special{pn 8}%
\special{pa 1888 792}%
\special{pa 1889 824}%
\special{pa 1889 856}%
\special{pa 1885 889}%
\special{pa 1877 923}%
\special{pa 1865 956}%
\special{pa 1850 987}%
\special{pa 1831 1015}%
\special{pa 1809 1039}%
\special{pa 1785 1057}%
\special{pa 1759 1068}%
\special{pa 1730 1073}%
\special{pa 1700 1072}%
\special{pa 1668 1066}%
\special{pa 1635 1057}%
\special{pa 1602 1045}%
\special{pa 1568 1032}%
\special{pa 1550 1025}%
\special{fp}%
%
\special{pn 8}%
\special{pa 988 792}%
\special{pa 989 824}%
\special{pa 989 856}%
\special{pa 985 889}%
\special{pa 977 923}%
\special{pa 965 956}%
\special{pa 950 987}%
\special{pa 931 1015}%
\special{pa 909 1039}%
\special{pa 885 1057}%
\special{pa 859 1068}%
\special{pa 830 1073}%
\special{pa 800 1072}%
\special{pa 768 1066}%
\special{pa 735 1057}%
\special{pa 702 1045}%
\special{pa 668 1032}%
\special{pa 650 1025}%
\special{fp}%
%
\special{pn 8}%
\special{pa 312 792}%
\special{pa 350 868}%
\special{fp}%
\special{pa 312 792}%
\special{pa 275 868}%
\special{fp}%
%
\special{pn 8}%
\special{pa 1212 792}%
\special{pa 1250 868}%
\special{fp}%
\special{pa 1212 792}%
\special{pa 1175 868}%
\special{fp}%
\put(6.0200,-9.8800){\makebox(0,0)[lb]{$*_1$}}%
\put(15.0000,-9.8500){\makebox(0,0)[lb]{$*_2$}}%
\put(10.6200,-15.6200){\makebox(0,0)[lb]{$*_0$}}%
\end{picture}}%
\caption{$r_1,r_2,r_3,r_4$}
\label{fig:paths_r1234}
\end{minipage}
\end{figure}

In order to describe the generalized Dehn twist along $\gamma$ as an automorphism of $\hskein(\Sigma_{0,3},*_i,*_j)$ for any $i,j$, 
it is enough to determine how it acts on each of $\psi_{0,1}(1\otimes r_1)$, $\psi_{0,2}(1\otimes r_2)$,
$\psi_{1,1}(1\otimes r_3)$ and $\psi_{2,2}(1\otimes r_4)$.
The last two ones are fixed by the generalized Dehn twist since $r_3$ and $r_4$ are disjoint from $\gamma$.
Let us compute the image of $\psi_{0,1}(1\otimes r_1)$
which, by Proposition~\ref{prop_gen_Dehn_twist_GS}, 
is equal to $\psi(t_{\gamma}(r_1)):=\psi_{0,1}(1 \widehat{\otimes} t_{\gamma}(r_1))$. 
The computation of the image of $\psi_{0,2}(1\otimes r_2)$ being similar, it is omitted here.

Firstly, we compute
\begin{align}
\label{eq:huit}& 
{\unitlength 0.1in%
\begin{picture}(36.4000,8.0000)(3.2000,-10.0000)%
%
\special{pn 13}%
\special{ar 1800 600 600 400 0.0000000 6.2831853}%
%
\special{pn 13}%
\special{ar 1500 600 150 150 0.0000000 6.2831853}%
%
\special{pn 13}%
\special{ar 2100 600 150 150 0.0000000 6.2831853}%
%
\special{pn 4}%
\special{sh 1}%
\special{ar 1800 1000 16 16 0 6.2831853}%
%
\special{pn 4}%
\special{sh 1}%
\special{ar 1500 750 16 16 0 6.2831853}%
\special{sh 1}%
\special{ar 2100 750 16 16 0 6.2831853}%
%
\special{pn 8}%
\special{ar 2099 596 201 251 4.7123890 1.5707963}%
%
\special{pn 8}%
\special{pa 2099 847}%
\special{pa 2067 848}%
\special{pa 2036 845}%
\special{pa 2007 838}%
\special{pa 1980 825}%
\special{pa 1955 809}%
\special{pa 1931 789}%
\special{pa 1909 765}%
\special{pa 1888 739}%
\special{pa 1867 711}%
\special{pa 1848 681}%
\special{pa 1829 650}%
\special{pa 1810 617}%
\special{pa 1798 596}%
\special{fp}%
%
\special{pn 8}%
\special{pa 2099 345}%
\special{pa 2067 344}%
\special{pa 2036 347}%
\special{pa 2007 354}%
\special{pa 1980 367}%
\special{pa 1955 383}%
\special{pa 1931 403}%
\special{pa 1909 427}%
\special{pa 1888 453}%
\special{pa 1867 481}%
\special{pa 1848 511}%
\special{pa 1829 542}%
\special{pa 1810 575}%
\special{pa 1798 596}%
\special{fp}%
%
\special{pn 8}%
\special{pa 1496 345}%
\special{pa 1528 344}%
\special{pa 1559 347}%
\special{pa 1588 354}%
\special{pa 1615 366}%
\special{pa 1640 383}%
\special{pa 1664 403}%
\special{pa 1686 426}%
\special{pa 1708 452}%
\special{pa 1728 481}%
\special{pa 1748 511}%
\special{pa 1767 542}%
\special{pa 1785 574}%
\special{pa 1798 596}%
\special{fp}%
%
\special{pn 13}%
\special{ar 3360 600 600 400 0.0000000 6.2831853}%
%
\special{pn 13}%
\special{ar 3060 600 150 150 0.0000000 6.2831853}%
%
\special{pn 13}%
\special{ar 3660 600 150 150 0.0000000 6.2831853}%
%
\special{pn 4}%
\special{sh 1}%
\special{ar 3360 1000 16 16 0 6.2831853}%
%
\special{pn 4}%
\special{sh 1}%
\special{ar 3060 750 16 16 0 6.2831853}%
\special{sh 1}%
\special{ar 3660 750 16 16 0 6.2831853}%
%
\special{pn 8}%
\special{ar 3056 596 201 251 1.5707963 4.7123890}%
%
\special{pn 8}%
\special{ar 3659 596 201 251 4.7123890 1.5707963}%
%
\special{pn 8}%
\special{pa 3659 847}%
\special{pa 3627 848}%
\special{pa 3596 845}%
\special{pa 3567 838}%
\special{pa 3540 825}%
\special{pa 3515 809}%
\special{pa 3491 789}%
\special{pa 3469 765}%
\special{pa 3448 739}%
\special{pa 3427 711}%
\special{pa 3408 681}%
\special{pa 3389 650}%
\special{pa 3370 617}%
\special{pa 3358 596}%
\special{fp}%
%
\special{pn 8}%
\special{pa 3659 345}%
\special{pa 3627 344}%
\special{pa 3596 347}%
\special{pa 3567 354}%
\special{pa 3540 367}%
\special{pa 3515 383}%
\special{pa 3491 403}%
\special{pa 3469 427}%
\special{pa 3448 453}%
\special{pa 3427 481}%
\special{pa 3408 511}%
\special{pa 3389 542}%
\special{pa 3370 575}%
\special{pa 3358 596}%
\special{fp}%
%
\special{pn 8}%
\special{pa 3056 345}%
\special{pa 3088 344}%
\special{pa 3119 347}%
\special{pa 3148 354}%
\special{pa 3175 366}%
\special{pa 3200 383}%
\special{pa 3224 403}%
\special{pa 3246 426}%
\special{pa 3268 452}%
\special{pa 3288 481}%
\special{pa 3308 511}%
\special{pa 3327 542}%
\special{pa 3345 574}%
\special{pa 3358 596}%
\special{fp}%
%
\special{pn 8}%
\special{ar 1500 600 200 256 3.1415927 4.7123890}%
%
\special{pn 8}%
\special{pa 1300 600}%
\special{pa 1297 633}%
\special{pa 1297 665}%
\special{pa 1305 694}%
\special{pa 1320 720}%
\special{pa 1342 745}%
\special{pa 1372 769}%
\special{pa 1405 791}%
\special{pa 1439 805}%
\special{pa 1467 807}%
\special{pa 1487 791}%
\special{pa 1498 759}%
\special{pa 1500 752}%
\special{fp}%
%
\special{pn 8}%
\special{pa 1800 596}%
\special{pa 1786 625}%
\special{pa 1771 653}%
\special{pa 1757 682}%
\special{pa 1743 710}%
\special{pa 1729 739}%
\special{pa 1714 768}%
\special{pa 1698 799}%
\special{pa 1686 830}%
\special{pa 1680 861}%
\special{pa 1683 888}%
\special{pa 1695 914}%
\special{pa 1715 937}%
\special{pa 1742 959}%
\special{pa 1772 980}%
\special{pa 1800 998}%
\special{fp}%
%
\special{pn 8}%
\special{pa 3358 598}%
\special{pa 3343 626}%
\special{pa 3328 655}%
\special{pa 3316 684}%
\special{pa 3304 714}%
\special{pa 3289 743}%
\special{pa 3269 768}%
\special{pa 3241 788}%
\special{pa 3209 801}%
\special{pa 3175 808}%
\special{pa 3142 806}%
\special{pa 3111 796}%
\special{pa 3086 777}%
\special{pa 3062 754}%
\special{pa 3060 752}%
\special{fp}%
%
\special{pn 8}%
\special{pa 3052 846}%
\special{pa 3084 850}%
\special{pa 3115 855}%
\special{pa 3146 863}%
\special{pa 3175 874}%
\special{pa 3204 888}%
\special{pa 3231 903}%
\special{pa 3258 921}%
\special{pa 3285 940}%
\special{pa 3311 960}%
\special{pa 3337 981}%
\special{pa 3360 1000}%
\special{fp}%
%
\special{pn 8}%
\special{pa 1924 408}%
\special{pa 1918 452}%
\special{fp}%
\special{pa 1924 408}%
\special{pa 1885 429}%
\special{fp}%
%
\special{pn 8}%
\special{pa 1500 346}%
\special{pa 1540 326}%
\special{fp}%
\special{pa 1500 346}%
\special{pa 1540 366}%
\special{fp}%
%
\special{pn 8}%
\special{pa 3080 346}%
\special{pa 3040 326}%
\special{fp}%
\special{pa 3080 346}%
\special{pa 3040 366}%
\special{fp}%
%
\special{pn 8}%
\special{pa 3458 442}%
\special{pa 3464 398}%
\special{fp}%
\special{pa 3458 442}%
\special{pa 3497 421}%
\special{fp}%
\put(25.2000,-6.5000){\makebox(0,0)[lb]{$-$}}%
\put(3.2000,-6.5000){\makebox(0,0)[lb]{$\sigma( \lvert\!\lvert \gamma \rvert\!\rvert)(r_1)=$}}%
\end{picture}}%
 \\
\notag & \hspace{4.4em} = r_2 r_4 r_2^{-1} r_1 r_3^{-1}
-r_1 r_3 r_1^{-1} r_2 r_4^{-1} r_2^{-1} r_1
\end{align}
and we observe that  
$\gamma$ is the free homotopy class of $r_1 r_3 r_1^{-1}r_2 r_4^{-1} r_2^{-1}$. 
Next, we claim that 
\begin{equation} \label{eq:psisigmagamma}
\psi \big( \sigma(\zettaisq{\gamma})(r_1) \big)
=\psi \big( (r_5-r_5^{-1})r_1 -
\zettaisq{r_4} \otimes r_1(r_3-r_3^{-1}) \big)
\end{equation}
where $r_5 := r_1r_3r_1^{-1}r_2r_4r_2^{-1}$ is the loop parallel to the boundary component based at $*_0$
and we use the following shorthand notations: $\psi:= \psi_{0,1}$ and 
$r:=1\otimes r \in S'(\mathbb{Q}\, \vert\!\vert \pi_1 \vert\!\vert) \otimes \Q \pi_{0,1}$
for any $r\in \Q \pi_{0,1}$. 
To justify \eqref{eq:psisigmagamma}, observe that \eqref{eq:huit} implies 
\begin{align*}
\sigma(\zettaisq{\gamma})(r_1) &=
r_2(r_4 + r_4^{-1})r_2^{-1}r_1r_3^{-1}
-r_2r_4^{-1}r_2^{-1}r_1r_3^{-1} \\
& \quad
-r_1r_3r_1^{-1}r_2(r_4 + r_4^{-1})r_2^{-1}r_1
+r_1r_3r_1^{-1}r_2r_4r_2^{-1}r_1 \\
&\equiv
\zettaisq{r_4} \otimes r_1r_3^{-1} - r_5^{-1} r_1
- \zettaisq{r_4} \otimes r_1r_3 + r_5 r_1 \\
&= (r_5 - r_5^{-1}) r_1 - \zettaisq{r_4} \otimes r_1(r_3 - r_3^{-1})
\end{align*} 
where the second identity is modulo the kernel of $\psi$ (as it is described by Proposition \ref{prop:psi}.(iv)). 
To continue, note that  
$$
\sigma(x^n)(\psi(y)) = n x^{n-1} \sigma(x)(\psi(y))
$$
for any  $n\ge 0$, $x\in \hskein(\Sigma_{0,3})$ and any path $y$ from $*_0$ to $*_1$;
besides Proposition \ref{prop_gen_Dehn_twist_GS} implies that
 $\psi(L(\gamma)) = \psi \big( (1/2) (\arccosh ( \zettaisq{\gamma}/2))^2 \big)$. 
Hence \eqref{eq:psisigmagamma} implies that the action $\psi(\sigma(L(\gamma))(r_1)) = \sigma(\psi(L(\gamma))(\psi(r_1))$ 
is described by using the derivative 
\[
\chi(x) := \frac{d}{dx} \left(
\frac{1}{2} \left(\arccosh \left( \frac{x}{2} \right) \right)^2 \right)
= \frac{x}{4} \sum_{i=0}^{\infty}
\frac{i!i!}{(2i+1)!} \left( 4-x^2 \right)^{i}.
\]
More explicitly, we obtain
\begin{align*}
& \psi \big( \sigma(L(\gamma))(r_1) \big) \\
=\  & \psi( \chi(\zettaisq{\gamma}))
\psi \big( (r_5-r_5^{-1}) r_1
-\zettaisq{r_4} \otimes r_1(r_3-r_3^{-1}) \big) \\
=\ & \psi \big(
(\chi(\zettaisq{\gamma}) \otimes (r_5 - r_5^{-1}) ) \cdot r_1 - r_1 \cdot (\chi(\zettaisq{\gamma}) \zettaisq{r_4} \otimes (r_3 - r_3^{-1}) ) \big).
\end{align*}
The elements $\psi(\zettaisq{\gamma})$,  $\psi(\zettaisq{r_4})$, 
$\psi(r_3)$ and $\psi(r_5)$ are annihilated by $\sigma( \psi(L(\gamma)))$.
Thus we can compute the action of the exponential of $\sigma(\psi(L(\gamma)))$ on $\psi(r_1)$ 
and conclude that 
\begin{align*}
&\psi( t_{\gamma} (r_1))
= \psi \big( \exp( \sigma( L(\gamma))) (r_1) \big)
= \exp  \big(  \sigma( \psi(L(\gamma)))  \big) (\psi(r_1)) \\
&  =
\psi \left(
\exp ( \chi(\zettaisq{\gamma}) \otimes (r_5 - r_5^{-1}) )\cdot r_1 \cdot
\exp( - \chi(\zettaisq{\gamma}) \zettaisq{r_4} \otimes (r_3 - r_3^{-1}) ) \right). 
\end{align*}

\subsection{Invariants of integral homology $3$-spheres}

Using the formulas for the action of Dehn twists on skein modules  
(Theorems \ref{thm_main_Dehn_skein} and \ref{thm_main_Dehn_tskein}),
we can construct invariants of integral homology $3$-spheres.

We first explain constructions of embeddings of the \emph{Torelli group} $\calI := \calI(\Sigma)$ of $\Sigma$ 
into the completions of the skein algebras $\skein(\Sigma)$ and $\tskein(\Sigma)$.
For this, we assume that the genus $g$ of $\Sigma$ is at least three. 
Recall that $\calI$ is the kernel of the action of the mapping class group $\calM$ on $H_1(\Sigma;\Z)$.
Based on earlier results by Birman \cite{Birman1971} and Powell \cite{Powell1978}, Johnson \cite{Johnson79} proved that the group $\calI$ is generated by elements of the form $t_{C_1} {t_{C_2}}^{-1}$, where $C_1$ and $C_2$ are disjoint non-separating simple closed curves in $\Sigma$ such that $C_1 \cup C_2$ bounds a subsurface of $\Sigma$.
Such a couple $(C_1,C_2)$ is called a \emph{bounding pair}.

The Lie brackets on $\skein(\Sigma)$ and $\tskein(\Sigma)$ 
extend to the completions $\cskein(\Sigma)$ and $\ctskein(\Sigma)$, respectively.
Since these Lie brackets are of degree $(-2)$,  
the third terms of the filtrations $\filt{3} \cskein (\Sigma)$ and $\filt{3} \ctskein (\Sigma)$ are pronilpotent Lie algebras. 
Hence, by using the Baker--Campbell--Hausdorff series, we can regard them as pronilpotent groups.

\begin{thm}[{\cite[Theorem 3.13 \&  Corollary 3.15]{TsujiTorelli}}] \label{thm:Torelli_S}
Assigning the skein element $\zeta_{\skein} (t_{C_1} {t_{C_2}}^{-1})
:=L_{\skein} (C_1)- L_{\skein} (C_2)$ to any bounding pair $(C_1,C_2)$
defines an injective group homomorphism
\[
\zeta_{\skein} \colon \calI \longrightarrow \filt{3} \cskein (\Sigma).
\]
Furthermore, for any $\xi \in \torelli$,
we have
\[
\xi =
\exp \big(\sigma( \zeta_{\skein}(\xi))\big) \colon
\cskein (\Sigma, \bullet,*) \longrightarrow
\cskein (\Sigma,\bullet,*).
\]
\end{thm}

\begin{thm}[{\cite[Theorem 7.13 \& Corollary 7.14]{TsujiHOMFLY-PTskein}}] \label{thm:Torelli_A}
Assigning the skein element $\zeta_{\tskein} (t_{C_1} {t_{C_2}}^{-1})
:=L_{\tskein} (C_1)- L_{\tskein} (C_2)$ to any bounding pair $(C_1,C_2)$
defines an injective group homomorphism
\[
\zeta_{\tskein} \colon \calI \longrightarrow \filt{3} \ctskein (\Sigma).
\]
Furthermore, for any $\xi \in \torelli$, we have
\[
\xi =
\exp \big(\sigma( \zeta_{\tskein}(\xi))\big) \colon
\ctskein (\Sigma, \bullet,*) \longrightarrow
\ctskein (\Sigma,\bullet,*).
\]
\end{thm}

The proofs of the above two theorems use the infinite presentation of the Torelli group by Putman \cite{Pu2008}.

\begin{rem} \label{rem:BSCC}
For any separating simple closed curve $C$ in $\Sigma$, we have
\begin{equation} \label{eq:zetaBSCC}
\zeta_{\skein}(t_C) = L_{\skein}(C)
\quad \text{and} \quad
\zeta_{\tskein}(t_C) = L_{\tskein}(C).
\end{equation}
If the genus $g$ is $1$ or $2$, twists along separating simple closed curves are also needed 
among the generators of $\torelli$ and one can still define the homomorphisms $\zeta_{\skein}$ and $\zeta_{\tskein}$ 
taking \eqref{eq:zetaBSCC} as definition on those generators.
\end{rem}

We now turn our attention to oriented integral homology $3$-spheres. 
We fix a Heegaard splitting of $S^3=H_g^+ \cup_\iota H_g^-$,
where $H_g^+$ and $H_g^-$ are handlebodies of genus $g$
and $\iota$ is an orientation-reversing diffeomorphism
from $\partial H_g^+$ to~$\partial H_g^-$.
We regard the surface $\Sigma$ as being obtained by deleting the interior of a closed disk from $\partial H_g^+$.
Given an element $\xi \in \calI$, the $3$-manifold
\[
M_{\xi} := H_g^+ \cup_{\iota \circ \xi} H_g^-
\]
is an oriented integral homology $3$-sphere.
Conversely, any oriented integral homology $3$-sphere arises in this way for some $g$, see \cite[\S 2]{Mo89}.
Let $\mathcal{H}$ be the set of diffeomorphism classes of oriented integral homology $3$-spheres, 
which forms a commutative monoid under the connected sum.

Let $e\colon \Sigma \times [-1,1] \to S^3$ be a tubular neighborhood of $\Sigma$.
Since the skein algebras $\skein(S^3)$ and $\tskein(S^3)$ are isomorphic to their respective ground rings, 
$e$ induces a $\Q [\rho][[h]]$-linear map 
$e_* \colon \ctskein (\Sigma) \to \Q [\rho][[h]]$
and a $\Q [[A+1]]$-linear map $e_*\colon \cskein (\Sigma)\to \Q [[A+1]]$.

\begin{thm}[{\cite[Theorem 1.1]{Tsujihom3}}]
There is an invariant of oriented  integral homology $3$-spheres $z_{\skein}\colon \mathcal{H} \to \Q [[A+1]]$
 which is defined by 
\[
z_{\skein}( M_{\xi}) :=
\sum_{i=0}^\infty \frac{1}{i! (-A+A^{-1})^i}
e_* \big((\zeta_{\skein} (\xi))^i\big) \ \in \Q[[A+1]]
\]
for any element $\xi \in \torelli$.
\end{thm}

\begin{thm}[{\cite[Theorem 9.1]{TsujiHOMFLY-PTskein}}]
There is an invariant of oriented  integral homology $3$-spheres $z_{\tskein}\colon \mathcal{H} \to \Q[\rho][[h]]$
 which is defined by 
\[
z_{\tskein}( M_{\xi}) :=
\sum_{i=0}^\infty \frac{1}{i! h^i}
e_* \big((\zeta_{\tskein} (\xi))^i\big) \ \in \Q[\rho][[h]]
\]
for any element $\xi \in \torelli$.
\end{thm}

We refer to the papers \cite{Tsujihom3} and \cite{TsujiHOMFLY-PTskein}
for further explanations about these invariants, and we only mention here some of their properties:

\begin{itemize}
\item[(i)] the maps $z_{\skein}$ and $z_{\tskein}$ are monoid homomorphisms;
\item[(ii)] for any $n>0$,
the truncations
$z_{\skein}\colon \mathcal{H} \to \Q [[A+1]]/((A+1)^{n+1})$ and
$z_{\tskein}\colon \mathcal{H} \to \Q [\rho ] [[h]] / (h^{n+1})$
are finite-type invariants of degree~$2n$  
using the definition of \cite{GGP,Habiro00};
\item[(iii)] if $\xi \in \zeta_{\skein}^{-1} (F^{2n+1} \cskein (\Sigma))$, we have
$z_{\skein} (M_{\xi}) \in 1+ ((A+1)^n)$;
\item[(iv)]  if $\xi \in \zeta_{\tskein}^{-1} (F^{2n+1} \ctskein (\Sigma))$, we have
$z_{\tskein} (M_{\xi}) \in 1+ (h^n)$.
\end{itemize}
Note that the assumptions of (iii) and (iv) are satisfied if $\xi$ is in the $(2n-1)$st term of the lower central series of $\calI$.

There is a ring isomorphism  $\Q[[q-1]] \cong \Q[[A+1]]$ 
obtained by substituting $q$ with $A^4$. 
With this identification, we can regard the invariant $z_{\skein}$ as taking values in $\Q [[q-1]]$.
We expect that $z_{\skein}\colon \mathcal{H}\to \Q [[q-1]]$ equals the Ohtsuki series \cite{Ohtsuki},
which is the ``perturbative'' quantum invariant of integral homology $3$-spheres derived from the quantum group $U_q(\mathfrak{sl}_2)$.
We also expect that, for any integer $N$, 
the ``perturbative'' $\mathfrak{sl}_N$-quantum invariant \cite{Le00} of $M\in \mathcal{H}$ 
can be recovered from $z_{\tskein}(M)$
 by the substitution $\rho \mapsto (N/2) \log q /(-q^{1/2} + q^{-1/2})$ and $h \mapsto - q^{1/2} + q^{-1/2}$.


\begin{thebibliography}{00}


\bibitem{Birman1971}
J.\,  Birman,
\emph{On Siegel's modular groups}.
Math. Ann. {\bf 191} (1971), 59--68.


\bibitem{BH95}
G.\, Brumfiel and H.\, Hilden,
\emph{$\operatorname{SL}(2)$-representations of finitely presented groups}. 
Contemporary Mathematics, 187, Amer$.$ Math$.$ Soc$.$, Providence, RI, 1995. 

\bibitem{Bullock97}
D.\ Bullock,
\emph{Rings of $\mathrm{SL}_2 (\mathbb{C})$ and 
the Kauffman bracket skein module}.
Comment. Math. Helv. {\bf 72} (1997), no.~4, 521--542. 


\bibitem{Dehn38}
M.\, Dehn,
\emph{Die Gruppe der Abbildungsklassen}.
Acta Math.\, \textbf{69} (1938), no.~1, 135--206.

\bibitem{GGP}
S. Garoufalidis, M. Goussarov and M. Polyak,
\emph{Calculus of clovers and finite type invariants of 3-manifolds.} 
Geom$.$ Topol$.$ \textbf{5} (2001), 75--108.

\bibitem{GL05}
S.\, Garoufalidis and J.\, Levine,
\emph{Tree-level invariants of three-manifolds, Massey products and the Johnson homomorphism}.
In: \emph{Graphs and patterns in mathematics and theoretical physics}, volume 73 of Proc.\ Sympos.\ Pure Math., Amer.\ Math.\ Soc., Providence 2005, 173--203.

\bibitem{Gol86}
W.\, Goldman,
\emph{Invariant functions on Lie groups and Hamiltonian flows of surface group representations}.
Invent.\, Math.\, \textbf{85} (1986), no.~2, 263--302.


\bibitem{Goussarov98}
M. Goussarov,
\emph{Interdependent modifications of links and invariants of finite degree.} 
Topology \textbf{37} (1998), no.~3, 595--602. 



\bibitem{Habegger00}
N.\, Habegger,
\emph{Milnor, Johnson, and tree level perturbative invariants}. Preprint (2000).

\bibitem{HP}
N.\, Habegger and W.\, Pitsch,
\emph{Tree level Lie algebra structures of perturbative invariants}.
J. Knot Theory Ramifications \textbf{12} (2003), no.~3, 333--345. 

\bibitem{Habiro00}
K.\, Habiro,
\emph{Claspers and finite type invariants of links}.
Geom. Topol. \textbf{4} (2000), 1--83.

\bibitem{HM}
K.\, Habiro and G.\, Massuyeau,
\emph{From mapping class groups to monoids of homology cobordisms: a survey}. 
Handbook of Teichm\"uller theory, Vol$.$ III, 465--529. IRMA Lect. Math. Theor. Phys. \textbf{17}, Eur. Math. Soc., Z\"urich, 2012. 

\bibitem{Johnson79}
D.\, Johnson,
\emph{Homeomorphisms of a surface which act trivially on homology}.
Proc. Amer. Math. Soc. {\bf 75} (1979), no.~1, 119--125. 

\bibitem{Joh83}
D.\, Johnson, 
\emph{A survey of the Torelli group}.
In: Low-dimensional topology (San Francisco, Calif., 1981),  Contemp. Math. 20, Amer. Math. Soc., Providence 1983, 165--179.

\bibitem{KKgroupoid}
N.\, Kawazumi and Y.\, Kuno,
\emph{Groupoid-theoretical methods in the mapping class groups of surfaces}.
Preprint, arXiv:1109.6479v3

\bibitem{KK14}
N.\, Kawazumi and Y.\, Kuno,
\emph{The logarithms of Dehn twists}.
Quantum Topol.\, \textbf{5} (2014), no.~3, 347--423.

\bibitem{KK15}
N.\, Kawazumi and Y.\, Kuno,
\emph{Intersection of curves on surfaces and their applications to mapping class groups}.
Ann.\, Inst.\, Fourier (Grenoble) \textbf{65} (2015), no.~6, 2711-2762.

\bibitem{KK16}
N.\, Kawazumi and Y.\, Kuno,
\emph{The Goldman-Turaev Lie bialgebra and the Johnson homomorphisms}.
Handbook of Teichm\"{u}ller theory, Vol.~V, 97--165,
IRMA Lect.\, Math.\, Theor.\, Phys.\, \textbf{26},
Eur.\, Math.\, Soc., Z\"{u}rich, 2016.

\bibitem{Ki78}
R.\, Kirby, 
\emph{A calculus for framed links in $S^3$}. 
Invent. Math. \textbf{45} (1978), no.~1, 35--56. 

\bibitem{KS09}
M.\, Korkmaz and A.\, Stipsicz,
\emph{Lefschetz fibrations on $4$-manifolds}.
Handbook of Teichm\"{u}ller theory, Vol.~II, 271--296, IRMA Lect.\, Math.\, Theor.\, Phys.\, \textbf{13},
Eur.\, Math.\, Soc., Z\"{u}rich, 2009.

\bibitem{Ku12}
Y.\, Kuno,
\emph{A combinatorial construction of symplectic expansions}.
Proc.\ Amer.\ Math.\ Soc. \textbf{140} (2012), no.~3, 1075--1083.


\bibitem{Kuno13}
Y.\, Kuno,
\emph{The generalized Dehn twist along a figure eight}.
J.\, Topol.\, Anal.\, \textbf{5} (2013), no.~3, 271--295.

\bibitem{KM19}
Y.\, Kuno and G.\, Massuyeau,
\emph{Generalized Dehn twists on surfaces and homology cylinders}.
Preprint, arXiv:1902.02592v1.

\bibitem{Le00}
T.\, Le,
\emph{On perturbative $PSU(n)$ invariants of rational homology 3-spheres}.
Topology {\bf 39} (2000), no.~4, 813--849. 

\bibitem{Lic62}
W.\, Lickorish,
\emph{A representation of orientable combinatorial $3$-manifolds}.
Ann.\, of Math.\, (2) \textbf{76} (1962), 531--540.

\bibitem{Ma12}
G.\, Massuyeau,
\emph{Infinitesimal Morita homomorphisms and the tree-level of the LMO invariant}.
Bull. Soc. Math. France \textbf{140} (2012), no.~1, 101--161.

\bibitem{MT13}
G.\, Massuyeau and V.\, Turaev,
\emph{Fox pairings and generalized Dehn twists}.
Ann.\, Inst.\, Fourier (Grenoble) \textbf{63} (2013), no.~6, 2403--2456.


\bibitem{MT14}
G.\, Massuyeau and V.\, Turaev,
\emph{Quasi-Poisson structures on representation spaces of surfaces}. 
Int.\, Math.\, Res.\, Not.\, \textbf{2014:1} (2014) 1--64.

\bibitem{MT17}
G.\, Massuyeau and V.\, Turaev,
\emph{Brackets in the Pontryagin algebras of manifolds}.
M\'em. Soc. Math. France {\bf 154} (2017).
 
\bibitem{MT18}
G.\, Massuyeau and V.\, Turaev,
\emph{Brackets in representation algebras of Hopf algebras}.
J. Noncomm. Geom. {\bf 12} (2018), no.~2, 577--636.

\bibitem{Mo89}
S.\, Morita,
\emph{Casson's invariant for homology 3-spheres and characteristic classes of surface bundles. I}.
Topology {\bf 28} (1989), no.~3, 305--323.

\bibitem{Mo93}
S.\,  Morita, 
\emph{Abelian  quotients  of  subgroups  of  the  mapping  class  group  of  surfaces}.
Duke Math. J. {\bf 70} (1993), 699--726.

\bibitem{Muller}
G.\, Muller,
\emph{Skein algebra and cluster algebras of marked surfaces}.
Quantum Topol. {\bf 7} (2016), no.~3, 435--503.

\bibitem{Ohtsuki}
T.\, Ohtsuki,
\emph{A polynomial invariant of integral homology $3$-spheres}.
Math. Proc. Cambridge Philos. Soc. {\bf 117} (1995), no.~1, 83--112. 

\bibitem{Papa75}
C.\,  Papakyriakopoulos,
\emph{Planar regular coverings of orientable closed surface}.
Knots, groups, and $3$-manifolds, 
261--292. Ann.\, of Math.\ Studies, No.~84, Princeton Univ.\, Press, Princeton, N.J., 1975.

\bibitem{Per06}
B.\, Perron,
\emph{A homotopic intersection theory on surfaces:  applications to mapping class group and braids}.
Enseign. Math. (2) \textbf{52} (2006), no.~1--2, 159--186.

\bibitem{Powell1978}
J.\, Powell,
\emph{Two theorems on the mapping class group of a surface}.
Proc. Amer. Math. Soc. {\bf 68} (1978), no.~3, 347--350.

\bibitem{Przy91}
J.\, Przytycki,
\emph{Skein modules of $3$-manifolds}.
Bull. Polish Acad. Sci. Math. {\bf 39} (1991), no.~1-2, 91--100.

\bibitem{PS00}
J.\, Przytycki and A.\, Sikora,
\emph{On skein algebras and $\operatorname{SL}_2(\mathbb{C})$-character varieties}. 
Topology {\bf 39} (2000), no.~1, 115--148.

\bibitem{Pu2008}
A.\, Putman,
\emph{An infinite presentation of the Torelli group}.
Geom. Funct. Anal. {\bf 19} (2009), no.~2, 591--643.

\bibitem{Qui69}
D.\, Quillen, 
\emph{Rational homotopy theory}. 
Ann. of Math. (2) {\bf 90} (1969), 205--295.


\bibitem{Sta65}
J.\, Stallings,
\emph{Homology and central series of groups}.
J.\ Algebra \textbf{2} (1965), 170--181.


\bibitem{TsujiCSAI}
S.\, Tsuji,
\textit{Dehn twists on Kauffman bracket skein algebras}.
Kodai Math. J. \textbf{41} (2018), 16--41.


\bibitem{TsujiCSApure}
S.\, Tsuji,
\emph{The quotient of a Kauffman bracket skein algebra by the square of an augmentation ideal}.
J. Knot Theory Ramifications \textbf{26} (2017), no.~5, 1750030, 34 pp.

\bibitem{TsujiTorelli}
S.\, Tsuji,
\textit{An action of the Torelli group on the Kauffman bracket skein module}.
Math. Proc. Camb. Philos. Soc. \textbf{165} (2017), no$.$1, 163--178.


\bibitem{Tsujihom3}
S.\, Tsuji,
\textit{Construction of an invariant for integral homology 3-spheres via completed Kauffman bracket skein algebras}.
Preprint, arXiv:1607.01580v4.


\bibitem{TsujiHOMFLY-PTskein}
S.\, Tsuji,
\textit{A formula for the action of Dehn twists on HOMFLY-PT
 skein modules and its applications}.
Preprint, arXiv:1801.00580v1.


\bibitem{Tur78}
V.\,  Turaev,
\emph{Intersections of loops in two-dimensional manifolds} (Russian).
Mat.\, Sb.\, \textbf{106(148)} (1978), no.~4, 566--588.

\bibitem{Tur81}
V.\,  Turaev,
\emph{Multiplace generalizations of the Seifert form of a classical knot} (Russian).
Mat. Sb. {\bf 116(158)} (1981), 370--397. 

\bibitem{Tur91}
V.\,  Turaev,
\emph{Skein quantization of Poisson algebras of loops on surfaces}.
 Ann. Sci. \'Ecole Norm. Sup. (4) \textbf{24} (1991), no.~6, 635--704.

\bibitem{VdB08}
M.\, Van den Bergh,
\emph{Double Poisson algebras}.
Trans. Amer. Math. Soc. {\bf 360} (2008), no.~11, 5711--5769.


\end{thebibliography}
\end{document}